\documentclass[12pt,twoside]{amsart}

\usepackage{pdfsync}

\usepackage{enumitem}
\usepackage{amssymb}
\usepackage[dvips]{graphicx}
\usepackage{palatino}
\usepackage[left=3cm,top=3.8cm,right=3cm]{geometry}        % for a good layout

\usepackage{xcolor}
\newcommand{\texorpdfstring}[2]{#1}
\usepackage[pagebackref]{hyperref}
\hypersetup{
    colorlinks,
    linkcolor={red!60!black},
    citecolor={blue!60!black},
   urlcolor={blue!90!black}
}

\numberwithin{equation}{section}

\newcommand{\N}{\mathbb{N}}         % natural numbers
\newcommand{\R}{\mathbb{R}}         % real line
\newcommand{\Z}{\mathbb{Z}}
\newcommand{\Q}{\mathbb{Q}}

\newcommand{\NN}{\mathcal{N}}
\newcommand{\DD}{\mathcal{D}}

\newcommand{\PP}{\mathbb{P}}        % probability
\newcommand{\supp}{\text{supp}}        % Support

\newcommand{\lam}{\lambda}
\newcommand{\e}{\varepsilon}

\newcommand{\wt}{\widetilde}

\newcommand{\1}{\mathbf{1}}

\newcommand{\be}{\begin{equation}}
\newcommand{\ee}{\end{equation}}

\DeclareMathOperator{\hdim}{dim_{\mathsf{H}}}

\DeclareMathOperator{\ubdim}{{\overline{dim}_{\mathsf{B}}}}

\newtheorem{thm}{Theorem}[section]
\newtheorem{lemma}[thm]{Lemma}
\newtheorem{prop}[thm]{Proposition}
\newtheorem{cor}[thm]{Corollary}

\theoremstyle{remark}

\newtheorem{defn}[thm]{Definition}
\newtheorem{conj}[thm]{Conjecture}

\keywords{$\times p$-invariant sets, dynamical rigidity, self-similar measures, Bernoulli convolutions, intersections of Cantor sets}

\subjclass[2010]{Primary: 11K55, 28A80, 37C45, Secondary: 28A78, 28D05, 37A45}

\begin{document}

\title[Furstenberg's intersection conjecture and self-similar measures]{On Furstenberg's intersection conjecture, self-similar measures, and the $L^q$ norms of convolutions}

\author{Pablo Shmerkin}
\address{Department of Mathematics and Statistics, Torcuato Di Tella University, and CONICET, Buenos Aires, Argentina}
\email{pshmerkin@utdt.edu}
\urladdr{http://www.utdt.edu/profesores/pshmerkin}
\thanks{P.S. was partially supported by Projects PICT 2013-1393 and PICT 2014-1480 (ANPCyT) }

\begin{abstract}
We study a class of measures on the real line with a kind of self-similar structure, which we call \emph{dynamically driven self-similar measures}, and contain proper self-similar measures such as Bernoulli convolutions as special cases. Our main result gives an expression for the $L^q$ dimensions of such dynamically driven self-similar measures, under certain conditions. As an application, we settle Furstenberg's long-standing conjecture on the dimension of the intersections of $\times p$ and $\times q$-invariant sets. Among several other applications, we also show that Bernoulli convolutions have an $L^q$ density for all finite $q$, outside of a zero-dimensional set of exceptions.

The proof of the main result is inspired by M. Hochman's approach to the dimensions of self-similar measures and his inverse theorem for entropy. Our method can be seen as an extension of Hochman's theory from entropy to $L^q$ norms, and likewise relies on an inverse theorem for the decay of $L^q$ norms of discrete measures under convolution. This central piece of our approach may be of independent interest, and is an application of well-known methods and results in additive combinatorics: the asymmetric version of the Balog-Szemer\'{e}di-Gowers Theorem due to Tao-Vu, and some constructions of Bourgain.
\end{abstract}

\maketitle

\tableofcontents

\section{Introduction and main results}

\subsection{Transversality of \texorpdfstring{$\times p$, $\times q$}{times p, times q}}

In the 1960s, H. Furstenberg proposed a series of conjectures which, in different ways, aim to capture the heuristic principle that ``expansions in multiplicatively independent bases (such as $2$ and $3$) should have no common structure''. Recall that $p,q\in \N$ are called multiplicatively independent if they are not powers of a common integer or, equivalently, $\log p/\log q$ is irrational. For $p\in\N_{\ge 2}$, let $T_p:[0,1)\to [0,1)$, $x\mapsto p x\bmod 1$ denote multiplication by $p$ on the circle. In \cite{Furstenberg67}, Furstenberg proved a pioneering result of this type: if $p,q\in\N_{\ge 2}$ are multiplicatively independent, then no infinite proper closed subset of $[0,1]$ can be simultaneously invariant under $T_p$ and $T_q$. This gave rise to the famous $\times 2,\times 3$ conjecture, which remains open today: if $\mu$ is a Borel probability measure on the circle invariant under $T_2$ and $T_3$, then $\mu$ is a linear combination of Lebesgue measure and a purely atomic measure.

Furstenberg proposed other conjectures with a more geometric flavor. Let $A,B$ be closed subsets of the circle $[0,1)$ invariant under $T_p, T_q$ respectively, with $p,q$ again multiplicatively independent. Furstenberg conjectured that
\[
\hdim(A+B) = \min(\hdim(A)+\hdim(B),1),
\]
where $\hdim$ stands for Hausdorff dimension, and $A+B=\{a+b:a\in A, b\in B\}$ is the arithmetic sum. This fits into the general heuristic principle mentioned above, since the inequality $\hdim(A+B) \le \min(\hdim(A)+\hdim(B),1)$ always holds, and a strict inequality should only occur if $A$ and $B$ have some shared structure at many scales. This conjecture was proved in \cite{PeresShmerkin09} in the special case that $A,B$ are defined  by restricting the digits in their base $p,q$ expansion to a fixed digit set, and in \cite{HochmanShmerkin12} in the general case. Moreover, in \cite{HochmanShmerkin12} a corresponding result for invariant measures was obtained: if $\mu,\nu$ are Borel probability measures invariant under $\times p,\times q$ respectively, then
\begin{equation} \label{eq:dim-conv-invariant-measures}
\hdim(\mu*\nu) =  \min(\hdim(\mu)+\hdim(\nu),1).
\end{equation}
Here $\hdim$ denotes the lower Hausdorff dimension of a measure, defined as
\begin{equation*} %\label{eq:measure-dissonance}
\hdim(\eta) = \inf\{\hdim(A):\eta(A)>0\}.
\end{equation*}

We note that this result is trivial if either $\mu$ or $\nu$ have zero entropy (since zero entropy implies zero dimension), but in the positive entropy case it is stronger than the $\times 2,\times 3$ conjecture. We recall the Rudolph-Johnson theorem, asserting that if $\mu,\nu$ are ergodic and invariant under $\times p,\times q$ (with $\log p/\log q$ irrational) and $\mu$ has positive but not full entropy with respect to $\times p$, then $\mu$ and $\nu$ are singular. We showed in \cite{HochmanShmerkin12} that the Rudolph-Johnson Theorem can be obtained as an easy corollary of \eqref{eq:dim-conv-invariant-measures}.

There is an obvious heuristic relationship between the size of the sumset $A+B$ and the size of the fibers $\ell_z=\{ (x,y):x\in A,y\in B, x+y=z\}$. Namely, if the sumset is ``large'' then ``many fibers'' should be small, and conversely. Another conjecture of Furstenberg, and one of the few to be stated explicitly in print \cite[Conjecture 1]{Furstenberg70}, asserts that for sets invariant under $\times 2,\times 3$, \emph{all} fibers should be small:
\begin{conj} \label{conj:Furstenberg}
If $A,B$ are closed subsets of the circle $[0,1)$, invariant under $T_p, T_q$ respectively, with $p$ and $q$ multiplicatively independent, then
\begin{equation} \label{eq:upper-bound-dim-fibers}
\hdim(A\cap B) \le \max(\hdim(A)+\hdim(B)-1,0).
\end{equation}
\end{conj}
In Furstenberg's terminology, the dynamics of $T_p$ and $T_q$ should be \emph{transverse}. Again, this fits into the general heuristics of ``lack of common structure'' since a fiber of larger than expected size can be seen as some shared structure between $A$ and $B$ (and hence between expansions in bases $p$ and $q$). To see why the right-hand side in \eqref{eq:upper-bound-dim-fibers} is the natural bound, one can think of the analogous formula for the dimension of the intersection of transversal linear subspaces, or  Marstrand's intersection theorem asserting that for any Borel set $E\subset \R^2$,
\[
\hdim(E\cap \ell)\le \max(\hdim(E)-1,0)
\]
for \emph{almost all} lines $\ell$, and this fails for any smaller value on the right-hand side (we note that $E=A\times B$ has dimension $\hdim(A)+\hdim(B)$). See for example \cite[Chapter 10]{Mattila95}.

Also in \cite{Furstenberg70}, Furstenberg showed that if
\[
\hdim(A\cap g(B)) \ge c
\]
for some invertible affine map $g:\R\to\R$, then for almost all slopes $a$ there is an affine map $g_a(x)=a x+ b(a)$ such that
\[
\hdim(A\cap g_a(B)) \ge c.
\]
Using this, it is not hard to show that Conjecture \ref{conj:Furstenberg} holds when $\hdim(A)+\hdim(B)\le 1/2$; see \cite[Theorem 7.9]{Hochman14b} for an exposition of the argument. More generally, combining Furstenberg's result with estimates of Wolff \cite{Wolff99} on the dimension of sets that contain, for almost every $v\in S^1$, a subset of a line in direction $v$ with Hausdorff dimension at least $c$, one gets
\[
\hdim(A\cap g(B)) \le \max(\hdim(A)+\hdim(B)-1/2,0).
\]
Note that this is vacuous if $\hdim(A)\ge 1/2$.

We say that $A\subset [0,1)$ is a $p$-Cantor set if it is the set of points whose base $p$-expansion digits lie in some proper set $D\subset \{0,1,\ldots,p-1\}$ with at least $2$ elements. In a different direction, in \cite{FHR14} it was shown that if $A$ and $B$ are a $p$-Cantor set and a $q$-Cantor set respectively, then $A$ cannot be affinely embedded into $B$ if $0<\hdim(A)<\hdim(B)<1$. More precisely, it follows from \cite[Theorem 1.6]{FHR14} that in this case there is some (non-effective) $\delta=\delta(A,B)>0$ such that
\[
\hdim(A\cap g(B)) \le \hdim(A)-\delta
\]
for all $C^1$ diffeomorphisms $g$ of $\R$ (here, and whenever clear from context, we think of $A,B$ as subsets of $[0,1)\subset \R$ rather than the circle). One can deduce the same result for general invariant sets by a standard upper approximation. De-Jun Feng (Private Communication) developed an algorithm that yields effective values of $\delta$ in specific cases, for example if $A$ is the middle-one quarter Cantor set and $B$ is the middle-thirds Cantor set; the computed values are still far from those predicted by Furstenberg's conjecture.

D-J. Feng (Private Communication) also constructed, for any multiplicatively independent $p,q$ and for any $0<s,t<1$ and $\e>0$,  closed $T_p,T_q$-invariant sets $A,B\subset [0,1)$ of dimension $s,t$ respectively, for which
\[
\hdim(A\cap g(B)) \le \max(\hdim(A)+\hdim(B)-1,0)+\e,
\]
for all affine maps $g$. Although this comes close, we note that not a single example of sets $A,B$ (for some multiplicatively independent $p,q$)  for which the conjecture holds, was known, apart from the trivial cases in which one of the sets has dimension $0$ or $1$, and the case in which $\hdim(A)+\hdim(B)\le 1/2$, as explained above.

In this article, we prove the following strong version of Furstenberg's conjecture which, in his terminology, says that the maps $T_p$ and $T_q$ on the circle are \emph{strongly transverse}:
\begin{thm} \label{thm:Furstenberg}
Let $p,q\in\N_{\ge 2}$ be multiplicatively independent. Then for any closed sets $A,B$ of the circle $[0,1)$ invariant under $T_p, T_q$ respectively, and for any invertible affine map $g:\R\to \R$,
\[
\ubdim(A\cap g(B)) \le \max(\hdim(A)+\hdim(B)-1,0).
\]
\end{thm}
Here $\ubdim$ denotes upper box-counting dimension, which is always at least as large as Hausdorff dimension.

The method we use to establish Theorem \ref{thm:Furstenberg} yields several other new results on classical problems in fractal geometry and dynamics. Before discussing our general approach, we present some of these results.

\subsection{Dimension and densities of Bernoulli convolutions}

Given $\lam\in(0,1)$, let $\nu_\lam$ be the distribution of the random series $\sum_{n=0}^\infty \pm\lam^n$, with the signs chosen independently with equal probabilities. This is the family of \emph{Bernoulli convolutions}, whose study goes back to the 1930s. For $\lam\in (0,1/2)$, it is well known that $\nu_\lam$ is (up to an affine bijection) a constant multiple of Hausdorff measure (of the appropriate dimension) on the central Cantor set constructed by removing a central interval of length $1-2\lam$ from $[0,1]$ and iterating. The properties of $\nu_\lam$ for $\lam\in [1/2,1)$ have been studied for some 80 years but are far from being properly understood. We prove new properties of the densities and dimension of $\nu_\lam$ outside of a small set of parameters.

Perhaps the most significant open problem on Bernoulli convolutions is to determine for which values of $\lam$ the measure $\nu_\lam$ turns out to be absolutely continuous. Erd\H{o}s already in 1939 \cite{Erdos39} showed that if $\lam^{-1}$ is a Pisot number (an algebraic unit $>1$ such that all its algebraic conjugates are $<1$ in modulus), then $\nu_\lam$ is singular. It is still not known if there is any $\lam\in (1/2,1)$ such that $\nu_\lam$ is singular and $\lam^{-1}$ is not Pisot.

In light of this open problem, a fruitful strand of research developed to prove results of the form: $\nu_\lam$ is absolutely continuous, with certain regularity of the density, outside of some ``small'' set. This line was also initiated by Erd\H{o}s \cite{Erdos40}, who proved that for every $k\in\N$ there is $\e_k>0$ such that $\nu_\lam$ has a $k$-times continuously differentiable density for almost all $\lam \in (1-\e_k,1)$. Several decades later, Kahane \cite{Kahane71} noted that Erd\H{o}s' argument yields a stronger statement, namely, that for every $k\in\N$,
\[
\lim_{\e\downarrow 0} \hdim\{\lam\in (1-\e,1): \nu_\lam \text{ does not have a $C^k$ density }\} = 0.
\]
The proof of Erd\H{o}s-Kahane is based on a combinatorial study of the Fourier transform of $\nu_\lam$, and no other proof of the statement is known.

The Erd\H{o}s-Kahane argument only gives non-trivial information very close to $1$. In a landmark paper from 1995, Solomyak \cite{Solomyak95} showed that $\nu_\lam$ is absolutely continuous with an $L^2$ density for almost all $\lam\in (1/2,1)$. A simpler proof was obtained by Peres and Solomyak \cite{PeresSolomyak96}. The $L^2$ part of the result is a by-product of the transversality technique used by Solomyak, and a natural question is whether $L^2$ can be replaced by a better space. In \cite{PeresSchlag00}, Peres and Schlag proved that for any $\e>0$ there is some (explicit) $\delta>0$ such hat $\nu_\lam$ has fractional derivatives of order $\delta$ in $L^2$, for almost all $\lam\in (1/2+\e,1)$. By the Sobolev embedding theorem, in particular this implies that $\nu_\lam$ has a density in $L^q$ for some $q=q(\e)>2$ for almost all $\lam\in (1/2+\e,1)$. Their result still relies on transversality techniques, which cannot go beyond $L^2$ for $\lam$ close to $1/2$.

Besides improving on the smoothness of the density, another natural line to pursue is to make the exceptional set of $\lam$ smaller. In the same article \cite{PeresSchlag00}, Peres and Schlag proved that for every $\e>0$,  there is an explicit $\delta>0$ such that
\[
\hdim\{\lam\in (1/2+\e, 1):\nu_\lam \text{ does not have an $L^2$ density } \} \le 1-\delta.
\]

Much more recently, the author \cite{Shmerkin14} (relying on deep work of Hochman \cite{Hochman14} that will be discussed in some detail below) proved that $\nu_\lam$ is absolutely continuous for all $\lam$ outside of a set of zero Hausdorff dimension. Moreover, in \cite{ShmerkinSolomyak16} it was shown that, again outside of a set of zero Hausdorff dimension of parameters, $\nu_\lam$ has a density in $L^q$ for some $q>1$ that is not explicit and depends on $\lam$.

These three lines of work yield somewhat complementary results: the stronger the information about the densities, the weaker the information about the exceptional set. They also leave open the question of what is the smallest natural function space that contains the density of $\nu_\lam$ for almost all $\lam$. In this article, we  prove:
\begin{thm} \label{thm:abc-cont-BCs}
\begin{enumerate}[label={\upshape(\roman*).}]
\item There exists a set $\mathcal{E}\subset (1/2,1)$ of zero Hausdorff dimension such that if $\lam\in (1/2,1)\setminus \mathcal{E}$, then $\nu_\lam$ has a density in $L^q$ for all finite $q>1$.
\item There exists a set $\mathcal{E}'\subset (1/\sqrt{2},1)$ of zero Hausdorff dimension such that if $\lam\in (1/\sqrt{2},1)\setminus \mathcal{E}'$, then $\nu_\lam$ has a continuous density.
\end{enumerate}
\end{thm}

The new contribution is part (i); part (ii) then follows by a standard argument. In turn, part (i) follows from a new result about \emph{dimensions} of Bernoulli convolutions, together with a result from \cite{ShmerkinSolomyak16}. To state the dimensional result, we define the following set (which appears already in \cite{Hochman14}).

\begin{defn} \label{def:ssc-BCs}
Let $\mathcal{P}_n$ be the family of all non-zero polynomials of degree at most $n$ and coefficients in $\{-1,0,1\}$. Let
\begin{equation*} %\label{eq:def-super-exp-set-BCs}
\mathcal{E} = \left\{ \lam\in (1/2,1):  \frac1n \log\left(\min_{P\in\mathcal{P}_n} |P(\lam)|\right) \to -\infty \right\} .
\end{equation*}
\end{defn}
It is shown in \cite{Hochman14} that $\mathcal{E}$ has zero packing dimension (in particular, zero Hausdorff dimension) and does not contain any algebraic number which is not a root of a polynomial in $\mathcal{P}_n$ for some $n$. In particular, no rational number in $(1/2,1)$ is in $\mathcal{E}$.

\begin{thm} \label{thm:infinity-dim-BCs}
Let $\lam\in (1/2,1)\setminus \mathcal{E}$. Then for every $\e>0$ there is $C=C(\e,\lam)>0$ such that
\[
\nu_\lam(B(x,r)) \le C \, r^{1-\e}  \quad\text{for all } x\in \R, r\in (0,1].
\]
\end{thm}
It is known (see \cite{FengHu09}) that for any $\lam$, the limit
\[
\lim_{r\downarrow 0} \frac{\log\nu_\lam(B(x,r))}{\log r}
\]
exists and is constant $\nu_\lam$-almost everywhere; this constant value is denoted $\dim(\nu_\lam)$ and equals the Hausdorff, packing and entropy dimensions of $\nu_\lam$. In \cite{Hochman14}, it is proved that if $\lam\in (1/2,1)\setminus\mathcal{E}$, then $\dim(\nu_\lam)=1$. Theorem \ref{thm:infinity-dim-BCs} strengthens this, since it implies in particular that
\[
\liminf_{r\downarrow 0} \frac{\log\nu_\lam(B(x,r))}{\log r}  \ge 1
\]
for \emph{all} (rather than almost all) $x$. On the other hand, for any locally finite measure $\mu$ on the real line it holds that
\[
\limsup_{r\downarrow 0} \frac{\log\mu(B(x,r))}{\log r}  \le 1
\]
for $\mu$ almost all $x$. Nevertheless, for any $\lam \in (1/2,1)$ there are two points $x$ (the boundary points of the support of $\nu_\lam$) for which
\[
\lim_{r\downarrow 0}  \frac{\log\nu_\lam(B(x,r))}{\log r}  = \frac{\log 2}{\log(1/\lam)}>1,
\]
and if $\lam$ is close to $1/2$ there is a positive dimensional set of such points, see \cite[Theorem 1.5]{JSS11}. These remarks indicate that Theorem \ref{thm:infinity-dim-BCs} is optimal in a number of ways.

We obtain similar results for more general self-similar measures, including biased Bernoulli convolutions. We compute the $L^q$ dimension of arbitrary self-similar measures on the real line under Hochman's exponential separation assumption: see Theorems \ref{thm:dim-ssm} and \ref{thm:dim-general-ssm}. We also establish absolute continuity with $L^q$ density for general parametrized families of homogeneous self-similar measures, outside of a codimension $1$ set of possible exceptions in the super-critical region. See Theorem \ref{thm:abs-cont-parametrized-ssm} for details.

Very recently, some striking progress on the dimensions and absolute continuity of Bernoulli convolutions for algebraic parameters was achieved by P. Varj\'{u} \cite{Varju16} and E. Breuillard and P. Varj\'{u} \cite{BreuillardVarju15}. The latter article also uncovers some deep connections between Bernoulli convolutions, the famous Lehmer's conjecture from number theory, and the growth of subgroups of linear groups. This line of work goes in a transversal direction to ours: while they obtain new information for many algebraic (and not only) parameters, which our work is far from being able to replicate, their methods do not seem to be able to give information about Frostman exponents or $L^q$ densities for any $q>1$.

\subsection{\texorpdfstring{$L^q$}{Lq} dimensions, Frostman exponents,  and the size of fibers}

At first sight, Theorems \ref{thm:Furstenberg} and \ref{thm:infinity-dim-BCs} may appear to have little in common. However, we will obtain both as rather direct consequences of a single general result. Our common approach is based on \emph{$L^q$ dimensions}. Let $\mu$ be a Borel probability measure on $[0,1]$. We denote the family of $2^{-m}$-intervals $\{ [j 2^{-m}, (j+1) 2^{-m})\}$, $j\in\Z$ by $\DD_m$. If $q>1$, then
\[
\frac{\log \sum_{I\in\DD_m} \mu(I)^q}{(1-q)m} \in [0,1],
\]
for any $m\in\N$, as can be easily seen from H\"{o}lder's inequality. Here and throughout the article, the logarithms are to base $2$. Moreover, a small value indicates that $\mu$ is nearly concentrated on few intervals in $\DD_m$, while a value close to $1$ implies that $\mu(I), I\in\DD_m$ is a fairly uniform probability vector. Thus, it makes sense to consider the limit as $m\to\infty$ of the left-hand side as a notion of dimension of $\mu$.

\begin{defn}
Let $q\in (1,\infty)$. If $\mu$ is a probability measure on $\R$ with bounded support, then
\[
\tau(\mu,q) = \tau_\mu(q) = \liminf_{m\to\infty} -\frac{\log \sum_{I\in\DD_m} \mu(I)^q}{m}
\]
is the \emph{$L^q$ spectrum} of $\mu$, and
\[
D(\mu,q) =  D_\mu(q) = \frac{\tau_\mu(q)}{q-1}
\]
is the \emph{$L^q$ dimension} of $\mu$.
\end{defn}
It is also possible to define $L^q$ dimensions for other values of $q$, but we will not need to do so here. It is well-known that, for a fixed measure $\mu$, the map $q\mapsto D(\mu,q)$ is continuous and decreasing on $(1,\infty)$. Moreover,
\begin{equation*} %\label{eq:hausdorff-dim-larger-than-Lq-dim}
\hdim\mu \ge \lim_{q\downarrow 1} D(\mu,q).
\end{equation*}
See \cite{FLR02} for proofs of these standard facts.

If $\mu$ is a finite measure on a metric space $X$, we say that $\mu$ has \emph{Frostman exponent} $s$ if $\mu(B(x,r)) \le C\,r^s$ for some $C>0$ and all $x\in X,r>0$. It is easy to see that $L^q$ dimensions for large $q$ provide information about Frostman exponents:
\begin{lemma} \label{lem:Lq-dim-to-Frostman-exp}
Let $\mu$ be a probability measure on a compact interval of $\R$. If $D(\mu,q)> s$ for some $q\in (1,\infty)$, then there is $r_0>0$ such that
\[
\mu(B(x,r)) \le \, r^{(1-1/q)s} \text{ for all } x\in\R, r\in (0,r_0].
\]
\end{lemma}
\begin{proof}
If $D(\mu,q)>s$, then there is $s'>s$ such that for all large enough $m$ and each $J\in\DD_m$,
\[
\mu(J)^q \le \sum_{I\in\DD_m} \mu(I)^q \le 2^{-m (q-1) s'}.
\]
Since any ball can be covered by $O(1)$ dyadic intervals of size smaller than the radius, we get that if $r$ is sufficiently small then
\[
\mu(B(x,r)) \le C\, r^{(1-1/q)s'},
\]
where $C$ is independent of $x$ and $r$. This gives the claim.
\end{proof}

Hence, in order to establish Theorem \ref{thm:infinity-dim-BCs} it is enough to show that, under the hypotheses of the theorem, $D(\nu_\lam,q)=1$ for arbitrarily large $q$; and this is what we will do.

Next, we show how Frostman exponents (and therefore, also $L^q$ dimensions) of projected measures give information about the size of fibers. We recall the definition of upper box-counting (or Minkowski) dimension in a totally bounded metric space $(X,d)$. Given $A\subset X$, let $N_\e(A)$ denote the maximal cardinality of an $\e$-separated subset of $A$. The upper box-counting dimension of $A$ is then defined as
\[
\ubdim(A) = \limsup_{\e\downarrow 0} \frac{\log(N_\e(A))}{\log(1/\e)}.
\]
\begin{lemma} \label{lem:Frostman-exp-to-small-fiber}
Let $X$ be a compact metric space, and suppose $\pi:X\to \R$ is a Lipschitz map. Let $\mu$ be a probability measure on $X$ such that $\mu(B(x,r)) \ge r^s$ for all $x\in X$ and all sufficiently small $r$ (independent of $x$). If $\pi\mu$ has Frostman exponent $\alpha$, then there exists $C>0$ such that for all balls $B_\e$ of radius $\e$ in $\R$, any $\e$-separated subset of $\pi^{-1}(B_\e)$ has size at most $C\e^{-(s-\alpha)}$.

In particular, for any $y\in\R$,
\[
\ubdim(\pi^{-1}(y)) \le s-\alpha
\]
\end{lemma}
\begin{proof}
Let $(x_j)_{j=1}^M$ be an $\e$-separated subset of $\pi^{-1}(B_\e)$ with $\e$ small. Then
\[
\mu\left(\bigcup_{j=1}^M B(x_j,\e/2)\right) \ge M (\e/2)^s,
\]
while the set in question projects onto an interval of size at most $O(\e)$. Hence $M=O(\e^{\alpha-s})$, giving the claim.
\end{proof}

\subsection{A class of dynamically-driven self-similar measures}
\label{subsec:DDSSM}

It is easy to see that in order to prove Theorem \ref{thm:Furstenberg}, it is enough to consider the case in which $A$ is a $p$-Cantor set and $B$ is a $q$-Cantor set, that is, $A$ is the set of points whose base $p$-expansion digits lie in some set $D_1\subset\{0,1,\ldots,p-1\}$, and likewise for $B$ and a set $D_2\subset \{0,1,\ldots,q-1\}$. Let $\Delta_i = \frac1{|D_i|} \sum_{d\in D_i} \delta_d$, and let $\eta_1$, $\eta_2$ be the distributions of the random sums $\sum_{i=1}^\infty X_i p^{-i}$, $\sum_{i=1}^\infty Y_i q^{-i}$, respectively, where $X_i$ are i.i.d. random variables with distribution $\Delta_1$, and $Y_i$ are i.i.d. random variables, also independent of the $X_i$, with distribution $\Delta_2$. Finally, set $\mu=\eta_1\times\eta_2$.

It is easy to see that $\mu(B(x,r)) = \Theta(r^{\hdim A+\hdim B})$ for $x\in\supp(\mu)=A\times B$. Our goal is to apply Lemma \ref{lem:Frostman-exp-to-small-fiber} to $\mu$ and, in light of Lemma \ref{lem:Lq-dim-to-Frostman-exp}, we will do this by investigating the $L^q$ dimension of projections of $\mu$.  Up to a smooth change of coordinates in the parametrization, and an affine change of coordinates in the projections, the family of linear projections of $\mu$ in directions with strictly positive slope is given by
\[
\{ \mu_x := \eta_1 * S_{e^x} \eta_2: x\in \R \} ,
\]
where $S_a(x)=ax$ scales by $a$. Note that $\mu_x$ is an infinite convolution of Bernoulli random variables, since $\eta_1,\eta_2$ are. Unlike $\eta_1,\eta_2$, the measures $\mu_x$ are not self-similar because $\eta_1,\eta_2$ are constructed with different contraction ratios. However, it is still possible to express $\mu_x$ in a way that resembles self-similarity, but with the geometry at different scales driven by a dynamical system. Namely, suppose $p<q$ and let $X=[0,\log q)$, $\mathbf{T}:X\to X$, $x\mapsto x+\log p\bmod(\log q)$. Moreover,  for each $x\in X$, let $\Delta(x)$ be the finitely supported measure given by
\[
\Delta(x) = \left\{
\begin{array}{ll}
  \Delta_1* S_{e^x}\Delta_2 & \text{if } x\in [0,\log p) \\
  \Delta_1 & \text{if } x\in [\log p,\log q)
\end{array}
\right..
\]
It is then easy to see that $\mu_x$ is the distribution of the random sum $\sum_{i=1}^\infty Z_i p^{-i}$, where the $Z_i$ are independent and have distribution $\Delta(\mathbf{T}^i x)$. Indeed, let
\[
n'(x)=|\{j\in\{1,\ldots,n\}: \mathbf{T}^j(x)\in [0,\log p)\}|.
\]
Note that
\[
\mathbf{T}^n(x)= x+  n \log p - n'(x) \log q,
\]
so that
\[
e^{\mathbf{T}^n(x)}p^{-n} = e^x q^{-n'(x)}.
\]
Hence the distribution $\mu_{x,n}$ of $\sum_{i=1}^n Z_i p^{-i}$ is equal to the distribution of
\[
\sum_{i=1}^n X_i p^{-i}+  \sum_{i=1}^{n'(x)} e^x Y_i q^{-i},
\]
where $X_i, Y_i$ are independent and have distribution $\Delta_1,\Delta_2$ respectively. This shows that $\mu_{x,n}\to\mu_x$ weakly.

Although in different language, this decomposition of $\mu_x$ can be traced back to Furstenberg  \cite{Furstenberg70}, and was also used more explicitly in \cite{NPS12} to study the $L^2$ dimensions of $\mu_x$.

Based on the above discussion, we introduce the following setup. Let $\mathcal{A}$ be the collection of all probability measures supported on a finite set, i.e.
\[
\mathcal{A} = \left\{ \sum_{i=1}^{N} p_i \delta(t_i) : N\in\N, p_i> 0, \sum_i p_i =1 , t_i\in\R\right\}.
\]
(We denote a delta mass at $t$ either by $\delta_t$ or $\delta(t)$.) We topologize $\mathcal{A}$ in the natural way: it consists of countably many connected components, corresponding to the number of atoms $N$, and for each $N$ it inherits the topology from $\R^{2N}$.

If $\mu$ is a measure on a metric space $X$ and $f:X\to Y$ is a Borel map, then we denote by $f\mu$ the push-forward measure: $f\mu(A)=\mu(f^{-1}A))$. Fix $\lam\in (0,1)$. If $\Delta_i$ is a sequence of measures in $\mathcal{A}$, all supported on a fixed compact interval, then we can form the infinite Bernoulli convolution
\[
\mu = *_{i=0}^\infty S_{\lam^i} \Delta_i.
\]
(Equivalently, $\mu$ is the distribution of the random sum $\sum_{i=0}^\infty \lam^i Z_i$, where the $Z_i$ are independent and have distribution $\Delta_i$.) We are interested in the situation in which the $\Delta_i$ are generated dynamically. Let $(X,\mathbf{T})$ be a dynamical system, and suppose $\Delta:X\to \mathcal{A}$ is a map such that, for some compact interval $I_0$, $\supp(\Delta(x))\subset I_0$ for all $x\in X$. Then we can consider the family of measures
\begin{equation} \label{eq:def-mu-x}
\mu_x = *_{i=0}^\infty S_{\lam^i} \Delta(\mathbf{T}^i x), \quad x\in X.
\end{equation}
These measures enjoy a dynamical version of self-similarity. Write
\begin{equation} \label{eq:def-mu-x-n}
\mu_{x,n} = *_{i=0}^{n-1} S_{\lam^i} \Delta(\mathbf{T}^i x).
\end{equation}
Then, clearly,
\begin{equation} \label{eq:mu-x-self-similarity}
\mu_x = \mu_{x,n} * S_{\lam^n} \mu_{\mathbf{T}^n x}.
\end{equation}

We will call the tuple $\mathcal{X}=(X,\mathbf{T},\Delta,\lam)$ a \emph{model} generating the measures $\mu_x$. We will also refer to the measures $\mu_x$ themselves as \emph{dynamically driven self-similar measures}.

Trivially, Bernoulli convolutions also fall into this setting, with $X$ the one-point space.

\subsection{\texorpdfstring{$L^q$}{Lq} dimensions of dynamically driven self-similar measures}

In order to prove Theorem \ref{thm:Furstenberg} along the lines we have been describing, we need to derive estimates on the $L^q$ dimensions of $\eta_1 * S_{e^x} \eta_2$ for \emph{all} values of $x$. As a matter of fact, by self-similarity, it is enough to deal with all $x$ in some nonempty open set, but it is not enough to gain information for almost all values of $x$. Note that the underlying dynamical system $(X,\mathbf{T})$ is an irrational rotation on the circle (thanks to $p$ and $q$ being multiplicatively independent) while, in the case of Bernoulli convolutions, $(X,\mathbf{T})$ is the trivial one-point system. In the general case of dynamically driven self-similar measures generated by a model $(X,\mathbf{T},\Delta,\lam)$, if one hopes to gain any information for all $x\in X$, it is reasonable to impose strong rigidity and continuity assumptions on the dynamics. The next definition, clearly satisfied by our two main classes of examples, introduces the kind of regularity that will be needed in the abstract setting. Recall that a Borel transformation $T:X\to X$ is called \emph{uniquely ergodic} if there exists exactly one Borel probability measure $\PP$ on $X$ such that $T\PP=\PP$.

\begin{defn} \label{def:pleasant}
We say that a model $(X,\mathbf{T},\Delta,\lam)$ is \emph{pleasant} if $X$ is a compact metric space, $\mathbf{T}$ is a uniquely ergodic transformation on $X$, the measures $\mu_x$ are all non-atomic and supported on some fixed bounded interval, and the map $x\mapsto \mu_{x}$ is continuous (in the weak topology), outside of a null set (with respect to the unique invariant measure).
\end{defn}
In most of our applications, $X$ will equal either the trivial group $\{0\}$ or the circle, and in all applications $X$ will be a torus or the product of a torus and a cyclic group. In all cases, $\mathbf{T}$ will be a translation on $X$. We recall that if $X$ is a compact Abelian group, and $\mathbf{T}(x)=x+y$ is translation by $y\in X$, then $\mathbf{T}$ is uniquely ergodic if and only if the orbit $\{ n y:n\in\N\}$ is dense in $X$. See e.g. \cite[Theorem 4.14]{EinsiedlerWard11}.

We will also need to impose a separation condition, albeit an extremely weak one.
\begin{defn} %\label{def:exponential-separation}
Let $\mathcal{X}=(X,\mathbf{T},\Delta,\lam)$ be a pleasant model with unique invariant measure $\mathbb{P}$. We say that $\mathcal{X}$ has \emph{exponential separation} if for $\mathbb{P}$-almost all $x$ there is $R>0$ such that the following holds for infinitely many $n$: all the atoms of $\mu_{x,n}$ are distinct and $\lam^{R n}$-separated. By the atoms of $\mu_{x,n}$ being distinct we mean that
\[
|\supp(\mu_{x,n})| = \prod_{i=0}^{n-1} |\supp(\Delta(\mathbf{T}^i x))|,
\]
i.e. there are no exact coincidences among the atoms that make up $\supp(\mu_{x,n})$.
\end{defn}
This definition coincides with the notion of (lack of) super-exponential separation introduced in \cite{Hochman14} in the case of self-similar measures (i.e. when $X$ is a one-point set). As will become clear later, if $X$ is infinite, then under very mild non-degeneracy assumptions on the map $\Delta$, exponential separation holds almost automatically.

The following is the main result of the paper, from which more general versions of Theorems \ref{thm:Furstenberg},  \ref{thm:abc-cont-BCs} and \ref{thm:infinity-dim-BCs}, as well as other applications, will follow.

\begin{thm} \label{thm:L-q-dim-dyn-ssm}
Let $(X,\mathbf{T},\Delta,\lam)$ be a pleasant model with exponential separation, and denote the unique invariant measure by $\mathbb{P}$. Assume further that the map $x\mapsto \Delta(x)$ is continuous $\PP$-almost everywhere, and the number of atoms of $\Delta(x)$ is uniformly bounded. Then for all $q\in (1,+\infty)$
\begin{equation} \label{eq:main-thm}
\lim_{m\to\infty} -\frac{\log\sum_{I\in\DD_m} \mu_x(I)^q}{(q-1)m}  = \min\left( \frac{\int_X \log\|\Delta(x)\|_q^q \,d\mathbb{P}(x)}{(q-1)\log\lam},1\right),
\end{equation}
uniformly in $x\in X$. That is, the limit in the definition of $L^q$ dimension of $\mu_x$ exists and equals the constant value on the right-hand side, for all $x\in X$.
\end{thm}

In the above statement, and throughout the paper, the $L^q$ norm of a finitely supported measure $\Delta$ is given by
\[
\|\Delta\|_q^q = \sum_{y\in\supp(\Delta)} \Delta(y)^q.
\]

We underline that the exponential separation assumption has to be checked on a set of full $\mathbb{P}$-measure, and this is often very easy to do. On the other hand, the conclusion of Theorem \ref{thm:L-q-dim-dyn-ssm} holds for \emph{all} $x\in X$.

\subsection{Outline of proof}

We conclude this introduction by presenting an outline of the main steps of the proof of Theorem \ref{thm:L-q-dim-dyn-ssm}. The overall strategy is inspired by the ideas of \cite{Hochman14}. Additional complications are caused by the fact that our model allows measures which are not strictly self-similar; this will be dealt with the help of a cocycle introduced in \cite{NPS12}. The key difference, however, is that Hochman's method is based on entropy, while we need to deal with $L^q$ norms. As we will see, this forces substantial changes in the implementation of the general strategy.

At the heart of \cite{Hochman14} is an inverse theorem for the growth of entropy under convolutions, see \cite[Theorem 2.7]{Hochman14}. We prove an inverse theorem for the decay of $L^q$ norms under convolutions, which may be of independent interest. This theorem is stated in Section \ref{sec:inverse-thm} and proved in Section \ref{sec:inverse-thm-proof}. Here we give a heuristic description. Let $\mu,\nu$ be two probability measures supported on $2^{-m}\Z\cap [0,1]$. By Young's inequality, $\|\mu*\nu\|_q \le \|\mu\|_q$. The question the inverse theorem aims to answer is: what can be said if we are close to an equality? Here, and in the rest of the paper, ``close'' is meant in a very weak sense: up to some small exponential loss.  More concretely, the inverse theorem asserts that if $\|\mu*\nu\|_q \ge 2^{-\e m} \|\mu\|_q$ for some small $\e>0$, then $\mu$ and $\nu$ are forced to have a multi-scale structure of a certain kind. We note that equality in Young's theorem happens if either $\mu$ is the uniform measure on $2^{-m}\Z\cap [0,1]$, or if $\nu$ is a single atom. The inverse theorem asserts that, after restricting $\mu,\nu$ to suitable subsets $A,B$ which are ``large'' and ``regular'' in a certain sense, there is a multi-scale decomposition such that, at each scale, either $\mu|_A$ is ``almost uniform'' or $\nu|_B$ is ``almost discrete''. In spirit  this is not unlike  \cite[Theorem 2.7]{Hochman14}, although the details differ substantially; see Section \ref{sec:inverse-thm} below for further discussion. The two main tools in the proof of the inverse theorem come from additive combinatorics: an asymmetric version of the Balog-Szemer\'{e}di-Gowers Theorem, due to Tao and Vu, and a structure result on sets with ``small'' sumset, due to Bourgain. These results are recalled in Section \ref{sec:inverse-thm-proof}.

We note that the inverse theorem is a statement about arbitrary measures; no self-similarity is involved. Now let us consider a pleasant model $(X,\mathbf{T},\Delta,\lam)$ generating measures $\mu_x,x\in X$. The right-hand side in \eqref{eq:main-thm} is easily seen to be an upper bound for the left-hand side (for all $x$), so the task is to show the reverse inequality. The self-similarity expressed by \eqref{eq:mu-x-self-similarity}, in conjunction with the pleasantness of the model, can be used to show that there is a function $T:(1,\infty)\to [0,1]$, such that  $\tau_{\mu_x}=T$ for  $\PP$-almost all $x$, and $\tau_{\mu_x}\ge T$ for \emph{all} $x\in X$ - see Proposition \ref{prop:cocycle} and Corollary \ref{cor:unif-continuity}. Thus, in order to complete the proof, one needs to show that $T(q)/(q-1)$ equals the right-hand side of \eqref{eq:main-thm}.

We point out that the strategy of studying $L^q$ dimensions via the function $T(q)$ is borrowed from \cite{NPS12}. The innovation of this work consists in being able to calculate $T(q)$ for a wider range of models and, crucially, for all finite $q\ge 1$ (while the method of \cite{NPS12}, based on Marstrand's projection theorem, is restricted to $q\in (1,2]$).

It is known from general considerations that $T(q)$ is concave, so in particular it is continuous and differentiable outside of at most a countable set. The rest of the proof focuses on the study of $T(q)$ for a fixed differentiability point $q$. The ``multifractal structure'' of a measure $\mu$ is known to behave in a regular way for points $q$ of differentiability of the spectrum $\tau_\mu$. Extending some elementary results in this direction to the function $T(q)$, we show that if $\alpha=T'(q)$ exists and $\tau_{\mu_x}(q)=T(q)$ (which we have seen happens for almost all $x$) then, for large enough $m$, ``almost all'' of the contribution to the sum $\sum_{I\in\DD_m}\mu_x(I)^q$ comes from $\approx 2^{T^*(\alpha)m}$ intervals $I$ such that $\mu_x(I)\approx 2^{\alpha m}$; here $T^*$ is the Legendre transform of $T$. Moreover, using the self-similarity of $\mu_x$, we establish also a multi-scale version of this fact, see Proposition \ref{prop:Lq-over-small-set-is-small}.

Let $\mu_x^{(m)}$ supported on $2^{-m}\Z$ be given by
\begin{equation} \label{eq:def-mu-m}
\mu_x^{(m)}(j 2^{-m}) = \mu_x([j2^{-m},(j+1)2^{-m})).
\end{equation}
Then $\mu_x^{(m)}$ is a discretization of $\mu_x$ at scale $2^{-m}$, and $\|\mu_x^{(m)}\|_q^q = \sum_{I\in\DD_m} \mu_x(I)^q \lesssim 2^{-m T(q)}$.
The inverse theorem, together with the study of the multifractal structure of $\mu_x$, is used to show that either $T(q)=q-1$ (in which case we are done) or, otherwise, the following holds: if $\rho$ is an arbitrary measure supported on $2^{-m}\Z\cap [0,1]$ such that $\|\rho\|_q \le 2^{-\sigma m}$, then
\begin{equation} \label{eq:outline-conv-flattens-mu-x}
\|\rho*\mu_x^{(m)}\|_q^q \le 2^{-\e m} 2^{-T(q)m}  \quad\text{for all }x\in X,
\end{equation}
where $\e=\e(\sigma,q)>0$. That is, convolving $\mu_x$ with $\rho$ results in an exponential flattening of the $L^q$ norm (a priori this is not necessarily true for all $x$, since $\|\mu_x^{(m)}\|_q^q$ can be far smaller than $2^{-T(q)m}$ for some $x$, but all that is needed later is an exponential gain over $2^{-T(q)m}$). The heuristic reason for this is the following: suppose the opposite is true. The inverse theorem then asserts that there is a regular subset $A$ of $\supp(\mu_x^{(m)})$ which captures much of the $L^q$ norm. By the inverse theorem, and since $\rho$ is assumed to have exponentially small $L^q$ norm, $A$ must have almost full growth (or branching) on a positive density set of  scales in a multi-scale decomposition. But $A$ itself does not have full growth (this follows from the assumption $T(q)<q-1$, which rules out $\mu_x^{(m)}$ having too small $L^q$ norm). So there must also be a positive density set of scales on which $A$ has smaller than average growth. The regularity of the multifractal spectrum discussed above rules this out, since it forces $A$ to have an almost constant growth on almost all scales.

The conclusion of the proof of Theorem \ref{thm:L-q-dim-dyn-ssm} from \eqref{eq:outline-conv-flattens-mu-x} goes along the same lines of \cite{Hochman14}. By the exponential separation assumption, there is $x\in X$ such that $\tau_{\mu_x}(q)=T(q)$ and, for some $R=R(x)\in \N$,
\[
\frac{\log \|\mu_{x,n}^{(R m)}\|_q^q}{(q-1)n\log(1/\lam)}  = \frac{\|\mu_{x,n}\|_q^q}{(q-1)n\log(1/\lam)} = \frac{\sum_{i=0}^{n-1} \log\|\Delta(\mathbf{T}^i x)\|_q^q}{(q-1)n\log(1/\lam)},
\]
where $m=m(n)$ is chosen so that $2^{-m}\sim \lam^n$. Under our running assumption that $T(q)<q-1$, the ergodic theorem for uniquely ergodic systems implies that the right-hand side above tends to the right-hand side of \eqref{eq:main-thm} as $n\to\infty$. Hence, it remains to show that
\begin{equation} \label{eq:outline-Lq-norm-scale-Rm}
\lim_{n\to\infty}\frac{\log \|\mu_{x,n}^{(R m)}\|_q^q}{n\log(1/\lam)} = T(q).
\end{equation}
In other words, we need to show that the $L^q$ norm of $\mu_{x,n}$ at scale $2^{-m}\approx \lam^n$ (which is easily seen to be comparable to the $L^q$ norm of $\mu_x$ at scale $2^{-m}$) nearly exhausts the $L^q$ norm of $\mu_{x,n}$ at the much finer scale $2^{-R m}$ which, in turn, equals the full $L^q$ norm of $\mu_{x,n}$, by the exponential separation assumption.

To show \eqref{eq:outline-Lq-norm-scale-Rm}, we recall that $\mu_x = \mu_{x,n}* S_{\lam^n}\mu_{\mathbf{T}^n x}$, and use this to decompose
\[
\mu_{x}^{((R+1) m)} =  \sum_{I\in\DD_m} \mu(I)  \wt{\rho}_I * S_{\lam^n}\mu_{\mathbf{T}^n x},
\]
where $\wt{\rho}_I$ is the normalized restriction of $\mu_{x,n}$ to $I$. Since the supports of $\wt{\rho}_I * S_{\lam^n}\mu_{\mathbf{T}^n x}$ have bounded overlap, it is not hard to deduce that
\[
\|\mu_{x}^{((R+1) m)}\|_q^q \approx  \sum_{I\in\DD_m} \mu_x(I)^q \| \rho_I * \mu_{\mathbf{T}^n x}^{(R m)}\|_q^q,
\]
where $\rho_I  = S_{\lam^{-n}}\wt{\rho}_I$. This is the point where we apply \eqref{eq:outline-conv-flattens-mu-x}, to conclude that if on the right-hand side above we only add over those $I$ such that $\|\rho_I\|_q \ge 2^{-\sigma q}$, where $\sigma>0$ is arbitrary, then, provided $n$ is large enough depending on $\sigma$, we still capture almost all of the left-hand side. This follows since \eqref{eq:outline-conv-flattens-mu-x} can be shown to imply that the contribution of the remaining $I$ is exponentially smaller than the left-hand side. A similar calculation, now with $\mu_{x,n}^{((R+1)m)}$ in place of $\mu_{x}^{((R+1) m)}$ in the left-hand side, then shows that \eqref{eq:outline-Lq-norm-scale-Rm} holds, finishing the proof.

We point out that, simultaneously and independently of this work, Meng Wu \cite{Wu16} obtained an elegant alternative proof of Theorem \ref{thm:Furstenberg}. Wu's proof is purely ergodic-theoretical and completely different from ours. His methods do not seem  to yield any analogs of Theorem \ref{thm:L-q-dim-dyn-ssm} and, in particular, are unable to reproduce our results on the dimensions and densities of Bernoulli convolutions. Nevertheless, some of our concrete applications (besides Furstenberg's conjecture) also follow from Wu's approach: this is the case for Corollaries \ref{cor:dim-intersection-ssm} and \ref{cor:dim-intersection-sss-lines}.

\subsection{Organization of the paper and summary of applications}

We outline the organization of the rest of the paper. Sections \ref{sec:inverse-thm}--\ref{sec:proof-of-main-thm} are devoted to the proof of Theorem \ref{thm:L-q-dim-dyn-ssm}, while the remaining Sections \ref{sec:dim-ssm-and-applications}--\ref{sec:abs-cont} contain the applications of Theorem \ref{thm:L-q-dim-dyn-ssm}. More precisely:
\begin{itemize}
 \item In Section  \ref{sec:inverse-thm} we state and discuss the inverse theorem for the $L^q$ norms of convolutions of discrete measures. The inverse Theorem is proved in Section \ref{sec:inverse-thm-proof}.
 \item Section \ref{sec:properties-dyn-ssm} develops some properties of dynamically driven self-similar measures. In Section \ref{sec:proof-of-main-thm}, these are combined with the inverse theorem to conclude the proof of Theorem \ref{thm:L-q-dim-dyn-ssm}.
 \item In Section \ref{sec:dim-ssm-and-applications} we apply Theorem \ref{thm:L-q-dim-dyn-ssm} to study $L^q$ dimensions and Frostman exponents of self-similar measures on the line. In particular, we prove Theorem \ref{thm:dim-ssm}, which generalizes Theorem \ref{thm:infinity-dim-BCs} to homogeneous self-similar measures on $\R$, and Theorem \ref{thm:dim-general-ssm}, which extends this to arbitrary self-similar measures on the line (not necessarily homogeneous).
 \item In Section \ref{sec:convolutions-and-Furst-conj}, we conclude the proof of Theorem \ref{thm:Furstenberg}. We also study the $L^q$ dimensions of convolutions of self-similar measures (Theorems \ref{thm:dim-conv-ssm} and \ref{thm:dim-conv-many-ssm}), and deduce a variant of Furstenberg's conjecture for self-similar sets, Corollary \ref{cor:dim-intersection-ssm}.
 \item Section \ref{sec:planar-sss} contains further applications of Theorem \ref{thm:L-q-dim-dyn-ssm} to projections and sections of planar self-similar sets and measures. In particular, we prove an upper bound for the dimensions of arbitrary linear sections of some self-similar sets on the plane, see Corollary \ref{cor:dim-intersection-sss-lines}.
 \item Finally, in Section \ref{sec:abs-cont} we turn our focus to the densities of the measures studied in the previous sections. We present a general result in the framework of dynamically defined measures, Theorem \ref{thm:abs-continuity}, and deduce Theorem \ref{thm:abc-cont-BCs}, as well as several other applications, as corollaries.
\end{itemize}

\subsection{Notation}

We use Landau's $O(\cdot)$ and related notation: if $X,Y$ are two positive quantities, then $Y=O(X)$ means that $Y\le C X$ for some constant $C>0$, while $Y=\Omega(X)$ means that $X=O(Y)$, and $Y=\Theta(X)$ that $Y=O(X)$ and $X=O(Y)$. If the constant $C$ is allowed to depend on some parameters, these are often denoted by subscripts. For example, $Y=O_q(X)$ means that $Y\le C(q) X$, where $C(q)$ is a function depending on the parameter $q$.

The following table summarizes some of the notational conventions to be used throughout the paper.

\begin{tabular}{lll}
\hline
$\N$ & & Natural numbers $\{1,2,\ldots\}$ \tabularnewline
$B(x,r)$ & & Open ball of center $x$, radius $r$. \tabularnewline
$\hdim$ & & Hausdorff dimension \tabularnewline
$\ubdim$ & & Upper box-counting dimension \tabularnewline
$[n]$ & & $\{0,1,\ldots,n-1\}$ \tabularnewline
$\delta,\e,\eta,\kappa,\sigma$ & & Small positive numbers \tabularnewline
$\mu,\nu,\eta,\rho$ & & Measures (always positive and finite, often discrete) \tabularnewline
$\mu^{(m)}$ & & Discretization of $\mu$ at scale $2^{-m}$ \tabularnewline
$\|\cdot\|_q$ & & Discrete $L^q$ norm \tabularnewline
$q'$ & & Dual exponent to $q$ \tabularnewline
$\delta_t,\delta(t)$ & & Delta mass at $t$ \tabularnewline
$\mathcal{A}$ & & Space of finitely supported measures \tabularnewline
$\Delta_i, \wt{\Delta}$  & & Elements of $\mathcal{A}$ \tabularnewline
$\Delta(x)$ & &  $\mathcal{A}$-valued functions \tabularnewline
$(X,\mathbf{T},\Delta,\lam)$ &  & A model generating DDSSMs \tabularnewline
$\mu_x$ & & The DDSSM corresponding to $x\in X$ \tabularnewline
$\mu_{x,n}$ & & Discrete approximations to $\mu_x$ \tabularnewline
$S_\lam$ & & Map that scales by $\lam$ \tabularnewline
$\tau(\mu,q)$ or $\tau_\mu(q)$ & & $L^q$ spectrum  \tabularnewline
$D(\mu,q)$ or $D_\mu(q)$ & & $L^q$ dimension \tabularnewline
$\mathcal{E},\mathcal{E}_i$ & & Small exceptional sets \tabularnewline
$\mathcal{D}_s$ & & Dyadic intervals of length $2^{-s}$ \tabularnewline
$\mathcal{D}_s(A)$ & & Elements of $\mathcal{D}_s$ hitting $A$ \tabularnewline
$\mathcal{N}_s(A)$ or $\mathcal{N}(A,s)$ & & $|\mathcal{D}_s(A)|$ \tabularnewline
$D$ & & $2^D$=base for tree representation of sets \tabularnewline
$\ell$ & & Height of tree representing a set \tabularnewline
$\mathcal{S},\mathcal{S}',\mathcal{S}_i$ & & Subsets of $[\ell]$ (representing sets of scales) \tabularnewline
$R_s, R'_s, R''_s$ & & Branching numbers of trees representing regular sets \tabularnewline
$T(q)$ & & The function from Proposition \ref{prop:cocycle} \tabularnewline
$f,g,h$ & & Maps $\R\to\R$, often affine \tabularnewline
$(f_i)_{i\in\mathcal{I}}$ && Iterated function system of similarities \tabularnewline
\hline
\end{tabular}

\bigskip

\textbf{Acknowledgments}. I am grateful to Mike Hochman and Izabella {\L}aba for inspiring discussion related to the themes in this paper, and to Julien Barral and Eino Rossi for a careful reading and for suggesting numerous small corrections. I also thank the anonymous referees for helpful comments.

\section{An inverse theorem for the decay of \texorpdfstring{$L^q$}{Lq} norms under convolution}
\label{sec:inverse-thm}

Let $\mu,\nu$ be probability measures on $\R$ (or the circle $\R/\Z$). For any reasonable notion of smoothness, the convolution $\mu*\nu$ is at least as smooth as $\nu$. A natural question is then: if $\mu*\nu$ is not ``much smoother'' than $\mu$, can we deduce any information about the measures $\mu$ and $\nu$? Of course, this depends on the notion of smoothness under consideration, and on the precise meaning of ``much smoother''.

We are interested in general, possibly fractal, measures, and their discrete approximations. A general method for defining notions of dimension (or smoothness) of a measure is to discretize it at a certain scale $\e$, measure smoothness at that scale in some standard way (for example, by means of entropy or $L^q$ norms) and then study the growth/decay of this quantity as $\e\downarrow 0$. Indeed, we have seen that $L^q$ dimensions are defined precisely in this way, and there is a parallel notion for entropy.

Let $\mu$ be a probability measure on $\R/\Z$. Its normalized level $m$ entropy is
\[
H_m(\mu) = \frac{1}{m}\sum_{I\in\mathcal{D}_m} -\mu(I)\log(\mu(I)),
\]
with the usual convention $0\log 0=0$.  In \cite[Theorem 2.7]{Hochman14}, Hochman showed that if
\[
H_m(\nu*\mu) \le H_m(\mu)+\e,
\]
where $\e>0$ is small, then $\nu$ and $\mu$ have a certain structure which, very roughly, is of this form: the set of dyadic scales $0\le s<m$ can be split into three sets $\mathcal{A}\cup \mathcal{B}\cup \mathcal{C}$. At scales in $\mathcal{A}$, the measure $\nu$ looks ``roughly atomic'', at scales in $\mathcal{B}$ the measure $\mu$ looks ``roughly uniform'', and the set $\mathcal{C}$ is small. This theorem was motivated in part by its applications to the dimension theory of self-similar measures, as discussed above.

Our goal is to develop a corresponding theory for $L^q$ norms. Given $m\in\N$, we will say that $\mu$ is a $2^{-m}$-measure if $\mu$ is a probability measure supported on $2^{-m}\Z\cap [0,1)$ (and we sometimes identify $[0,1)$ with the circle). Recall from \eqref{eq:def-mu-m} that if $\mu$ is a probability measure on $[0,1)$, we denote by $\mu^{(m)}$ the associated $2^{-m}$-measure, that is, $\mu^{(m)}(j 2^{-m}) = \mu([j 2^{-m},(j+1)2^{-m}))$. We also recall that, given a purely atomic measure $\mu$, we define the $L^q$ norms
\[
\|\mu\|_q = \left(\sum \mu(y)^q\right)^{1/q},
\]
and $\|\mu\|_\infty= \max_y \mu(y)$.

From now on, the convolutions are always assumed to take place on the circle unless otherwise indicated; however, all results immediately transfer to the real line, using the fact that the map $(x,y)\mapsto x+y$ is two-to-one on the circle so, for example, if $\mu,\nu$ are $2^{-m}$-measures, then the $L^q$ norms of $\mu*\nu$ as convolutions on the circle or the real line are comparable up to a multiplicative constant.

By Young's inequality (which in this context is a direct consequence the convexity of $t\mapsto t^q$), we know that $\|\mu*\nu\|_q\le \|\mu\|_q$, for any $q\ge 1$. We aim to understand under what circumstances $\|\mu*\nu\|_q \approx \|\mu\|_q$, where the closeness is in a weak, exponential sense.  More precisely, we are interested in what structural properties of the measures $\mu,\nu$ ensure an exponential flattening of the $L^q$ norm of the form
\begin{equation} \label{eq:intro-L2-norm-flattens}
\|(\mu*\nu)^{(m)}\|_q \le 2^{-\e m} \|\mu^{(m)}\|_q.
\end{equation}
The Balog-Szemer\'{e}di-Gowers Theorem (particularly, its asymmetric formulation, see Theorem \ref{thm:BSG} below) can be seen as providing a partial answer in a special case, i.e. when $\mu^{(m)},\nu^{(m)}$ are indicator functions.

While we are not aware of any general results in this direction, we note that a special case has received considerable attention: if $A\subset 2^{-m}\Z$, then $\|\mathbf{1}_A*\mathbf{1}_A\|_2^2$ is nothing but the additive energy of $A$ (see \eqref{eq:def-additive-energy} below), and estimates of the form
\[
\|\mathbf{1}_A*\mathbf{1}_A\|_2^2 \le |A|^{3-\e}
\]
arise repeatedly in dynamics, combinatorics and analysis: see e.g. \cite{DyatlovZahl16, ALL17} for some recent examples. In particular, S. Dyatlov and J. Zahl \cite[Theorem 6]{DyatlovZahl16} showed that if $\mu$ is an Ahlfors-regular measure, that is, if there are $C,s>0$ such that
\[
C^{-1} r^s \le \mu(B(x,r)) \le C r^s \quad\text{for all }x\in\supp(\mu), r\in (0,1],
\]
then
\[
\|(\mu*\mu)^{(m)}\| \le 2^{-\e m}\|\mu^{(m)}\|_2,
\]
where $\e>0$ depends only on the parameters $C,s$. Their proof does not appear to readily extend to the convolution of two different measures, or beyond the Ahlfors-regular case. Outside of the Euclidean setting, the $L^2$ norm of self-convolutions has been studied in many groups as part of the Bourgain-Gamburd expansion machine developed to prove that Cayley graphs are expanders, see e.g. \cite{BourgainGamburd08}.

Here we go in a different direction, by investigating general geometric conditions on the measures $\mu,\nu$ that ensure flattening in the sense of \eqref{eq:intro-L2-norm-flattens}. We make the trivial observation that if $\nu=\delta_{k 2^{-m}}$ or $\mu=\lam=$Lebesgue measure on $\R/\Z$, then $\|(\mu*\nu)^{(m)}\|_q = \|\mu^{(m)}\|_q$. Furthermore, if $\nu=2^{-\e m}\delta_x + (1- 2^{\e m})\lambda$ and $\mu$ is an arbitrary measure, then we still have $\|(\mu*\nu)^{(m)}\|_q \ge 2^{-\e m}\|\mu^{(m)}\|_q$. This shows that a subset of measure $2^{-\e m}$ is able to prevent smoothening in the sense of \eqref{eq:intro-L2-norm-flattens}, so that (unlike the case of entropy) in order to guarantee exponential smoothing we need to impose conditions on the structure of the measures inside sets of exponentially small measure.

There are also less trivial situations in which $\|\mu*\mu\|_q\approx \|\mu\|_q$. Let $D\gg 1$ be a large integer, fix $\ell\gg D$, and for  given subset $\mathcal{S}$ of $\{0,\ldots,\ell-1\}$ define $A$ as the set of all $x\in 2^{-\ell D}\Z\cap [0,1)$, such that the $s$-th digit in the $2^{-D}$-base expansion of $x$ is $0$ for all $s\in\mathcal{S}$ (and is arbitrary otherwise). Then it is not hard to check that $\|\mathbf{1}_A*\mathbf{1}_A\|_q\approx \|\mathbf{1}_A\|_1\|\mathbf{1}_A\|_q$. In more combinatorial terms, $A$ looks like an arithmetic progression at all scales. In similar ways one constructs probability measures $\mu,\nu$ supported on sets of widely different sizes, such that $\|\mu*\nu\|_q\approx \|\mu\|_q$.

Our inverse theorem asserts that if \eqref{eq:intro-L2-norm-flattens} fails to hold then one can find subsets $A\subset \supp(\mu)$ and $B\subset\supp(\nu)$, such that $A$ captures a ``large'' proportion of the $L^q$ norm of $\mu$ and $B$ a ``large'' proportion of the mass of $\nu$, and moreover $\mu|_A,\nu|_B$ are fairly regular (for example, they are constant up to a factor of $2$). The main conclusion, however, is that $A$ and $B$ have a structure resembling the example above, and also the conclusion of Hochman's inverse theorem for entropy: if $D$ is a large enough integer, then for each $s$, either $B$ has no branching between scales $2^{sD}$ and $2^{(s+1)D}$ (in other words, once the first $s$ digits in the $2^D$-adic expansion of $y\in B$ are fixed, the next digit is uniquely determined), or $A$ has nearly full branching between scales $2^{sD}$ and $2^{(s+1)D}$ (whatever the first $s$ digits of $x\in A$ in the $2^D$-adic expansion, the next digit can take almost any value).

Before stating the theorem, we summarize our notation for dyadic intervals to be used throughout the paper (some of it was introduced before):
\begin{itemize}
\item $\DD_s$ is the family of dyadic intervals $[j 2^{-s},(j+1) 2^{-s})$. We also refer to elements of $\DD_s$ as $2^{-s}$-intervals.
\item Given a set $A\subset\R$ or $\R/\Z$, we write $\DD_s(A)$ for the family of $2^{-s}$-intervals that hit $A$. We also write $\mathcal{N}(A,s)$ or $\NN_s(A)$ for $|\DD_s(A)|$, i.e. the number of $2^{-s}$ intervals that hit $A$.
\item Given $x\in\R$ or $\R/\Z$, we write $\DD_s(x)$ for the only $2^{-s}$-interval that contains $x$.
\item We write $a I$ for the interval of the same center as $I$ and length $a$ times the length of $I$.
\end{itemize}
We also write $[\ell]=\{0,1,\ldots,\ell-1\}$.

\begin{thm} \label{thm:inverse-thm}
Given $q\in (1,\infty)$, $\delta>0$ and $D_0\in\N$, there are $\e>0$, $D\ge D_0$, such that the following holds for all large enough $\ell$.

Let $m=\ell D$, and let $\mu,\nu$ be $2^{-m}$-measures such that
\[
\|\mu*\nu\|_q \ge 2^{-\e m}\|\mu\|_q.
\]
After translating the measures $\mu,\nu$ by appropriate numbers of the form $k 2^{-m}$, there exist sets $A\subset \supp(\mu), B\subset \supp(\nu)$, such that:
\begin{enumerate}
\item[\rm{(A-i)}] $\|\mu|_A\|_q \ge 2^{-\delta m}\|\mu\|_q$, where $\mu|_A$ denotes the (non-normalized) restriction of $\mu$ to $A$.
\item[\rm{(A-ii)}] $\mu(y)\le 2\mu(x)$ for all $x,y\in A$.
\item[\rm{(A-iii)}] There is a sequence $R'_s$, $s\in [\ell]$, such that $\NN_{(s+1)D}(A\cap I)=R'_s$ for all $I\in \DD_{s D}(A)$.
\item[\rm{(A-iv)}] $x\in \frac12 \DD_{s D}(x)$ for every $x\in A, s\in [\ell]$.

\item[\rm{(B-i)}] $\nu(B)\ge 2^{-\delta m}$.
\item[\rm{(B-ii)}]$\nu(y)\le 2\nu(x)$ for all $x,y\in B$.
\item[\rm{(B-iii)}] There is a sequence $R''_s$, $s\in [\ell]$, such that $\NN_{(s+1)D}(B\cap I)=R''_s$ for all $I\in\DD_{s D}(B)$.
\item[\rm{(B-iv)}] $y\in \frac12 \DD_{s D}(y)$ for every $y\in B, s\in [\ell]$.
\end{enumerate}

Moreover,
\begin{enumerate}
\item[\rm{(v)}] For each $s$, either $R''_s=1$, or
\begin{equation} \label{eq:scales-with-full-branching}
 R'_s \ge 2^{(1-\delta)D}.
 \end{equation}
\item[\rm{(vi)}] Let $\mathcal{S}$ be the set of $s$ such that \eqref{eq:scales-with-full-branching} holds. Then
\[
 \log(\|\nu\|_q^{-q'})   - \delta m \le D |\mathcal{S}| \le  \log(\|\mu\|_q^{-q'})+\delta m.
\]
\end{enumerate}
\end{thm}

Here, and throughout the paper, $q'=q/(q-1)$ denotes the dual exponent. We make some remarks on the statement.
\begin{enumerate}[label={\upshape\alph*)}]
\item The initial translation of the measures, as well as their convolution, take place  on the circle. However, by decomposing the measures into finitely many pieces it is easy to deduce the same statement with both the translation and the convolution taking place on the real line.
\item The translation is only needed for (A-iv) and (B-iv), which are technical claims that we include in the theorem as they are often useful in applications.
\item The main claim in the theorem is part (v). Obtaining sets $A,B$ satisfying (A-i)--(A-iv) and (B-i)--(B-iv) is not hard, and (vi) is a straightforward calculation using (v).
\item The theorem fails for $q=1$ and $q=\infty$. In the first case there is an equality $\|\mu*\nu\|_1=\|\mu\|_1$ for any $2^{-m}$-measures, and in the second case there is always an equality $\|\1_A*\1_{-A}\|_\infty= \|\1_A\|_1$. On the other hand, the case of arbitrary $1<q<\infty$ is easily reduced to the case $q=2$: see Lemma \ref{lem:q-to-2} below.
\end{enumerate}

We emphasize that the proof of Theorem \ref{thm:inverse-thm} (including the proofs of the results it relies on) is elementary, in particular avoiding any use of the Fourier transform or quantitative probabilistic estimates such as the Berry-Esseen Theorem, which is crucial in the approach of \cite{Hochman14}. The value of $\e$ is effective in principle, although it is certainly very poor; the worst loss occurs in the application of the asymmetric Balog-Szemer\'{e}di-Gowers Theorem (Theorem \ref{thm:BSG} below).

\section{Proof of the inverse theorem}
\label{sec:inverse-thm-proof}

\subsection{Preliminaries}

In this section we prove Theorem \ref{thm:inverse-thm}. We begin by describing the two main tools involved in the proof: a version of the Balog-Szemer\'{e}di-Gowers Theorem that is effective even when the sets have very different sizes, due to Tao and Vu, and the additive part of Bourgain's discretized sum-product theorem. We begin with the latter.

We say that $A\subset [0,1]$ or $\R/\Z$ is a $2^{-m}$-set if each element of $A$ is an integer multiple of $2^{-m}$. For a finite set $A\subset\R$, we define its doubling constant as $\sigma[A]=|A+A|/|A|$. We will call a $2^{-m}$ set $A$ such that $\sigma[A]\le 2^{\delta m}$ an \textbf{$(m,\delta)$-small doubling set}.

The structure of sets $A$ such that $\sigma[A]\le K$ (where $K$ is independent of $|A|$) is characterized by Freiman's Theorem (see e.g. \cite[Theorem 5.32]{TaoVu10}): such sets can be densely embedded in a generalized arithmetic progression. However Freiman's Theorem gives no information when the doubling constant grows exponentially with the size of the set. The following structural property of sets with small exponential doubling is proved by Bourgain \cite{Bourgain10}. Although it is not explicitly stated in \cite{Bourgain10}, this theorem emerges from the constructions in Sections 2 and 3, in particular see \cite[Equations (3.15), (3.20), (3.21), (3.22)]{Bourgain10}.

\begin{thm} \label{thm:Bourgain}
 Given a large $T\in\N$,   the following holds for sufficiently large $m_1\in \N$ (depending on $T$).

 Let $m=m_1 T$, and suppose $H$ is a $(m,2^{-2T-1})$-small doubling set. Then $H$ contains a subset $H'$ such that the following holds:
 \begin{enumerate}[label={\upshape(\roman*).}]
  \item $|H'|\ge 2^{-(2  \log T/\sqrt{T} )m} |H|$.
  \item There are a set $\mathcal{S}\subset \{0, \ldots,m_1-1\}$ and integers $R_s, n_s$, $s\in\mathcal{S}$, with $n_s\in [s T, (s+1)T)$, such that:
  \begin{enumerate}
   \item If $s\notin\mathcal{S}$, then $\mathcal{N}(H'\cap I,(s+1)T)= 1$ for each $I\in\mathcal{D}_{sT}(H')$.
   \item If $s\in\mathcal{S}$, then $\mathcal{N}(H'\cap I,n_s)=R_s$ for each $I\in\mathcal{D}_{sT}(H')$, and $\mathcal{N}(H'\cap J,(s+1)T)= 1$ for each $J\in\mathcal{D}_{n_s}(H')$.
   \item $2^{(1-T^{-1/2})(n_s-sT)} < R_s \le 2^{n_s - sT}$ for all $s\in\mathcal{S}$.
  \end{enumerate}
 \end{enumerate}
 In particular, $|H'|=\prod_{s\in\mathcal{S}} R_s$.
\end{thm}
Thus, the theorem says that a set with small exponential doubling contains a fairly dense subset which has no branching between the scales $2^{-Ts}$ and $2^{-(T+1)s}$ for $s\notin\mathcal{S}$ and between the scales $2^{-n_s}$ and $2^{-(T+1)s}$ for $s\in\mathcal{S}$; and has ``uniform and nearly full branching'' between the scales $2^{-Ts}$ and $2^{n_s}$, $s\in\mathcal{S}$.

We remark that the proof of Theorem \ref{thm:Bourgain} is ingenious but elementary, only relying on the Pl\"{u}nnecke-Ruzsa inequalities, for which a short elementary proof was recently found by Petridis \cite{Petridis11}.

Another crucial ingredient in the proof of Theorem \ref{thm:inverse-thm} is the following version of the celebrated Balog-Szemer\'{e}di-Gowers Theorem, due to Tao and Vu \cite{TaoVu10}, which allows the sets to have widely different sizes. Recall that the additive energy between two finite sets $A,B$ in a common ambient group is
\begin{equation} \label{eq:def-additive-energy}
E(A,B) = |\{ (a_1,a_2,b_1,b_2)\in A^2\times B^2: a_1+b_1=a_2+b_2 \}| = \|\1_A*\1_B\|_2^2.
\end{equation}
\begin{thm}[Asymmetric Balog-Szemer\'{e}di-Gowers] \label{thm:BSG}
Given $\kappa>0$, there is $\tau>0$ such that the following holds for $m\in\N$ large enough. Let $A,B\subset[0,1]$ or $\R/\Z$ be $2^{-m}$-sets such that
\[
E(A,B) \ge 2^{-\tau m} |A| |B|^2  =2^{-\tau m} \|\mathbf{1}_B\|_1^2 \|\mathbf{1}_A\|_2^2.
\]
Then there are a $(m,\kappa)$-small doubling set $H$ and a $2^{-m}$-set $X$ such that:
\begin{enumerate}[label={\upshape(\roman*).}]
 \item $|A\cap (X+H)| \ge 2^{-\kappa m} |A|\ge  2^{-2\kappa m}|X||H|$,
 \item $|B\cap H| \ge 2^{-\kappa m} |B|$.
\end{enumerate}
\end{thm}
\begin{proof}
This follows from \cite[Theorem 2.35]{TaoVu10}. Indeed, take $L=2^m$, $\alpha=2^{-\tau m}/2$, $\e=\kappa/4$. Then by making $\tau>0$ small enough in terms of $\kappa$, we can ensure that
\[
\Omega_\kappa \left(\alpha^{O_\kappa(1)} L^{-\kappa/4}\right)\ge \Omega_\kappa(2^{-(\kappa/2)m}) \ge 2^{-\kappa m}
\]
if $\delta$ is small enough and $m$ large enough in terms of $\kappa$.
\end{proof}
Thus, the theorem asserts that a big part of $B$ is contained in a set with small doubling $H$, and a big part of $A$ is densely contained in a union of (nearly) disjoint translates of $H$ (with $X$ being the set of translations). In particular, $H$ cannot be much smaller than $B$ (but it can be much larger), and $X$ has size approximately $|A|/|H|$. The proof of Theorem \ref{thm:BSG} is also elementary, although it is rather lengthy.

\subsection{Overview}

We give a rough sketch of the proof of the inverse theorem. Our goal is to apply the asymmetric Balog-Szemer\'{e}di-Gowers Theorem, Theorem \ref{thm:BSG}.  In \S\ref{subsec:analytical-lemmas} we present two lemmas involving $L^q$ norms. Recall that our assumption is that $\|\mu*\nu\|_q\ge 2^{-\e m}\|\mu\|_q$.  In Lemma  \ref{lem:non-flattening-to-subset} we extract two sets $A,B$, which already satisfy properties (A-i), (A-ii), (B-i), (B-ii), and such that similar bounds hold for their indicator functions. Lemma \ref{lem:q-to-2} (a simple application of H\"{o}lder's inequality) shows that one can pass from the $L^q$ norm to the $L^2$ norm, enabling the application of Theorem \ref{thm:BSG}.

In \S\ref{subsec:combinatorial-lemmas}, we present some  combinatorial regularization lemmas, inspired in \cite{Bourgain10}. Theorem \ref{thm:BSG} produces a set $H$ of small exponential doubling such that $B+H$ is not much larger than $H$ and $A+H$ is not much larger than $A$. Together with the information on  the structure of $H$ provided by Theorem \ref{thm:Bourgain} (or, rather, the version given by Corollary \ref{cor:Bourgain} below), and with the lemmas in \S\ref{subsec:combinatorial-lemmas}, this allows us to deduce the remaining properties of $A$ and $B$ (after passing to suitable dense subsets).

Finally, (vi) is a straightforward consequence of the previous claims.

A point of notation: throughout this section, $\ell$ and $m$ will denote sufficiently large integers (given any other relevant data); any inequalities involving them are understood to hold if they are larger than a constant that is allowed to depend on any other parameters involved.

\subsection{Analytical lemmas}
\label{subsec:analytical-lemmas}

We begin with a lemma, based on Young's inequality and dyadic pigeonholing, that enables the use of the Balog-Szemer\'{e}di-Gowers Theorem. It is an $L^q$ asymmetric version of (the proof of) \cite[Proposition 2]{BourgainGamburd08}.
\begin{lemma} \label{lem:non-flattening-to-subset}
Given $\e>0$ and $q\in (1,\infty)$, the following holds for large enough $m\in\N$. Suppose $\mu,\nu$ are $2^m$-measures satisfying $\|\mu*\nu\|_q \ge 2^{-\e m}\|\mu\|_q$. Then there exist $j, j'\le 2\e q' m$ such that, setting
\begin{align*}
A &= \{ x: 2^{-j-1}\|\mu\|_q^{q'} < \mu(x) \le 2^{-j}\|\mu\|_q^{q'}\},\\
B &= \{ y: 2^{-j'-1}2^{-m} < \nu(y) \le 2^{-j'}2^{-m}\},
\end{align*}
the following holds:
\begin{enumerate}[label={\upshape(\roman*).}]
\item $\|\mathbf{1}_A*\mathbf{1}_B\|_q \ge 2^{-2\e m}\|\mathbf{1}_A\|_q\|\mathbf{1}_B\|_1$,
\item $\|\mu|_A\|_q \ge 2^{-2\e m} \|\mu\|_q$,
\item $\|\nu|_B\|_1 =\nu(B) \ge 2^{-2\e m}$.
\end{enumerate}
\end{lemma}
\begin{proof}

We use the notation $X\gtrsim Y$ to mean $X \ge C^{-1}m^{-C} Y$, where $C>0$ depends on $q$ only. For $j\in\Z$, let
\begin{align*}
A_j &= \{ x: 2^{-j-1}\|\mu\|_q^{q'} < \mu(x) \le 2^{-j}\|\mu\|_q^{q'}\},\\
B_j &= \{ y: 2^{-j-1}2^{-m} < \nu(y) \le 2^{-j}2^{-m}\}.
\end{align*}
Firstly, note that $A_j=\varnothing, B_j=\varnothing$ if $j\le -(m+1)$ since, by H\"{o}lder's inequality,
\[
1= \sum_x \mu(x) \le \|\mu\|_q 2^{m/q'} \,\Longrightarrow \, \|\mu\|_q^{q'} \ge 2^{-m}.
\]
Write $\ell =\lceil 2\e q' m\rceil$, and let $E=\cup_{j\ge\ell} A_j, F=\cup_{j\ge\ell} B_j$. Note that
\begin{align*}
\|\mu|_E\|_q^q &\le \left(\max_{x\in E}\mu(x)^{q-1}\right) \sum_{x\in E}\mu(x) \le 2^{-\ell(q-1)}\|\mu\|_q^q,\\
\|\nu|_F\|_1 &= \nu(F) \le 2^m 2^{-\ell-m} = 2^{-\ell}.
\end{align*}
 By Young's inequality,
\[
\max(\|\mu|_E *\nu\|_q,\|\mu*\nu|_F\|_q) \le 2^{-\ell/q'} \|\mu\|_q  \le 2^{-\e m} \|\mu*\nu\|_q.
\]
It follows from the bilinearity of convolution and the triangle inequality that, if $m\gg_\e 1$,
\[
\sum_{- m\le j,j'<\ell}   \|\mu|_{A_j}*\nu|_{B_{j'}}\|_q  \ge \frac12 \|\mu*\nu\|_q.
\]
Pigeonholing and applying Young's inequality once again, we can pick $j,j'<\ell$ such that, setting $A=A_j, B=B_{j'}$, we have
\[
\|\mu|_A\|_q \|\nu|_B\|_1 \ge  \|\mu|_A *\nu|_B\|_q \gtrsim  \|\mu*\nu\|_q \ge 2^{-\e m}\|\mu\|_q %\ge 2^{-\e m}\|\mu|_A\|_q.
\]
From here it follows that $\|\nu|_B\|_1 \gtrsim 2^{-\e m}$ and $\|\mu|_A\|_q \gtrsim 2^{-\e m}\|\mu\|_q$. Note that $2^{j'+m}\gtrsim |B|$. We conclude that
\begin{align*}
\|\1_A*\1_B\|_q &\gtrsim (2^{j}\|\mu\|_q^{-q'} 2^{j'+m}) \|\mu|_A*\nu|_B\|_q \\
&\gtrsim  (2^{j}\|\mu\|_q^{-q'} 2^{j'+m})  2^{-\e m}\|\mu|_A\|_q \\
&\gtrsim 2^{-\e m} \|\1_A\|_q \|\1_B\|_1.
\end{align*}
\end{proof}

The following simple consequence of H\"{o}lder's inequality will allow us to apply the Balog-Szmer\'{e}di-Gowers also in the context of $L^q$ norms, $q\in (1,+\infty)$.
\begin{lemma} \label{lem:q-to-2}
Let $A,B$ be two $2^{-m}$-sets and let $q\in (1,\infty)$. If $\|\1_A *\1_B\|_q\ge 2^{-\kappa m}|A|^{1/q}|B|$, then
\[
\|\1_A *\1_B\|_2^2 \ge 2^{-(\max(q,q'))\kappa m} |A||B|^2.
\]
\end{lemma}
\begin{proof}
Consider first the case $q\in (1,2)$. Applying H\"{o}lder's inequality in the form
\[
\sum_x f(x)^q = \sum_x f(x)^{2-q} f(x)^{2(q-1)} \le \left(\sum_x f(x)\right)^{2-q} \left(\sum_x f(x)^2\right)^{q-1}
\]
to $f=\1_A*\1_B$ yields
\[
\|\1_A*\1_B\|_q^q \le  |A|^{2-q} |B|^{2-q} \|\1_A*\1_B\|_2^{2(q-1)}.
\]
Hence, using the assumption,
\[
\|\1_A*\1_B\|_2^{2(q-1)}  \ge 2^{-q\kappa m}\frac{|A| |B|^q}{|A|^{2-q} |B|^{2-q}} = 2^{-q\kappa m}|A|^{q-1} |B|^{2(q-1)},
\]
which gives the claim when $q\in (1,2)$.

Suppose now $q\in (2,+\infty)$. Then
\[
2^{-q\kappa m} |A||B|^q \le \|\1_A*\1_B\|_q^q  \le \|\1_A*\1_B\|_2^2 \|\1_A*\1_B\|_\infty^{q-2} \le \|\1_A*\1_B\|_2^2 |B|^{q-2},
\]
and this completes the proof.
\end{proof}

\subsection{Combinatorial lemmas}
\label{subsec:combinatorial-lemmas}

In this section we establish several elementary combinatorial lemmas. In both the statement and the proof of Theorem \ref{thm:inverse-thm} an important r\^{o}le is played by sets with a ``regular tree structure''. We begin by formalizing this concept. Recall that $[\ell]=\{0,1,\ldots,\ell-1\}$.

\begin{defn}
Let $D,\ell\in\N$ and set $m=\ell D$. Given a sequence $(R_s)_{s\in [\ell]}$ taking values in $[1,2^D]$, we say that a $2^{-m}$-set $A$ is \textbf{$(D,\ell,R)$-uniform} if $\NN(A\cap I,(s+1)D)=R_s$ for each $s\in[\ell]$ and $I\in\DD_{s\ell}(A)$.

Further, we say that $A$ is \textbf{$(D,\ell)$-uniform} if there is a sequence $R$ such that $A$ is $(D,\ell,R)$-uniform.
\end{defn}

Given an arbitrary $2^{-m}$-set $A$ and $D|m$, one may associate to it the tree whose vertices of level $s$ are the $2^{-s D}$-intervals intersecting $A$. Then $A$ is $(D,\ell)$-uniform if and only if the associated tree is spherically symmetric, i.e. the number of offspring of a vertex is constant over all vertices at the same distance to the root (but may vary between vertices of different levels). We will often informally refer to the tree description of sets, for example by speaking of branching at certain levels.

In our first lemma we show that any $2^{-m}$ set contains a fairly large uniform subset. This fact goes back at least to \cite{Bourgain10}; we provide details for completeness.
\begin{lemma} \label{lem:regular-subset}
Let $D,\ell\in \N$, and let $A$ be a $2^{-m}$-set, where $m=\ell D$. Then there exists a $(D,\ell)$-uniform subset $A'\subset A$  such that
\[
|A'|\ge (2D)^{-\ell}|A| = 2^{(-\log(2D)/D) m}|A|.
\]
\end{lemma}
\begin{proof}
The construction is similar to that in \cite[Section 2]{Bourgain10}. We begin from the bottom of the tree, setting $A^{(\ell)}:=A$. Once $A^{(s+1)}$ is constructed, we let
\[
A^{(s,j)} =\bigcup \{ A^{(s+1)}\cap J : J\in\DD_{sD}(A^{(s+1)}), \NN(J\cap A^{(s+1)},(s+1)D) \in [2^j+1,2^{j+1}] \}.
\]
Since $j$ takes at most $D$ values, we can pick $j=j_s$ such that $|A^{(s,j)}|\ge |A^{(s+1)}|/D$. By removing at most half of the intervals in $A^{(s+1)}$ from each interval $J$ making up $A^{(s,j)}$, we obtain a set $A^{(s)}$ such that $|A^{(s)}|\ge |A^{(s+1)}|/(2D)$ and $\NN(J\cap A^{(s)},(s+1)D)=2^j$ for all $J\in \DD_{sD}(A^{(s)})$. We see inductively that $\NN(J\cap A^{(s)}, (s'+1) D)$ is constant over all $J\in\DD_{s' D}(A^{(s)})$, for all $s'=s,s+1,\ldots,\ell-1$.

The lemma follows by taking $A'=A^{(0)}$.
\end{proof}

The next simple lemma (which is also implicit in \cite{Bourgain10}) asserts that, given a $(D,\ell,R)$-uniform set, it is possible to reduce some of the numbers $R_s$ to $1$ without decreasing the size of the set too much.
\begin{lemma} \label{lem:collapsing}
Given $D,\ell\in\N$, the following holds. Suppose $A$ is $(D,\ell,R)$-uniform. Then, if $\mathcal{S}\subset[\ell]$ is any set, there exists a subset $A'\subset A$ which is $(D,\ell,R')$ uniform, where $R'_s=1$ for $s\in\mathcal{S}$ and $R'_s=R_s$ for $s\in [\ell]\setminus \mathcal{S}$, and
\[
|A'| \ge  \left(\prod_{s\in\mathcal{S}} \frac{1}{R_s}\right)|A| \ge 2^{-|\mathcal{S}|D}|A|.
\]
\end{lemma}
\begin{proof}
We inductively construct a sequence of sets $A^{(s)}$, $s\in [\ell]$. Set $A^{(0)}=A$. Once $A^{(s)}$ is defined, if $s\notin\mathcal{S}$ set $A^{(s+1)}=A^{(s)}$. Otherwise, for each $I\in\DD_{sD}(A^{(s)})$, let $J_I$ be any interval in $\DD_{(s+1)D}(A\cap I)$, and let $A^{(s+1)}$ be the union of all such intervals $J_I$. Since $R_s\le 2^{D}$, it is clear that $A'=A^{(\ell-1)}$ has the desired properties.
\end{proof}

Given a set $A$, the next lemma extracts a large subset $A'$ of a suitable translation of $A$, such that points in $A'$ are  ``not too close to the boundary'' of $2^D$-adic intervals.
\begin{lemma} \label{lem:centering}
Let $D\in\N_{\ge 2}$, $\ell\in\N$, and let $A$ be a $2^{-m}$-set in $\R/\Z$, where $m=\ell D$. Then there are a point $x=k 2^{-m}$,  and a subset $A'\subset A$ such that:
\begin{enumerate}[label={\upshape(\roman*).}]
\item $|A'|\ge  2^{-(\log 3/D) m}|A|$.
\item For all $y\in A'$ and all $s\in [\ell]$, $y+x\in\frac12 \DD_{sD}(y+x)$.
\end{enumerate}
\end{lemma}
\begin{proof}
We note the following simple fact: for any $y\in [0,1)$ and any $j\le m-2$,  there is $t\in\{-2^{-(j+2)},0,2^{-(j+2)}\}$ such that $y+t\subset \tfrac12 \DD_j(y)$. With this in mind, we prune the tree in a similar way to Lemma \ref{lem:regular-subset} to construct sets $A^{(s)}$, starting from $A^{(\ell)}=A$ and moving up to $A^{(0)}$, such that for each $s\in [\ell]$,
\begin{enumerate}
  \item There is $t_s\in\{ -2^{-(s D+2)},0,2^{-(s D+2)}\}$ such that $y+x_s:=y+\sum_{s'=s}^{\ell-1} t_{s'}\in\frac12\DD_{s D}(y+x_s)$, for all $y\in A^{(s)}$.
  \item Moreover, $|A^{(s)}|\ge |A^{(s+1)}|/3$.
\end{enumerate}
Set $x=x_0=\sum_{s=0}^{\ell-1} t_s$ and $A'=A^{(0)}$. It is clear that $|A'|\ge 3^{-\ell}|A|= 2^{-(\log 3/D) m}|A|$. Also, since $\sum_{s'=0}^{s-1} t_{s'}$ is a multiple of $2^{-sD}$, we have $y+x\in\DD_{s D}(y+x)$ for all $y\in A'$ and $s\in [\ell]$, as claimed.
\end{proof}

The next lemma will allow us to show that if $H$ has small doubling and $A+H$ is ``not too large'', then $A$ and $H$ have a certain shared structure.
\begin{lemma} \label{lem:large-sumset-regular-sets}
Let $D,\ell\in\N$, and write $m=\ell D$. Suppose $H,A$ are $2^{-m}$-sets such $H$ is $(D,\ell,R)$-uniform and $A$ is $(D,\ell,R')$-uniform. Then
\[
|A+H| \ge 2^{-(1/D)m} |H| \prod_{s:R_s=1} R'_s.
\]
\end{lemma}
\begin{proof}
Write $\mathcal{S}=\{ s:R_s=1\}$. By replacing $A$ with the subset given by Lemma \ref{lem:collapsing}, we may assume that $R'_s=1$ for all $s\notin\mathcal{S}$. This makes the problem symmetric: for each $s$, either $R_s=1$ or $R'_s=1$. With this in mind, we inductively show that for each $s=0,1,\ldots,\ell$, there are families $\mathcal{I}_s\subset \DD_{s D}(A)$, $\mathcal{J}_s\subset \DD_{s D}(H)$, such that:
\begin{enumerate}
\item $|\mathcal{I}_s||\mathcal{J}_s| \ge 2^{-s} \NN_{s D}(A)\NN_{s D}(H)$
\item The intervals $\{ I+J: I\in\mathcal{I}_s, J\in\mathcal{J}_s\}$ are pairwise disjoint.
\end{enumerate}
The base case $s=0$ is trivial. Suppose this holds for some $s<\ell$. Without loss of generality, $R_{s+1}=1$. Hence, for each $J\in\mathcal{J}_s$ we pick the single $J'\in\DD_{(s+1)D}(J\cap H)$ and let $\mathcal{J}_{s+1}$ be the union of all such $J'$. Next, for each $I\in\mathcal{I}_s$, let $(I'_j)_{j=1}^{N_I}$ be a subcollection of $\DD_{(s+1)D}(I\cap A)$ such that no two of the $I'_j$ are adjacent, and $N_I\ge \lceil R'_{s+1}/2\rceil$. We let $\mathcal{I}_{s+1}$ be the union of all $I'_j$ over all $I\in\mathcal{I}_s$. It is clear from this construction that (1)--(2) hold.

The claim follows from (1)--(2) applied with $s=\ell$.
\end{proof}

We conclude this section with a version of Theorem \ref{thm:Bourgain} in which the lengths of the intervals over which there is either no or close to full branching is kept constant (at the price of worsening the quantitative estimates). This reduction is a matter of simplicity; a version of Theorem \ref{thm:inverse-thm} in which the intervals of almost full/no branching have varying lengths could be deduced directly from Theorem \ref{thm:Bourgain}.

\begin{cor} \label{cor:Bourgain}
 Given a large $D\in\N$,   the following holds for sufficiently large $\ell\in \N$ (depending on $D$).

Let $m=\ell D$. Suppose $H$ is an $(m,2^{-2D^2-1})$-small doubling set. Then there is a subset $H_1\subset H$ such that the following holds:
 \begin{enumerate}[label={\upshape(\roman*).}]
  \item $|H_1|\ge 2^{-(4 (\log D)D^{-1/4} )m} |H|$.
  \item $H_1$ is $(D,\ell,R)$-uniform, where for each $u$ either $R_u=1$, or $\log R_u\ge (1-D^{-1/4})D$.
  \end{enumerate}
 In particular, $|H_1|=\prod_{u:R_u>1} R_u$.
\end{cor}
\begin{proof}
Let $H',\mathcal{S}$, $n_s$, $\widetilde{R}_s$ (in place of $R_s$) be as given by Theorem \ref{thm:Bourgain} with $T=D^2$. We assume that $m$ is a multiple of $T$; the general case can be deduced by applying this special case to $\max\{m_1 T: m_1 T\le m\}$.

Let $S=\sum_{s\in\mathcal{S}} n_s - s T$. If $S<T^{-1/4}m$, then
\[
|H| \le 2^{(2\log T/\sqrt{T}) m} |H'|  \le 2^{(2\log T/\sqrt{T}) m} 2^{T^{-1/4}m} \le 2^{2T^{-1/4} m}.
\]
so that a singleton satisfies the conditions in the statement. We therefore assume that $S\ge T^{-1/4}m$.

We apply Lemma \ref{lem:regular-subset} to $H'$ and $D$, to obtain a $(D,\ell,R)$-uniform set $H''$ such that
\begin{equation} \label{eq:H''-lower-bound}
|H''| \ge  2^{-(\log(2D)/D) m}|H'|,
\end{equation}
It is clear that $R_u=1$ for all $u$ of the form $sD+j, j\in[D]$, with $s\notin\mathcal{S}$, and also with $s\in\mathcal{S}$ and $jD \ge n_s-s T$, since over those scales already $H'$ had no branching.  Therefore, there is a set $\mathcal{U}$ such that $R_u=1$ for $u\notin\mathcal{U}$, and
\begin{equation} \label{eq:U-lower-bound}
 D|\mathcal{U}| \le S + (m/T)D  \le S(1+D^{-1/2}),
\end{equation}
using that $S\ge D^{-1/2}m$. Using Theorem \ref{thm:Bourgain}, \eqref{eq:H''-lower-bound} and $S\ge D^{-1/2}m$ again, we get
\begin{equation*} %\label{eq:U-lower-bound-2}
 (1-1/D)S \le \log|H'| \le \log|H''|+ \frac{\log(2D)}{D} D^{1/2} S,
\end{equation*}
so that, recalling \eqref{eq:U-lower-bound},
\[
\log|H''| \ge S(1-1/D-\log(2D) D^{-1/2}) \ge \frac{1-(2\log D) D^{-1/2}}{1+D^{-1/2}} D|\mathcal{U}|.
\]
Hence,
\[
\frac{1}{|\mathcal{U}|}\sum_{u\in\mathcal{U}} \frac{\log R_u}{D} = \frac{\log |H''|}{D|\mathcal{U}|} \ge 1-3(\log D)D^{-1/2}.
\]
Since $\log R_u/D\in [0,1]$ for all $u$, Markov's inequality yields that $\log(R_u) \ge (1- D^{-1/4})|D|$ for $u$ outside of a set $\mathcal{U}'$ with
\[
|\mathcal{U}'|\le 3(\log D)D^{-1/4}|\mathcal{U}| \le 3(\log D)D^{-1/4}m/D,
\]
provided $D$ is larger than an absolute constant. To obtain our final set $H_1$, we apply Lemma \ref{lem:collapsing} to $H''$ and the set $\mathcal{U}'$ (that is, we collapse all $R_u$ intervals to a single one for $u\in\mathcal{U}'$). Recalling Theorem \ref{thm:Bourgain}(i) and \eqref{eq:H''-lower-bound}, the resulting set satisfies
\[
|H_1| \ge 2^{-D|\mathcal{U}'|}|H''|   \ge 2^{-4(\log D)D^{-1/4} m}|H|,
\]
while the claim on the branching structure is clear from the construction.
\end{proof}

\subsection{Proof of Theorem \ref{thm:inverse-thm}}

\begin{proof}[Proof of Theorem \ref{thm:inverse-thm}]
 Let $D\in\N,\e>0$. In the course of the proof, we will impose several lower bounds to $D$ (depending on $D_0,\delta,q$ only) and upper bounds on $\e$ (depending on $D,\delta,q$ only), resulting in the verification of all the claims in the theorem. To begin, we assume $D\ge D_0$. In the course of the proof, we write $m=D\ell$, and understand $\ell$ and $m$ to be sufficiently large that any claims involving them hold.

Let $\tau>0$ be the value given by Theorem \ref{thm:BSG} for $\kappa:=2^{-2D^2-1}$. We take
\[
\e \le \frac{\tau}{2\max(q,q')}.
\]
(Later we will impose further conditions on $\e$.)

Apply Lemma \ref{lem:non-flattening-to-subset} to obtain sets $A_1$, $B_1$ and $j,j'\le 2\e q' m$ satisfying (i)--(iii) in the lemma (with $A_1, B_1$ in place of $A,B$). By our choice of $\e$ and Lemma \ref{lem:q-to-2},
\[
\|\mathbf{1}_{A_1} * \mathbf{1}_{B_1}\|_2^2 \ge 2^{-\tau m}|A_1| |B_1|^2,
\]
so that we can apply Theorem \ref{thm:BSG} to $A_1,B_1$ to  obtain an $(m,\kappa)$-small doubling set $H$ and a $2^{-m}$-set $X$ such that
\begin{align}
|A_1\cap (X+H)| & \ge 2^{-\kappa m}|A_1|, \label{eq:BSG-i} \\
|A_1| &\ge 2^{-\kappa m}|X||H|,  \label{eq:BSG-ii} \\
|B_1\cap H| & \ge 2^{-\kappa m}|B_1|. \label{eq:BSG-iii}
\end{align}
Thanks to Lemma \ref{lem:non-flattening-to-subset}, the sets $A_1, B_1$ already satisfy (A-ii), (B-ii). As the final sets $A,B$ will be subsets of $A_1, B_1$, these properties are established.

Our next step is to pass to suitable regular subsets of (a translation of) $A_1, B_1, H$:
\begin{enumerate}
\item  By our choice $\kappa=2^{-2 D^2-1}$, we can apply Corollary \ref{cor:Bourgain} to $H$. Let $H'\subset H$ be the resulting set, with branching numbers $R_s, s\in [\ell]$.
\item We first apply Lemma \ref{lem:centering} (this is the point where we need to translate the original measure), and then  Lemma \ref{lem:regular-subset} and \eqref{eq:BSG-i} , to the set $A_1\cap (X+H)$, to obtain a set $A\subset A_1\cap (X+H)$ such that:
    \begin{enumerate}
    \item $|A|\ge 2^{- (2\log D/D) m}|A_1\cap (X+H)|\ge 2^{-(3\log D/D)m}|A_1|$. Hence, in light of (A-ii), property (A-i) holds if $D$ is taken large enough in terms of $\delta$.
    \item The set $A$ is $(\ell,D,R')$-uniform for some sequence $(R'_s)_{s\in[\ell]}$. This shows that (A-iii) holds.
    \item $x\in \frac12 \DD_{s D}(x)$ for all $x\in A$ and $s\in [\ell]$. That is, (A-iv) holds.
    \end{enumerate}
\item Similarly, we apply Lemma \ref{lem:centering}, and then Lemma \ref{lem:regular-subset} and \eqref{eq:BSG-iii}, to $B_1\cap H$ to obtain a set $B_2\subset B_1\cap H$ (not yet our final set $B$) such that:
    \begin{enumerate}
    \item $|B_2|\ge 2^{- (2\log D/D) m}|B_1\cap H|\ge 2^{-(3\log D/D)m}|B_1|$.
    \item The set $B_2$ is $(\ell,D,\wt{R})$-uniform for some sequence $(\wt{R}_s)_{s\in[\ell]}$. This shows that (B-iii) holds for $B_2$.
    \item $y\in \frac12 \DD_{s D}(y)$ for all $y\in B_2$. As the final set $B$ will be a subset of $B_2$, this establishes (B-iv).
    \end{enumerate}
\end{enumerate}

%The strategy to establish (v) will be as follows: if we had $H'=H, A=X+H$ and $B\subset H$, then (v) would be immediate, since $A$ would have almost full branching at all scales for which $H$ does, while $B$ would have no branching at any scales for which the same holds for $H$. These three facts hold, in some sense, approximately, so that the same conclusion can still be reached, after discarding a negligible set of scales.

Next, we note that as $A+H'\subset X+H+H$, we can use \eqref{eq:BSG-ii} and (2)(a) above to estimate
\begin{equation} \label{eq:A+H'-small}
|A+H' | \le |X||H+H| \le 2^{\kappa m}|X||H| \le 2^{2\kappa m}|A_1| \le 2^{(4\log D/D)m}|A|.
\end{equation}
Let $\mathcal{S}_0 = \{ s\in[\ell]: R_s=1\}$, $\mathcal{S}_1=[\ell]\setminus\mathcal{S}_0$, so that $\mathcal{S}_1$ indexes the scales over which $H'$ has almost full branching. We will see that $A$ has almost full branching for a large subset of scales $\mathcal{S}\subset\mathcal{S}_1$; eventually $B$ will be obtained from $B_2$ by collapsing all the branching at the scales in $[\ell]\setminus\mathcal{S}$ using Lemma \ref{lem:collapsing}.

According to Lemma \ref{lem:large-sumset-regular-sets} applied to $A$ and $H'$ (which we have seen meet the hypotheses),
\begin{equation} \label{eq:A+H'-large}
|A+H'| \ge 2^{-(1/D) m} |H'| \prod_{s\in\mathcal{S}_0} R'_s.
\end{equation}
Since $|A|=\prod_s R'_s$, $|H'|=\prod_s R_s$ and $R_s\ge 2^{(1-D^{-1/4})D}$ for $s\in\mathcal{S}_1$, we may combine \eqref{eq:A+H'-small}  and \eqref{eq:A+H'-large} to deduce that
\[
\prod_{s\in\mathcal{S}_1} R'_s = \frac{|A|}{\prod_{s\in\mathcal{S}_0} R'_s} \ge 2^{-(5\log D/D)m}|H'| \ge 2^{-(5\log D/D)m} 2^{(1-D^{-1/4})|\mathcal{S}_1|D}.
\]
Consider two cases.
\begin{enumerate}
 \item If $|\mathcal{S}_1|< D^{-1/2}\ell$ (which we note implies $H'$, hence $H$ and $B$, are very small) we set $\mathcal{S}' = \mathcal{S}_1$ and $\mathcal{S}=\varnothing$.
 \item If $|\mathcal{S}_1|\ge D^{-1/2}\ell$, then we further deduce from the above that
\[
\prod_{s\in\mathcal{S}_1} R'_s \ge 2^{(1-2 D^{-1/4})D|\mathcal{S}_1|}.
\]
Let
\begin{align*}
\mathcal{S}&=\{ s\in\mathcal{S}_1: R'_s \ge 2^{(1-D^{-1/8})D}\},\\
\mathcal{S}'&=\{ s\in\mathcal{S}_1: R'_s < 2^{(1-D^{-1/8})D}\}.
\end{align*}
Since $R'_s\le 2^D$ for all $s$, we have
\[
(1-2 D^{-1/4})|\mathcal{S}_1| \le \sum_{s\in\mathcal{S}_1} \frac{\log R'_s}{D} \le (1-D^{-1/8})|\mathcal{S}'| +  |\mathcal{S}_1|-|\mathcal{S}'|,
\]
so that
\[
|\mathcal{S}'| \le  2D^{-1/8} |\mathcal{S}_1| \le 2 D^{-1/8} \ell.
\]
\end{enumerate}
We note for later reference that, in either case
\begin{equation} \label{eq:scales-prunned}
|\mathcal{S}'| \le \max(2 D^{-1/8}\ell,D^{-1/2}\ell) = 2 D^{-1/8}\ell.
\end{equation}

We move on to the construction of $B$.  By Theorem \ref{thm:BSG} and Corollary \ref{cor:Bourgain},
\[
|B_2+H'| \le |(B_1\cap H)+H| \le 2^{\kappa m}|H| \le 2^{(5(\log D) D^{-1/4}) m}|H'|.
\]
Applying Lemma \ref{lem:large-sumset-regular-sets} to $B_2$ and $H'$, we deduce that
\[
\prod_{s\in\mathcal{S}_0} \wt{R}_s \le 2^{(1/D)m} 2^{(5(\log D) D^{-1/4}) m} \le 2^{(6(\log D) D^{-1/4})m}.
\]
We apply Lemma \ref{lem:collapsing} to $B_2$ and the set $\mathcal{S}_0$, to obtain a new set $B_3\subset B_2$ such that for all $s\in\mathcal{S}_0$ and $I\in\DD_{s D}(B_3)$, there is a single $J\in\DD_{(s+1)D}(B_3\cap I)$, while if $s\notin\mathcal{S}_0$, then $\NN(B_3\cap I,(s+1)D)=\wt{R}_s$ for all $I\in\DD_{s D}(B_3)$.  By Lemma \ref{lem:collapsing} and (3)(a) above,
\[
|B_3| \ge  2^{-(6(\log D)D^{-1/4})m}|B_2| \ge 2^{-(7(\log D)D^{-1/4})m} |B_1|.
\]
Finally, recall that we defined a set $\mathcal{S}'$, satisfying \eqref{eq:scales-prunned}.  We obtain our final set $B$ by applying Lemma \ref{lem:collapsing} to $B_3$ and $\mathcal{S}'$. Then
\[
|B| \ge 2^{-2 D^{-1/8} m} |B_3| \ge 2^{- 3 D^{-1/8} m} |B_1|,
\]
and $\NN_{(s+1)D}(I\cap B)=1$ for all $I=\DD_{s D}(B)$ for each $s\in\mathcal{S}_0\cup \mathcal{S}'=[\ell]\setminus\mathcal{S}$. We had already established (B-ii) and (B-iv). The set $B$ satisfies (B-i) if $D$ is large enough (thanks to (B-ii)); and it still satisfies (B-iii), with $R''_s=1$ for $s\notin\mathcal{S}$ and $R''_s=\wt{R}_s$ for $s\in\mathcal{S}$.

The claim (v) follows from the construction if $D$ is large enough: either $s\in\mathcal{S}$, in which case $R'_s\ge 2^{(1-o_{D\to\infty}(1))(D)}$ or $s\notin \mathcal{S}$, in which case $R''_s=1$ as we have just observed.

It remains to establish (vi). It follows from (B-i)--(B-ii)  that $\nu(x) \ge \frac12 2^{-\delta m}|B|^{-1}$ for all $x\in B$. On the other hand, we know from (B-iii) and (v) that $|B|\le 2^{D|\mathcal{S}|}$. We get
\[
\|\nu\|_q^{-q'} \le O_q(1) 2^{\delta m q'} |B| \le O_q(1) 2^{\delta m q'} 2^{D|\mathcal{S}|},
\]
which gives the left-hand inequality in (vi), with $O_q(\delta)$ in place of $\delta$.

By Lemma \ref{lem:non-flattening-to-subset}, $\mu(x) \ge \tfrac12  2^{-2\e q' m}\|\mu\|_q^{q'}$ for all $x\in A\subset A_1$, whence
\[
2^{-q} 2^{-(2qq'\e) m}\|\mu\|_q^{q q'} |A| \le \|\mu\|_q^q,
\]
so that $|A|\le  2^{3 q q'\e m} \|\mu\|_q^{-q'}$. Since $|A|\ge 2^{(1-\delta)D|\mathcal{S}|}$ by (A-iii) and (v), the right-hand side inequality in (vi) also follows (with $2\delta$ in place of $\delta$, say), concluding the proof.
\end{proof}

\section{Properties of dynamically driven self-similar measures}
\label{sec:properties-dyn-ssm}

\subsection{Preliminary lemmas}

In this section we initiate the study of measures generated by pleasant models (recall Definition \ref{def:pleasant}). We start by collecting some standard lemmas for later reference. The short proofs are included for completeness.

\begin{lemma} \label{lem:Holder}
Let $(Y,\mu,\mathcal{B})$ be a probability space. Suppose $\mathcal{P},\mathcal{Q}$ are finite families of measurable subsets of $Y$ such that each element of $\mathcal{P}$ can be covered by at most $M$ elements of $\mathcal{Q}$ and each element of $\mathcal{Q}$ intersects at most $M$ elements of $\mathcal{P}$. Then, for every $q\ge 1$,
\[
\sum_{P\in\mathcal{P}} \mu(P)^q  \le M^q \sum_{Q\in\mathcal{Q}} \mu(Q)^{q}
\]
\end{lemma}
\begin{proof}
Let $Q_{P,1},\ldots, Q_{P,M_P}$, $M_P\le M$, be a minimal sub-collection of $\mathcal{Q}$ that covers $P\in\mathcal{P}$. Using H\"{o}lder's inequality in the form $(\sum_{i=1}^m a_i)^q \le m^{q-1} \sum_{i=1}^m a_i^q$, we get
\[
\sum_{P\in\mathcal{P}} \mu(P)^q \le M^{q-1} \sum_{P\in\mathcal{P}} \sum_{i=1}^{M_P}\mu(Q_{P_i})^q \le M^q \sum_{Q\in\mathcal{Q}} \mu(Q)^q.
\]
\end{proof}

\begin{lemma} \label{lem:L-q-norm-almost-disj-supports}
Let $\mu=\sum_{i=1}^\ell \mu_i$, where $\mu_i$ are finitely supported measures on a space $Y$, such that each point is in the support of at most $M$ of the $\mu_i$. Then
\[
\|\mu\|_q^q \le M^{q-1} \sum_{i=1}^\ell \|\mu_i\|_q^q.
\]
\end{lemma}
\begin{proof}
For each $x$, H\"{o}lder's inequality, together with the assumption that $\mu_i(x)>0$ for at most $M$ values of $i$, gives $(\sum_i \mu_i(x))^q \le M^{q-1} \sum_i \mu_i(x)^q$. The claim follows.
\end{proof}

\begin{lemma} \label{lem:discr-norm-conv-equivalence}
For any probability measures $\mu,\nu$ on $\R/\Z$, and any $q\in (1,\infty)$,
\[
\|(\mu*\nu)^{(m)}\|_q^q = \Theta_q(1) \|\mu^{(m)}*\nu^{(m)}\|_q^q.
\]
\end{lemma}
\begin{proof}
Given $I=[k 2^{-m},(k+1)2^{-m})\in\DD_m$, let $P_I=\{ (x,y)\in(\R/\Z)^2: x+y\in I\}$ and
\[
Q_I = \bigcup_{ i\in\Z/2^m\Z } [i 2^{-m},(i+1) 2^{-m}) \times [(k-i) 2^{-m},(k-i+1) 2^{-m}).
\]
Then $\|(\mu*\nu)^{(m)}\|_q^q=\sum_I (\mu\times\nu)(P_I)^q$ and  $\|\mu^{(m)}*\nu^{(m)}\|_q^q=\sum_I (\mu\times\nu)(Q_I)^q$, so the claim follows from Lemma \ref{lem:Holder}.
\end{proof}

\subsection{A sub-multiplicative cocycle, and consequences}

Throughout the rest of this section, we use the following notation. We work with a measure-preserving system $(X,\mathbf{T},\mathbb{P})$, i.e. $\mathbf{T}:X\to X$ is a measurable map, and $\mathbf{T}\mathbb{P}=\mathbb{P}$. A model $\mathcal{X}=(X,\mathbf{T},\Delta,\lam)$ is fixed, and $\mu_x, \mu_{x,n}$ are as defined in \eqref{eq:def-mu-x}, \eqref{eq:def-mu-x-n}. Moreover, $m=m(n)$ will denote the smallest integer such that $2^{-m}\le \lam^n$ (the dependence is omitted when it is clear from context). We assume that
\begin{equation} \label{eq:assumption-support-mu-x}
\supp(\mu_x)\subset [0,1] \text{ for all }x\in X,
\end{equation}
which can always be achieved by a change of coordinates, i.e. by replacing the map $\Delta$ by $g\circ\Delta$ for an appropriate affine map $g$.

For each $x\in X$ we define a \emph{code space} $\Omega_x =  \prod_{n=0}^\infty \supp(\Delta(\mathbf{T}^n x))$ and a \emph{coding map} $\pi_x:\Omega_x \to \R$, via $\omega\mapsto \sum_{n=0}^\infty \omega_n \lam^n$. Then, by definition, $\mu_x$ is the push-down of the product measure $\prod_n \Delta(\mathbf{T}^n x)$ under this coding map. We also define the \emph{truncated coding maps} $\pi_{x,n}: \Omega_x \to \R$, $\omega\mapsto \sum_{i=0}^{n-1} \omega_i \lam^i$. Then $\mu_{x,n}$ is the image of $\prod_n \Delta(\mathbf{T}^n x)$ under the truncated coding map.

\begin{lemma} \label{lem:comparison-mu-m-mu}
For every $x\in X$, $\|\mu_x^{(m)}\|_q^q = \Theta_{\lam,q}(1) \|\mu_{x,n}^{(m)}\|_q^q$.
\end{lemma}
\begin{proof}
Let $\eta=\prod_{n=0}^\infty \Delta(\mathbf{T}^n x)$, so that $\mu_x=\pi_x\eta$ and $\mu_{n,x}=\pi_{n,x}\eta$. Then
\begin{align*}
\|\mu_x^{(m)}\|_q^q  &= \sum_{I\in\DD_m} \eta(\pi_x^{-1}I)^q,\\
\|\mu_{x,n}^{(m)}\|_q^q &= \sum_{I\in\DD_m} \eta(\pi_{x,n}^{-1}I)^q.
\end{align*}
Since $\|\pi_x-\pi_{x,n}\|_\infty \le O(\lam^n) = O_\lam(2^m)$, the lemma follows easily from Lemma \ref{lem:Holder}.
\end{proof}

We recall some well-known properties of the $L^q$ spectrum $\tau_\mu$. See e.g. \cite[Proposition 3.2]{LauNgai99} for the proofs.
\begin{lemma} \label{lem:basic-tau-properties}
For any probability measure on $\R$ of bounded support, the function $\tau=\tau_\mu:[0,\infty)\to \R$ is increasing, concave, and satisfies $\tau(1)=0$.
\end{lemma}

The next proposition introduces a sub-multiplicative cocycle (which was first used in \cite{NPS12}, in a special case) that will play a crucial r\^{o}le in the proof of Theorem \ref{thm:L-q-dim-dyn-ssm}. Let us define the following sequence of functions, parametrized by $q\in [1,\infty)$:
\[
\phi_n^q(x) = \| \mu_x^{(m(n))} \|_q^q.
\]
\begin{prop} \label{prop:cocycle}
For any $n,n'\in\N$,
\[
\phi_{n+n'}^q(x) \le O_{q,\lam}(1)  \phi_n^q(x) \phi_{n'}^q(\mathbf{T}^n x).
\]
In particular, for each $q\in [1,\infty)$ there exists a number $T(q)$ such that
\begin{equation} \label{eq:T-q-for-almost-all-x}
\lim_{n\to\infty} -\frac{1}{m} \log\|\mu_{x}^{(m)}\|_q^q = T(q)
\end{equation}
for $\mathbb{P}$-a.e. $x$. Moreover, for $\mathbb{P}$-a.e. $x$ it holds that $T(q) = \tau_{\mu_x}(q)$ for all $q\in [1,+\infty)$.  In particular, $T:[1,\infty)\to\R$ is increasing and concave, and $T(1)=0$.
\end{prop}
\begin{proof}
We estimate:
\begin{align*}
\|\mu_{x}^{(m(n+n'))}\|_q^q &\le O_{\lam,q}(1) \| \mu_{x,n}^{(m(n+n'))} *  \left(S_{\lam^n} \mu_{\mathbf{T}^n x}\right)^{(m(n+n'))}\|_q^q \\
&\le O_{\lam,q}(1) \sum_{I\in\DD_{m(n)}} \|\mu_{x,n}^{(m(n+n'))}|_I *  \left(S_{\lam^n} \mu_{\mathbf{T}^n x}\right)^{(m(n+n'))}\|_q^q\\
&\le O_{\lam,q}(1) \sum_{I\in\DD_{m(n)}} \mu_{x,n}(I)^q \sum_{J\in\DD_{m(n+n')}} \mu_{\mathbf{T}^n x}(S_{\lam^{-n}}J)^q \\
&\le O_{\lam,q}(1) \|\mu_{x}^{(m(n))}\|_q^q \|\mu_{\mathbf{T}^n x}^{(m(n'))}\|_q^q.
\end{align*}
We have used the self-similarity relation \eqref{eq:mu-x-self-similarity} and Lemma \ref{lem:discr-norm-conv-equivalence} in the first line, Lemma \ref{lem:L-q-norm-almost-disj-supports} in the second line (which is justified since the support of $S_{\lam^n} \mu_{\mathbf{T}^n x}$ has diameter $O_\lam(2^{-m(n)})$), Young's inequality in the third line, and Lemmas \ref{lem:Holder}, \ref{lem:comparison-mu-m-mu}  in the last line.

The subadditive ergodic theorem applied to the sequence of (bounded and measurable) functions $x\mapsto \log \phi_n^q(x) -  C_{\lam,q}$ for a sufficiently large constant $C_{\lam,q}$ yields \eqref{eq:T-q-for-almost-all-x}. More precisely, we know the convergence for the subsequence $m(n)$, $n\in\mathbb{N}$, but since this sequence has positive density,  \eqref{eq:T-q-for-almost-all-x} follows from the monotonicity of $m\mapsto \|\nu^{(m)}\|_q^q$.

Finally, if $(q_j)$ is a dense subset of $(1,\infty)$, then we know from the previous claim that $\tau_{\mu_x}(q_j)=T(q_j)$ for all $j$, for $\mathbb{P}$-almost all $x$. Since $\tau_\mu$ is concave and increasing, and $T(q)$ is clearly increasing, we deduce that the equality extends to all $q\in (1,\infty)$.

The last claim is immediate from Lemma \ref{lem:basic-tau-properties}
\end{proof}

In order to prove Theorem \ref{thm:L-q-dim-dyn-ssm}, we would like to draw conclusions for \emph{all} $x$ rather than almost all. Indeed, the strategy will be to prove that the convergence in \eqref{eq:T-q-for-almost-all-x} holds for all $x$, and $T(q)$ has the ``expected'' value. It is well known that for uniquely ergodic systems, the ergodic averages of sufficiently regular (a.e. continuous) observables converge uniformly. The next known lemma asserts that a one-sided version of this remains valid for subadditive cocycles.
\begin{lemma} \label{lem:uniform-convergence}
Let $(X,\mathbf{T},\mathbb{P})$ be a uniquely ergodic measure-preserving system, with $X$ a compact metric space, and $\mathbf{T}$ continuous. Suppose $\phi_n:X\to\R$ are continuous $\PP$-almost everywhere and bounded, and
\[
\phi_{n+n'}(x) \le \phi_n(x) + \phi_{n'}(\mathbf{T}^n x)
\]
for all $n,n'\in\N$, $x\in X$. Then, denoting by $L$ the $\mathbb{P}$-almost sure limit of $\phi_n(x)/n$, we have
\begin{equation} \label{eq:Furman-uniform}
\limsup_{n\to\infty} \frac1n \phi_n(x) \le L \quad\text{uniformly in }x\in X.
\end{equation}
\end{lemma}
\begin{proof}
For continuous $\phi_n$, the claim was observed by Furman \cite[Theorem 1]{Furman97}. In the case the $\phi_n$ are only a.e. continuous and bounded, a classical exercise in measure theory yields that for each $n$ and $\e>0$ there exists a continuous function $\phi_{n,\e}$ such that $\phi_n\le \phi_{n,\e}$ pointwise, and $\int \phi_{n,\e}-\phi_n \,d\PP\le \e$. Indeed, using compactness and the fact that $\PP$ is a Radon measure, we may find a finite open cover $(\cup_i B_i) \cup B'$ of $X$ such that the variation of $\phi_n$ on each $B_i$ is at most $\e/2$, the discontinuity set of $\phi_n$ is contained in $B'$, and $\PP(B')< \e/(2\|\phi_n\|_\infty)$. Let $(h_i),h'$ be a continuous partition of unity subordinated to $B_i, B'$, and define $\phi_{n,\e}=\sum_i h_i \|\phi_n|_{B_i}\|_\infty + h' \|\phi_n\|_\infty$.

Other than the uniformity in $x$, the claim \eqref{eq:Furman-uniform} follows from \cite[Theorem 3.5]{GSSY16}, which in turn is established by inspecting the proof of the subadditive ergodic theorem given by Katznelson and Weiss \cite{KatznelsonWeiss82} (recall that for uniquely ergodic systems all points are generic). To deduce the uniform convergence, we recall that the ergodic averages of the continuous functions $\phi_{n,\e}$ converge uniformly (thanks to unique ergodicity), and apply \cite[Eq.(18)]{GSSY16}.
\end{proof}

Furman \cite[Theorem 1]{Furman97} also showed that, even in the continuous case, the set of $x$ such that $f_n(x)/n\not\to L$ may be nonempty and, indeed, can equal any $F_\sigma$, $\mathbb{P}$-null set.

From Proposition \ref{prop:cocycle} and Lemma \ref{lem:uniform-convergence} we obtain the following crucial corollary; this is the main place where the pleasantness of the model gets used.
\begin{cor} \label{cor:unif-continuity}
Suppose $(X,\mathbf{T},\Delta,\lam)$ is a pleasant model. Then
\[
\liminf_{m\to\infty} -\frac1m \log\|\mu_x^{(m)}\|_q^q \ge T(q) \quad\text{uniformly in }x\in X,
\]
where $T(q)$ is the function from Proposition \ref{prop:cocycle}.
\end{cor}
\begin{proof}
Let $\psi_m:\R/\Z\to [0,1]$ be a continuous bump function supported on the interval $[-2^{-m},2^{-m}]$ such that $\psi_m\equiv 1$ on $[-2^{-m}/2,2^{-m}/2]$. It follows easily from Lemma \ref{lem:Holder} that
\[
\Psi_m(x):= \sum_{k=0}^{2^m-1} \left(\int \psi_m(t+k 2^{-m}) d \mu_x(t)\right)^q = \Theta_q(\|\mu_x^{(m)}\|_q^q).
\]
Since the model is pleasant, $\Psi_m$ is bounded and continuous $\PP$-a.e. The corollary is now immediate from (the proof of) Proposition \ref{prop:cocycle} and Lemma \ref{lem:uniform-convergence}.
\end{proof}
We point out that, in the special case given by Lemma \ref{lem:mu-x-conv-ssm} below, this corollary was first obtained in \cite{NPS12}.

\subsection{Multifractal structure}
%\label{subsec:Lq-spectrum}

Next, we investigate the scaling (or multifractal) properties of measures generated by pleasant models. Throughout the rest of this section, we always assume the following:

\textbf{Standing assumption}. $(X,\mathbf{T},\Delta,\lam)$ is a pleasant model,  $T(q)$ is the function given by Proposition \ref{prop:cocycle} for this model. Any constants or parameters are allowed to depend on the model (in particular, on the function $T$).

Later on, in \S\ref{subsec:general-ssm}, we will need small variants of the results of this section in which $T$ is replaced by the $L^q$ spectrum of a fixed (non-homogeneous) self-similar measure. With a view towards this, it may be useful to observe that the proofs only use the concavity of $T$ together with Corollary \ref{cor:unif-continuity}.

We will establish some regularity of the multifractal structure for those values of $q$ such that $T$ is differentiable at $q$. The Legendre transform plays a key role in multifractal analysis. Given a concave function $\tau:\R\to\R$, its Legendre transform $\tau^*:\R\to [-\infty,\infty)$ is defined as
\[
\tau^*(\alpha) = \inf_{q\in\R} \alpha q -\tau(q).
\]
It is easy to check that if $\tau$ is concave and is differentiable at $q$, then
\[
\tau^*(\alpha)=\alpha q-\tau(q) \text{ for } \alpha=\tau'(q).
\]

The next lemma is also well known; the short proof is included for completeness.
\begin{lemma} \label{lem:f-alpha-smaller-than-one}
If $T$ is differentiable at $q>1$, $T(q)<q-1$, and $\alpha=T'(q)$, then $T^*(\alpha) \le \alpha < 1$
\end{lemma}
\begin{proof}
Since $T(1)=0$ and  $T(q)<q-1$, we have $(T(q)-T(1))/(q-1)<1$. On the other hand, as $T$ is concave and differentiable at $q$, we must have $\alpha\le (T(q)-T(1))/(q-1)<1$. Furthermore, $T^*(\alpha)\le \alpha\cdot 1-T(1)=\alpha$, so the lemma follows.
\end{proof}

It is known that the multifractal structure of general measures displays some regularity for values of $q$ such that $\tau_\mu$ is differentiable at $q$ (or, dually, values of $\alpha$ such that $\tau^*$ is strictly concave at $\alpha$); see for example \cite[Theorem 5.1]{LauNgai99}. The following lemmas, which are proved with similar ideas, are a further illustration of this. For a single measure $\mu$, the heuristic to keep in mind is that, whenever $\alpha=\tau_\mu'(q)$ exists, almost all of the contribution to $\|\mu^{(m)}\|_q^q$ comes from $\approx 2^{\tau^*(\alpha) m}$ intervals, each of mass $\approx 2^{-\alpha m}$. In our case, we are dealing with a family $(\mu_x)_{x\in X}$; with the help of Corollary \ref{cor:unif-continuity} we will establish results which are uniform in $x$, at the price of dealing with $T(q)$ in place of $\tau_{\mu_x}(q)$.

\begin{lemma} \label{lem:size-set-A-in-terms-of-f-alpha}
Suppose that $\alpha_0=T'(q_0)$ exists for some $q_0\in (1,\infty)$.

Given $\e>0$, the following holds if $\delta$ is small enough in terms of $\e, q_0$ and $m$ is large enough in terms of $\e, q_0$ and $\delta$.

Suppose $\mathcal{D}'\subset\mathcal{D}_m$ is such that, for some $x\in X$:
\begin{enumerate}
\item[\textup{(1)}]  $2^{-\alpha m}\le\mu_x(I)\le 2\cdot 2^{-\alpha m}$ for all $I\in\mathcal{D}'$ and some $\alpha\ge 0$.
\item[\textup{(2)}]  $\sum_{I\in\mathcal{D}'} \mu_x(I)^{q_0} \ge 2^{-(T(q_0)+\delta)m}$.
\end{enumerate}
 Then $|\mathcal{D}'|\le 2^{m(T^*(\alpha_0)+\e)}$.
\end{lemma}
\begin{proof}
Set $\eta:=\e/(3q_0)$, and pick $\delta\le \eta^2/9$, and also small enough that, if $q_1=q_0-\delta^{1/2}$, then
\begin{equation} \label{eq:using-tau-diff-1}
 T(q_0)-T(q_1) \le \delta^{1/2}\alpha_0 + \delta^{1/2}\eta .
\end{equation}
On one hand, using (1) and Corollary \ref{cor:unif-continuity}, we get
\begin{equation*} %\label{eq:bound-norm-mu-m-1}
2^{-(T(q_1)-\delta)m} \ge \|\mu_x^{(m)}\|_{q_1}^{q_1} \ge |\mathcal{D}'| 2^{- \alpha  q_1 m},
\end{equation*}
if $m$ is large enough (depending on $q_0, T$, but not on $x$). On the other hand, by the assumptions (1)--(2),
\[
|\mathcal{D}'| 2^{-\alpha q_0 m} \ge 2^{-q_0} 2^{(-T(q_0)-\delta)m} \ge 2^{(-T(q_0)-2\delta)m}
\]
if $m\gg_{\delta,q_0} 1$. Eliminating $|\mathcal{D}'|$ from the last two displayed equations yields
\[
\alpha q_0 -T(q_0) -2\delta \le \alpha (q_0-\delta^{1/2})-T(q_0-\delta^{1/2})+\delta,
\]
so that, recalling \eqref{eq:using-tau-diff-1},
\[
\delta^{1/2}\alpha \le T(q_0)-T(q_0-\delta^{1/2}) +3\delta \le \delta^{1/2}\alpha_0 +\delta^{1/2}\eta+3\delta.
\]
Hence $\alpha-\alpha_0< 2\eta$, since we assumed $\delta\le (\eta/3)^2$.  Using this, a further application of Corollary  \ref{cor:unif-continuity} guarantees that if $m\gg_\e 1$, then
\[
2^{(-T(q_0)+\e/3)m} \ge \|\mu_x^{(m)}\|_{q_0}^{q_0} \ge 2^{-q_0\alpha m} |\mathcal{D}'| \ge 2^{-q_0\alpha_0 m} 2^{-(q_0 2 \eta) m} |\mathcal{D}'|.
\]
The conclusion follows from the formula $T^*(\alpha_0)=q_0 \alpha_0-T(q_0)$ and our choice $\eta=\e/(3q_0)$.
\end{proof}

\begin{lemma}  \label{lem:Lq-sum-large-mass}

Let $q_0>0$ be such that $\alpha_0=T'(q_0)$ exists. Given $\sigma>0$, there is $\e=\e(\sigma,q_0)>0$ such that the following holds for large enough $m$ (in terms of $\sigma, q_0$): for all $x\in X$,
\begin{equation} \label{eq:Lq-sum-large-mass}
\sum \{ \mu_x(I)^{q_0}: I\in\DD_m, \mu_x(I) \ge 2^{-m(\alpha_0-\sigma)} \} \le 2^{-m(T(q_0)+\e)}.
\end{equation}
\end{lemma}
\begin{proof}
Let $\eta\in (0,1)$ be small enough that
\begin{equation}   \label{eq:using-tau-diff-2}
T(q_0+\eta) \ge T(q_0) + \eta \alpha_0 - \delta,
\end{equation}
where $\delta= \eta\sigma/(4+2 q_0)$.

Let $\alpha_j = \alpha_0 -\delta j$, and write $N_x(\alpha_j,m)$ for the number of intervals $I$ in $\DD_m$ such that $2^{-m\alpha_j}\le\mu_x(I)< 2^{-m\alpha_{j+1}}$. By Corollary \ref{cor:unif-continuity}, for any fixed value of $q$, if $m\gg_q 1$ then,
\begin{equation*}  %\label{eq:bound-norm-mu-m-2}
N_x(\alpha_j,m) 2^{-m q\alpha_j} \le \|\mu_x^{(m)}\|_q^q \le 2^{-m(T(q)-\delta)}.
\end{equation*}
Applying this to $q=q_0+\eta$, and using \eqref{eq:using-tau-diff-2}, we estimate
\begin{align*}
N_x(\alpha_j,m) 2^{-m q_0\alpha_j} &\le 2^{m \eta \alpha_j} 2^{-m(T(q_0+\eta)-\delta)} \\
&\le 2^{2\delta m} 2^{-j\delta\eta m} 2^{-T(q_0)m}.
\end{align*}
Let $S_x$ be the sum in the left-hand side of \eqref{eq:Lq-sum-large-mass} that we want to estimate. Using that $\delta=\eta\sigma/(4+2 q_0)$, we conclude that
\begin{align*}
S_x  &\le \sum_{j: \delta (j+1)\ge \sigma} N_x(\alpha_j,m) 2^{-m q_0\alpha_{j+1}} \\
&\le \sum_{j: \delta (j+1)\ge \sigma}  2^{\delta q_0 m}  2^{2\delta m} 2^{-j\delta\eta m} 2^{-T(q_0)m}\\
&\le \sum_{j\ge 0} 2^{-j\delta\eta m} 2^{(2+q_0)\delta m}  2^{-\eta\sigma m} 2^{-T(q_0)m} \\
&\le O_{\delta\eta}(1) 2^{(\eta\sigma/2-\eta\sigma)m} 2^{-T(q_0)m},
\end{align*}
as claimed.
\end{proof}

\begin{lemma} \label{lem:Lq-sum-over-small-set}
Let $q_0>1$ be such that $\alpha_0=T'(q_0)$ exists. Given $\kappa>0$, there is $\e=\e(\kappa,q_0)>0$ such that the following holds for large enough $m$ (in terms of $q_0,\e$) and all $x\in X$.

If $\mathcal{D}'\subset\DD_m$ has $\le 2^{(T^*(\alpha_0)-\kappa)m}$ elements, then
\[
\sum_{I\in\mathcal{D}'} \mu_x(I)^{q_0} \le 2^{-(T(q_0)+\e)m}
\]
for all $x\in X$.
\end{lemma}
\begin{proof}
Let $\sigma=\kappa/(2 q_0)$ and fix $x\in X$. In light of Lemma \ref{lem:Lq-sum-large-mass}, we only need to worry about those $I$ with $\mu_x(I)\le 2^{-m(\alpha_0-\sigma)}$. But
\begin{align*}
\sum\{ \mu_x(I)^{q_0}: I\in\mathcal{D}',\mu_x(I) \le 2^{-m(\alpha_0-\sigma)} \} &\le 2^{(T^*(\alpha_0)-\kappa)m} 2^{-(q_0\alpha_0-q_0\sigma)m} \\
&=   2^{-(\kappa-q_0\sigma)m}2^{-T(q_0)m}.
\end{align*}
By our choice of $\sigma$, $\kappa-q_0\sigma=\kappa/2>0$, so this gives the claim.
\end{proof}

%We turn to multifractal properties specific to self-similar measures. The following easy consequence of H\"{o}lder's inequality will be used repeatedly, %see e.g. \cite[Lemma 5.3]{ShmerkinSolomyak16} for the proof.
%\begin{lemma}  \label{lem:Holder}
%Let $\mu$ be a measure on a space $K$, and let $\mathfrak{P},\mathfrak{P}'$ be two measurable partitions of $K$ such that each element of $\mathfrak{P}$ intersects at most $M$ elements of $\mathfrak{P}'$, and viceversa. Then
%\[
%M^{1-q} \sum_{I\in\mathfrak{P}} \mu(I)^q \le \sum_{J\in\mathfrak{P}'} \mu(J)^q \le M^{q-1}   \sum_{I\in\mathfrak{P}} \mu(I)^q .
%\]
%\end{lemma}

The second part of the following proposition can be used to give another (though closely related) proof of Proposition \ref{prop:cocycle}, and was obtained in \cite{PeresSolomyak00, NPS12} in special cases. The first part is proved in a similar way, relying on Lemma \ref{lem:Lq-sum-over-small-set}.

\begin{prop}   \label{prop:Lq-over-small-set-is-small}
Let $q>1$ be such that $\alpha=T'(q)$ exists.
\begin{enumerate}[label={\upshape(\roman*).}]
\item Given $\kappa>0$, there is $\eta=\eta(\kappa,q)>0$ such that the following holds for all large enough $m$: for any $s\in\N$, $I\in\DD_s$ and $x\in X$, if $\mathcal{D}'$ is a collection of intervals in $\DD_{s+m}(I)$ with $|\mathcal{D}'|\le 2^{(T^*(\alpha)-\kappa)m}$, then
\[
\sum_{J\in\mathcal{D}'} \mu_x(J)^q  \le 2^{-(T(q)+\eta)m} \mu_{x}(2 I)^q.
\]
\item Given $\delta>0$, the following holds for all large enough $m$: for any $I\in\DD_s$, $s\in\N$, and $x\in X$,
\[
\sum_{J\in\DD_{s+m}(I)} \mu_x(J)^q \le 2^{-(T(q)-\delta)m} \mu_{x}(2 I)^q.
\]
\end{enumerate}
\end{prop}
\begin{proof}
We prove (i) first. Let $n$ be the smallest integer such that $\lam^n < 2^{-s-2}$. Let $y_j$ be the atoms of $\mu_{x,n}$ such that $[y_j,y_j+\lam^{n}]\cap I\neq \varnothing$, let $p_j$ be their respective masses, and write
\[
\mu_{x,n,I} = \sum_j p_j \delta_{y_j}.
\]
Then the support of $\mu_{x,n,I}$ is contained in the $\lam^n$-neighborhood of $I$. Moreover, since $\delta_z * S_{\lam^n}\mu_{T^n x}$ is supported on $[z,z+\lam^n]$, thanks to \eqref{eq:assumption-support-mu-x}, it follows from the self-similarity relation $\mu_x = \mu_{x,n}* S_{\lam^n}\mu_{\mathbf{T}^n x}$ and the definition of $\mu_{x,n,I}$ that  $\mu_x|_I =(\mu_{x,n,I}*S_{\lam^n}\mu_{\mathbf{T}^n x})|_I$. Write
\[
p = \|\mu_{x,n,I}\|_1 = \sum_j p_j \le \mu_x(2 I),
\]
using that, again by \eqref{eq:assumption-support-mu-x}, the support of $\mu_{x,n}$ is contained in the $\lam^n$ neighborhood of the support of $\mu_x$, and that $4 \lam^n \le 2^{-s}$.

We can then estimate
\begin{align*}
\sum_{J\in\DD'}  \mu_x(J)^q &= \sum_{J\in\DD'} \left( \sum_j p_j \delta_{y_j} * S_{\lam^n}\mu_{\mathbf{T}^n x}(J)  \right)^q\\
&= \sum_{J\in\DD'} \left( \sum_j p_j \mu_{\mathbf{T}^n x}(\lam^{-n}(J-y_j)) \right)^q\\
&\le \sum_{J\in\DD'} p^{q-1} \sum_j p_j \, \mu_{\mathbf{T}^n x}(\lam^{-n}(J-y_j))^q\\
&= p^{q-1}\sum_j p_j \sum_{J\in\DD'} \mu_{\mathbf{T}^n x}(\lam^{-n}(J-y_j))^q,
\end{align*}
where we used the convexity of $t^q$ in the third line. Now for each fixed $j$, each interval $\lam^{-n}(J-y_j)$ with $J\in\DD'$ can be covered by $O_{\lam}(1)$ intervals in $\DD_m$, and reciprocally each interval in $\DD_m$ hits at most $2$ intervals among the $\lam^{-n}(J-y_j)$. We deduce from Lemmas \ref{lem:Holder} and \ref{lem:Lq-sum-over-small-set} that, still for a fixed $j$,
\[
 \sum_{J\in\DD'} \mu_{\mathbf{T}^n x}(\lam^{-n}(J-y_j))^q \le O_{\lam,q}(1) 2^{-(T(q)+\e)m},
\]
provided $m$ is taken large enough, where $\e=\e(\kappa,q)>0$ is given by Lemma \ref{lem:Lq-sum-over-small-set}. Combining the last three displayed equations yields the first claim with $\eta=\e/2$.

The second claim follows in the same way, adding over $\DD_{s+m}(I)$ instead of $\DD'$, and using Corollary \ref{cor:unif-continuity} instead of Lemma \ref{lem:Lq-sum-over-small-set}.
\end{proof}

\section{Proof of Theorem \ref{thm:L-q-dim-dyn-ssm}}
\label{sec:proof-of-main-thm}

\subsection{Flattening of \texorpdfstring{$L^q$}{Lq} norm for dynamically driven self-similar measures}

As noted in the introduction, we aim to prove a generalization of  \cite[Theorem 1.1]{Hochman14}, by following the same broad outline. One of the key steps in the proof of  \cite[Theorem 1.1]{Hochman14} consists in showing that convolving a self-similar measure with an arbitrary measure, on which only a lower bound on the entropy is assumed, results in an entropy increment: see \cite[Corollary 5.5]{Hochman14}. In turn, this is derived from the inverse theorem of \cite{Hochman14} by proving that the entropy of self-similar measures is roughly constant at most scales and locations, a property that Hochman termed \emph{uniform entropy dimension}, see \cite[Definition 5.1 and Proposition 5.2]{Hochman14} for precise details. Once again, we will follow a different path to obtain a statement for $L^q$ norms which is similar in spirit.

We continue to work with a fixed pleasant model $(X,\mathbf{T},\Delta,\lam)$, and the function $T$ from Proposition \ref{prop:cocycle}.
\begin{thm} \label{thm:conv-with-ssm-flattens}
Given $\sigma>0$ and $q>1$ such that $T$ is differentiable at $q$ and $T(q)<q-1$, there is $\e=\e(\sigma,q)>0$ such that the following holds for $m$ large enough in terms of all previous parameters:

If $\nu$ is a $2^{-m}$-measure with $\|\nu\|_q^{q'}\le 2^{-\sigma m}$, and $x\in X$, then
\[
\| \nu * \mu_x^{(m)}\|_q^q \le 2^{-(T(q)+\e) m}.
\]
\end{thm}

The analogy with \cite[Corollary 5.5]{Hochman14} is clear. However, there is no useful analog of the notion of uniform entropy dimension for $L^q$ norms. One of the key differences is that nearly all of the $L^q$ norm may be (and often is) captured by sets of extremely small measure; while sets of small measure also have small entropy. Instead, we will use the regularity of the multifractal spectrum established in the previous section in the following manner: if the flattening claimed in the theorem does not hold, then the inverse theorem provides a regular set $A$ which captures much of the $L^q$ norm of $\mu_x$. The upper bound on $\|\nu\|_q^{q'}$, together with (v)--(vi) in the inverse theorem imply that $A$ has nearly full branching for a positive proportion of $2^{D}$-scales, so it must have substantially less than average branching also on a positive proportion of scales. On the other hand, we will call upon the lemmas from the previous section to show that, in fact, $A$ must have nearly constant branching on nearly all scales (this is the part that uses the differentiability of $T$ at $q$), obtaining the desired contradiction.

\begin{proof}[Proof of Theorem \ref{thm:conv-with-ssm-flattens}]
Suppose $\nu$ is a $2^{-m}$-measure with $\|\nu\|_q^{q'} \le 2^{-\sigma m}$. In the course of the proof, we will choose many numbers which ultimately depend on $\sigma$ and $q$ only. To ensure that there is no circularity in their definitions, we indicate their dependencies: $\alpha=\alpha(q)$, $\kappa=\kappa(\alpha,\sigma)$, $\gamma=\gamma(q,\alpha,\kappa)$, $\delta'=\delta'(\alpha,\sigma,\kappa)$,   $\eta=\eta(q,\kappa)$, $\delta=\delta(q,\delta',\gamma,\eta)$, $\xi=\xi(q,\delta',\eta,\gamma)$,  $D_0=D_0(q,\sigma,\delta)$, $D=D(q,\delta,D_0)$, $\e=\e(q,\delta,D_0)$. Moreover, at different parts of the proof we will require $\delta',\delta,\xi$ to be smaller than certain (positive) functions of the parameters they depend on; in particular, all of the requirements can be satisfied simultaneously.

Finally, $m$ will be taken large enough in terms of all the previous parameters (hence ultimately in terms of $q$ and $\sigma$).

Write $\alpha=T'(q)$, and define $\kappa$ as
\begin{equation} \label{eq:def-kappa}
\kappa = (1-T^*(\alpha))\sigma/4.
\end{equation}
Then $\kappa>0$ thanks to Lemma \ref{lem:f-alpha-smaller-than-one}, and the assumption $T(q)<q-1$. (The reason for this choice will become clear later.)

We fix $x\in X$ for the rest of the proof, and observe that all estimates will in fact be independent of $x$. Let $\xi>0$ be a small enough number to be chosen later. If $\|\mu_x^{(m)}\|_q^q \le 2^{-(T(q)+\xi)m}$ then there is nothing to do, so from now on we assume that
\begin{equation} \label{eq:lower-bound-assumption}
\|\mu_x^{(m)}\|_q^q \ge 2^{-(T(q)+\xi)m}.
\end{equation}

We apply Proposition \ref{prop:Lq-over-small-set-is-small} to obtain a sufficiently large $D_0$  (in terms of $\delta,\sigma,q$, with $\delta$ yet to be specified) such that
\begin{enumerate}
\item For any $D'\ge D_0-2$, any $I\in\DD_{s'}$, $s'\in\N$, and any subset $\DD'\subset \DD_{s'+D'}(I)$ with $|\DD'|\le 2^{(T^*(\alpha)-\kappa)D'}$,
\[
\sum_{J\in\DD'} \mu_x(J)^q \le 2^{-(T(q)+\eta)D'}\mu_{x}(2I)^q,
\]
where $\eta$ depends on $\kappa$ and $q$, hence on $\sigma,q$ only.
\item For any $D'\ge D_0-2$ and any $I\in\DD_{s'}$, $s'\in\N$,
\[
\sum_{J\in\DD_{s'+D'}(I)} \mu_x(J)^q \le 2^{-(T(q)-\delta)D'}\mu_{x}(2I)^q.
\]
\item $1/D_0 <\delta$.
\end{enumerate}

Let $\e>0,D\in\N$ be the numbers given by Theorem  \ref{thm:inverse-thm} applied to $\delta, D_0$ and $q$. Suppose
\begin{equation} \label{eq:non-flattening-assumption}
\| \nu*\mu_x^{(m)}\|_q^q \ge 2^{-(T(q)+\e q/2) m}
\end{equation}
We will derive a contradiction from this provided $m=\ell D$ is large enough, proving the theorem with $\e q/2$ in place of $\e$ (if $m$ is not of the form $\ell D$, we apply the argument to $\lfloor m/D\rfloor D$ instead).

By Corollary \ref{cor:unif-continuity}, if $m$ is large enough (depending only on $\e,q$) and  \eqref{eq:non-flattening-assumption} holds, then
\[
\| \nu*\mu_x^{(m)}\|_q \ge 2^{-\e m} \|\mu_x^{(m)}\|_q.
\]
We apply Theorem \ref{thm:inverse-thm} to obtain (assuming $m$ is large enough) a set $A\subset\supp(\mu_x^{(m)})$ as in the theorem, with corresponding branching numbers $R'_s$.

The key to the proof is to show, using the structure of $A$ provided by Theorem \ref{thm:inverse-thm}, that
\begin{equation} \label{eq:small-branching-at-many-scales}
|\{ s\in [\ell]: R'_s \le 2^{(T^*(\alpha)-\kappa)D} \}| \ge \gamma \ell,
\end{equation}
where $\gamma>0$ depends on $q,\alpha$ and $\kappa$ only (and $\kappa$ is given by \eqref{eq:def-kappa}). We first show how to complete the proof assuming this. Consider the sequence
\[
L_s = -\log\sum_{I\in\DD_{s D}(A)} \mu_x(I)^q.
\]
By (2) applied with $s'=sD+2$ and $D'=D-2$,
\[
L_{s+1} \ge (T(q)-\delta)(D-2) -\log \sum_{I\in \DD_{sD+2}(A)} \mu_x(2I)^q.
\]
But if $I\in \DD_{sD+2}(A)$, then $2I$ is contained in a single interval in $\DD_{sD}(A)$ by property (A-iv) from Theorem \ref{thm:inverse-thm}, and conversely $J\in\DD_{sD}(A)$ hits at most two intervals $2I$, $I\in\DD_{sD+2}(A)$. We deduce that
\[
L_{s+1} \ge L_s +(T(q)-\delta)(D-2)-1
\]
for all $s\in[\ell]$. Likewise, by (1),
\[
L_{s+1} \ge L_s + (T(q)+\eta)(D-2)-1,
\]
whenever $R'_s \le  2^{(T^*(\alpha)-\kappa)D}$. Recall that $\eta$ depends on $q,\kappa$. In light of \eqref{eq:small-branching-at-many-scales}, and using also (3), we have
\begin{align*}
L_{\ell} &\ge (T(q)+\eta)\gamma \ell (D-2) + (T(q)-\delta)(1-\gamma)\ell (D-2)-\ell\\
 &\ge (T(q)+\eta\gamma-\delta(1-\gamma))m - 2\delta (T(q)+\eta) m -\delta m.
\end{align*}

Hence, by choosing $\delta$ small enough in terms of $T(q),\gamma$ and $\eta$ we can ensure that, for $m$ large enough,
\[
L_{\ell} = -\log\|\mu_x^{(m)}|_A\|_q^q \ge (T(q)+\eta\gamma/2)m.
\]
On the other hand, by (A-i) in Theorem \ref{thm:inverse-thm} and our assumption \eqref{eq:lower-bound-assumption},
\[
\|\mu_x^{(m)}|_A\|_q^q \ge 2^{-q\delta m}\|\mu_x^{(m)}\|_q^q \ge 2^{-q\delta m} 2^{-(T(q)+\xi)m}.
\]
From the last two displayed equations,
\[
\eta\gamma/2 \le q\delta +\xi.
\]
Recall that $\eta=\eta(\kappa,q),\gamma=\gamma(q,\alpha,\kappa)$ is yet to be specified, while $\delta$ so far was taken small enough in terms of $T(q),\gamma$ and $\eta$, and no conditions have been yet imposed on $\xi$. By ensuring $q\delta <\eta\gamma/8$ and $\xi\le \eta\gamma/8$ we reach a contradiction. Hence \eqref{eq:non-flattening-assumption} cannot hold, which is what we wanted to show.

It remains to establish \eqref{eq:small-branching-at-many-scales}. The idea is very simple: Theorem \ref{thm:inverse-thm} (together with the assumption that $\|\nu\|_q^{q'}\le 2^{-\sigma m}$) imply that $A$ has ``nearly full branching'' on a positive proportion of scales. On the other hand, Lemma \ref{lem:size-set-A-in-terms-of-f-alpha} says the size of $A$ is at most roughly $2^{T^*(\alpha)m}\ll 2^m$ (by Lemma \ref{lem:f-alpha-smaller-than-one}), so there must be a positive proportion of scales on which the average $2^D$-adic branching is far smaller than $2^{T^*(\alpha)D}$, which is what \eqref{eq:small-branching-at-many-scales} says.

We proceed to the details. Using (A-i), (A-ii) in Theorem \ref{thm:inverse-thm} and \eqref{eq:lower-bound-assumption}, we get that (for $m\gg_\delta 1$) there is $\wt{\alpha}>0$ such that $\mu_x(a)\in [2^{-\wt{\alpha} m},2^{1-\wt{\alpha} m}]$ for all $a\in A$, and
\begin{equation*} %\label{eq:sum-Lq-in-A-large}
\sum_{I\in\DD_m(A)} \mu_x(I)^q \ge 2^{-q\delta m} \sum_{I\in\DD_m}\mu_x(I)^q \ge  2^{-(T(q)+ q\delta+\xi)m}.
\end{equation*}
We let $\delta\le \delta'$  and $\xi$ be small enough in terms of $\delta'$ and $q$ that, invoking Lemma \ref{lem:size-set-A-in-terms-of-f-alpha},
\begin{equation} \label{eq:size-A-upper-bound}
|A| \le 2^{(T^*(\alpha)+\delta')m}.
\end{equation}

Let $\mathcal{S}'=[\ell]\setminus \mathcal{S}$, where $\mathcal{S} = \{ s: R'_s \ge 2^{(1-\delta)D}\}$. Using (A-iii) in Theorem \ref{thm:inverse-thm}, we see that
\begin{equation} \label{eq:size-A-lower-bound}
|A| = \prod_{s=0}^{\ell-1} R'_s \ge  2^{(1-\delta)D|\mathcal{S}|} \prod_{s\in\mathcal{S}'} R'_s.
\end{equation}
Let $m_1=D|\mathcal{S}|$, $m_2=D|\mathcal{S}'|=m-m_1$. Combining \eqref{eq:size-A-upper-bound} and \eqref{eq:size-A-lower-bound}, and using that $\delta\le \delta'$, we deduce
\begin{equation} \label{eq:upper-bound-Rs-S-prime}
\prod_{s\in\mathcal{S}'} R'_s \le 2^{-(1-\delta)m_1} 2^{(T^*(\alpha)+\delta')m} \le 2^{-(1-T^*(\alpha)-2\delta')m_1} 2^{(T^*(\alpha)+\delta')m_2}.
\end{equation}
Note that $1-T^*(\alpha)>0$ by Lemma \ref{lem:f-alpha-smaller-than-one}. At this point we take $\delta'$ small enough that $1-T^*(\alpha)-2\delta'>0$.
Using (vi) in Theorem \ref{thm:inverse-thm}, and the assumptions \eqref{eq:lower-bound-assumption} and $\|\nu\|_q^{q'}\le 2^{-\sigma m}$, we further estimate
\begin{equation} \label{eq:bound-total-branching-scales}
(\sigma - \delta)m  \le m_1 \le \left( (T(q)+\xi)/(q-1)+\delta\right)m.
\end{equation}
We can plug in the left inequality (together with $m_2\le m$) into \eqref{eq:upper-bound-Rs-S-prime}, to obtain the key estimate
\[
\log\prod_{s\in\mathcal{S}'} R'_s \le \left(T^*(\alpha)+ \delta' -(1-T^*(\alpha)-2\delta')(\sigma-\delta)\right)m_2.
\]
Recalling \eqref{eq:def-kappa}, this shows that by making $\delta'$ (hence also $\delta\le\delta'$) small enough in terms of $\alpha,\sigma,\kappa$, we have
\[
\log\prod_{s\in\mathcal{S}'} R'_s \le (T^*(\alpha)-2\kappa) m_2.
\]
Let $\mathcal{S}_1 = \{ s\in\mathcal{S}': \log R'_s \le (T^*(\alpha)-\kappa)D\}$. Recall that our goal is to show \eqref{eq:small-branching-at-many-scales}, i.e. $|\mathcal{S}_1|\ge \gamma(q,\alpha,\kappa)\ell$. We have
\[
D|\mathcal{S}'\setminus\mathcal{S}_1| \le \frac{1}{T^*(\alpha)-\kappa}\sum_{s\in\mathcal{S}'\setminus\mathcal{S}_1} \log R'_s \le \frac{T^*(\alpha)-2\kappa}{T^*(\alpha)-\kappa} D|\mathcal{S}'|,
\]
so that, using the right-most inequality in \eqref{eq:bound-total-branching-scales}, and recalling that $D|\mathcal{S}'|=m-m_1$,
\[
D|\mathcal{S}_1| \ge \frac{\kappa (m-m_1)}{T^*(\alpha)-\kappa} \ge \left(\frac{\kappa(1-(T(q)+\xi)/(q-1)-\delta)}{T^*(\alpha)-\kappa}\right)  m.
\]
By ensuring that $\delta,\xi$ are small enough in terms of $q$, the right-hand side above can be bounded below by
\[
\left(\frac{\kappa(1-T(q)/(q-1))/2}{T^*(\alpha)-\kappa}\right) m,
\]
confirming that \eqref{eq:small-branching-at-many-scales} holds with $\gamma=\gamma(q,\alpha,\kappa)$.
\end{proof}

\subsection{\texorpdfstring{$L^q$}{Lq} norms of \texorpdfstring{$\mu_{x,n}$}{discrete approximations} at finer scales}

Theorem \ref{thm:L-q-dim-dyn-ssm} will be an easy consequence of the following proposition,  which relies on Theorem \ref{thm:conv-with-ssm-flattens}. It is an analog of \cite[Theorem 1.4]{Hochman14}, and we follow a similar outline.
\begin{prop} \label{prop:ssm-scale-Rm-norm}
Let $(X,\mathbf{T},\Delta,\lam)$ be a pleasant model, and let $T$ be the function from Proposition \ref{prop:cocycle}. Let $q\in (1,\infty)$ be such that $T$ is differentiable at $q$ and $T(q)<q-1$, and let $x\in X$ be such that
\begin{equation} \label{eq:Lq-dim-exists-equal-Tq}
\lim_{m\to\infty} \frac1m \log\|\mu_x^{(m)}\|_q^q = -T(q).                                                                                                                                                                                                                                                                                                                                                                                                                                                                              \end{equation}

Fix $R\in\N$. Then
\[
\lim_{n\to\infty} \frac{\log \|\mu_{x,n}^{(Rm(n))}\|_q^q}{n\log\lam} = T(q),
\]
where $m(n)$ is the smallest integer with $2^{-m(n)}\le \lam^n$.
\end{prop}
\begin{proof}
Fix $n\in\N$. We write $m=m(n)$ for simplicity, and allow all implicit constants to depend on $q$ and the model only. Using the self-similarity relation \eqref{eq:mu-x-self-similarity} and Lemma \ref{lem:discr-norm-conv-equivalence}, we get
\begin{align*}
\|\mu_x^{((R+1)m)}\|_q^q &\le O(1) \| \mu_{x,n}^{((R+1)m)} * (S_{\lam^n} \mu_{\mathbf{T}^n x})^{((R+1)m)} \|_q^q \\
&= O(1) \big\|  \sum_{I\in\DD_m} \mu_{x,n}(I) (\mu_{x,n})_I^{((R+1)m)} * (S_{\lam^n} \mu_{\mathbf{T}^n x})^{((R+1)m)} \big\|_q^q.
\end{align*}
Here $(\mu_{x,n})_I  = \mu_{x,n}|_I / \mu_{x,n}(I)$ is the normalized restriction of $\mu_{x,n}$ to $I$ (note that we are only summing over $I$ such that $\mu_{x,n}(I)>0$). Since the measures $(\mu_{x,n})_I^{((R+1)m)} * (S_{\lam^n} \mu_{\mathbf{T}^n x})^{((R+1)m)}$ are supported on $I+[0,\lam^n]$, the support of each of them hits the supports of $O(1)$ others. We can then apply Lemma \ref{lem:L-q-norm-almost-disj-supports} to obtain
\[
\|\mu_x^{((R+1)m)}\|_q^q  \le O(1)\sum_{I\in\DD_m} \mu_{x,n}(I)^q  \|(\mu_{x,n})_I^{((R+1)m)} * (S_{\lam^n} \mu_{\mathbf{T}^n x})^{((R+1)m)} \|_q^q
\]
Let $\rho_{x,I} = S_{\lam^{-n}}(\mu_{x,n})_I$ (we suppress the dependence on $n$ from the notation, but keep it in mind). Note that $S_a(\eta)*S_a(\eta')=S_a(\eta*\eta')$ for any $a>0$ and measures $\eta,\eta'$. It follows from Lemmas \ref{lem:Holder} and \ref{lem:discr-norm-conv-equivalence}  that
\[
 \|(\mu_{x,n})_I^{((R+1)m)} * (S_{\lam^n} \mu_{\mathbf{T}^n x})^{((R+1)m)} \|_q^q \le O(1) \| \rho_{x,I}^{(Rm)} * \mu_{\mathbf{T}^n x}^{(Rm)} \|_q^q,
\]
so that, combining the last two displayed formulas,
\begin{equation} \label{eq:estimate-mu-Rp1}
\|\mu_x^{((R+1)m)}\|_q^q  \le O(1) \sum_{I\in\DD_m} \mu_{x,n}(I)^q \| \rho_{x,I}^{(Rm)} * \mu_{\mathbf{T}^n x}^{(Rm)} \|_q^q.
\end{equation}
On the other hand,  using Lemma \ref{lem:Holder} again,
\begin{equation} \label{eq:estimate-mu-m-Rp1}
\|\mu_{x,n}^{((R+1)m)}\|_q^q = \sum_{I\in\DD_m} \mu_{x,n}(I)^q \|(\mu_{x,n})_I^{((R+1)m)}\|_q^q \ge  \Omega(1) \sum_{I\in\DD_m} \mu_{x,n}(I)^q \| \rho_{x,I}^{(Rm)} \|_q^q.
\end{equation}

Fix $\sigma>0$, and let $\mathcal{D}' = \{ I\in\DD_{m}: \| \rho_{x,I}^{(Rm)}\|_q^q \le 2^{-\sigma m}\}$. According to Theorem \ref{thm:conv-with-ssm-flattens}, there is $\e=\e(\sigma,q)>0$ such that, if $n$ is taken large enough, then
\[
I\in\mathcal{D}' \quad\Longrightarrow\quad\|\rho_{x,I}^{(Rm)} * \mu_{\mathbf{T}^{n}(x)}^{(Rm)}\|_q^q \le 2^{-(T(q)+\e)Rm}.
\]
Applying this to \eqref{eq:estimate-mu-Rp1},  we get
\begin{align*}
\|\mu_x^{((R+1)m)}\|_q^q  &\le O(1) 2^{-(T(q)+\e)Rm} \sum_{I\in\DD'} \mu_{x,n}(I)^q  + O(1) \sum_{I\notin \DD'}  \mu_{x,n}(I)^q \|  \mu_{\mathbf{T}^n x}^{(Rm)} \|_q^q \\
&\le O(1) 2^{-(T(q)+\e)Rm} \|\mu^{(m)}_{x}\|_q^q  + O(1)\|  \mu_{\mathbf{T}^n x}^{(Rm)} \|_q^q \sum_{I\notin\DD'} \mu_{x,n}(I)^q\\
\end{align*}
using Young's inequality in the first line, and Lemma \ref{lem:comparison-mu-m-mu} in the second. On the other hand, our assumption \eqref{eq:Lq-dim-exists-equal-Tq} implies that
\[
 2^{-(T(q)+\e)Rm} \|\mu^{(m)}_{x}\|_q^q \le 2^{-\e m/2} \|\mu_x^{((R+1)m)}\|_q^q
\]
if $n$ is large enough (depending on $x$ and $R$). Inspecting the last two displayed equations, we deduce that if $n\gg_{x,\sigma} 1$, then
\[
\sum_{I\notin \DD'}  \mu_{x,n}(I)^q \ge \Omega(1) \frac{\|\mu_x^{((R+1)m)}\|_q^q}{\|\mu_{\mathbf{T}^n x}^{(Rm)} \|_q^q } \ge 2^{-m(T(q)+\sigma)},
\]
where for the right-most inequality we used  the assumption \eqref{eq:Lq-dim-exists-equal-Tq} and Corollary \ref{cor:unif-continuity}. Recalling \eqref{eq:estimate-mu-m-Rp1}, we conclude that
\begin{align*}
\|\mu_{x,n}^{((R+1)m)}\|_q^q  &\ge \Omega(1) \sum_{I\notin\DD'} \mu_{x,n}(I)^q \| \rho_{x,I}^{(Rm)} \|_q^q \\
&\ge \Omega(1) 2^{-\sigma m}\sum_{I\notin\DD'} \mu_{x,n}(I)^q \ge \Omega(1) 2^{-2\sigma m} 2^{-m T(q)}.
\end{align*}
The inequality $\|\mu_{x,n}^{((R+1)m)}\|_q^q \le \|\mu_{x,n}^{(m)}\|_q^q$ holds trivially, so that by Lemma \ref{lem:comparison-mu-m-mu}
\[
\|\mu_{x,n}^{((R+1)m)}\|_q^q  \le \|\mu_{x,n}^{(m)}\|_q^q \le 2^{\sigma m} 2^{-m T(q)},
\]
provided $n\gg_\sigma 1$. Since $\sigma>0$ was arbitrary and $2^{-m}=\Theta(\lam^n)$, this concludes the proof.
\end{proof}

\subsection{Proof of Theorem \ref{thm:L-q-dim-dyn-ssm}}

We can now conclude the proof of the theorem.
\begin{proof}[Proof of Theorem \ref{thm:L-q-dim-dyn-ssm}]
We continue to write $m=m(n)=\lceil n\log(1/\lam)\rceil$. To begin, we note that, without any assumptions on the model, for any $q\in (1,\infty)$,
\begin{equation} \label{eq:main-thm-upper-bound-1}
\|\mu_{x,n}^{(m)}\|_q^q \ge \|\mu_{x,n}\|_q^q \ge \prod_{i=0}^{n-1} \|\Delta(\mathbf{T}^i x)\|_q^q.
\end{equation}
(The latter inequality is an equality if and only if there are no overlaps among the atoms of $\mu_{x,n}$.) By our assumptions on the map $\Delta(\cdot)$, the function $x\mapsto \|\Delta(x)\|_q^q$ is bounded away from zero and continuous $\PP$-a.e. Then, by unique ergodicity,
\begin{equation} \label{eq:main-thm-upper-bound-2}
\lim_{n\to\infty} \frac1n\log\prod_{i=0}^{n-1} \|\Delta(\mathbf{T}^i x)\|_q^q = \int_X \log\|\Delta(x)\|_q^q \,d\mathbb{P}(x)\quad\text{uniformly in }x\in X.
\end{equation}
This property of uniquely ergodic systems is well known, or one can apply Lemma \ref{lem:uniform-convergence} to the additive sequence $\log\prod_{i=0}^{n-1} \|\Delta(\mathbf{T}^i x)\|_q^q$. Since $\|\nu^{(m)}\|_q^{q'} \ge 2^{-m}$ for any probability measure $\nu$, from \eqref{eq:main-thm-upper-bound-1}, \eqref{eq:main-thm-upper-bound-2} and Lemma \ref{lem:comparison-mu-m-mu}, we deduce that
\[
\limsup_{m\to\infty} -\frac{\log\|\mu_x^{(m)}\|_q^q}{(q-1)m}  \le  \min\left( \frac{\int_X \log\|\Delta(x)\|_q^q \,d\mathbb{P}(x)}{(q-1)\log\lam},1\right),
\]
uniformly in $x\in X$. In light of this and Corollary \ref{cor:unif-continuity}, the proof will be completed if we can show that for each $q\in (1,\infty)$, either $T(q)\ge q-1$ (so that in fact $T(q)=q-1$) or
\begin{equation} \label{eq:formula-T-q-want-to-prove}
T(q) = \frac{\int_X \log\|\Delta(x)\|_q^q \,d\mathbb{P}(x)}{\log\lam}.
\end{equation}
Since $T(q)$ is concave, it is enough to prove this for all $q$ such that $T$ is differentiable at $q$. Hence, we fix $q$ such that $T(q)<q-1$ and $T$ is differentiable at $q$, and we set out to prove \eqref{eq:formula-T-q-want-to-prove}.

By Proposition \ref{prop:cocycle} and the exponential separation assumption, there is $x\in X$ such that \eqref{eq:Lq-dim-exists-equal-Tq} holds, and the atoms of $\mu_{x,n}$ are $\lam^{R n}$-separated for infinitely many $n$ and some $R\in\N$ (indeed, this holds for $\PP$-almost all $x$). We known from Proposition \ref{prop:ssm-scale-Rm-norm}  that
\begin{equation} \label{eq:tau-mu-x-equals-T}
\lim_{n\to\infty} \frac{\log \|\mu_{x,n}^{(Rm(n))}\|_q^q}{n\log\lam} = T(q).
\end{equation}
On the other hand, if $n$ is such that the atoms of $\mu_{x,n}$ are $\lam^{R n}$-separated then (since $\lam^{Rn}\ge 2^{-R m(n)}$)
\begin{equation} \label{eq:tau-mu-x-exp-separation}
\|\mu_{x,n}^{(Rm(n))}\|_q^q = \|\mu_{x,n}\|_q^q = \prod_{i=0}^{n-1} \| \Delta(\mathbf{T}^i x)\|_q^q.
\end{equation}
 Combining Equations \eqref{eq:main-thm-upper-bound-2}, \eqref{eq:tau-mu-x-equals-T} and \eqref{eq:tau-mu-x-exp-separation}, we conclude that \eqref{eq:formula-T-q-want-to-prove} holds, finishing the proof.
\end{proof}

\section{\texorpdfstring{$L^q$}{Lq} dimensions of self-similar measures, and applications}
\label{sec:dim-ssm-and-applications}

In this section we apply Theorem \ref{thm:L-q-dim-dyn-ssm} to prove  Theorem \ref{thm:infinity-dim-BCs}; in fact, we will obtain a far more general result for self-similar measures on the line. We also derive some geometric applications.

\subsection{Background on self-similar sets and measures}
\label{subsec-sssm}

We begin by recalling some basic facts about self-similar sets and measures, fixing notation along the way. For further background, see e.g. \cite{Falconer97}.

Let $\mathcal{I}$ be a finite set with at least two elements. Let $(f_i)_{i\in\mathcal{I}}$ be a collection of strictly contracting similarities on $\R^d$ (usually referred to as an \emph{iterated function system} or IFS). That is, $f_i(x)=\lambda_i O_i(x)+t_i$, where $\lam_i\in (0,1)$, $O_i$ is an orthogonal map on $\R^d$, and $t_i\in\R^d$. Then there exists a unique nonempty compact set $A\subset\R^d$ such that
\[
A = \bigcup_{i\in\mathcal{I}} f_i(A).
\]
If a probability vector $(p_i)_{i\in\mathcal{I}}$ is also given, then there is a unique Borel probability measure $\mu$ such that
\[
\mu = \sum_{i\in\mathcal{I}} p_i\, f_i\mu.
\]
Moreover, $\supp(\mu)\subset A$, with equality if $p_i>0$ for all $i$.

If one replaces $\mathcal{I}$ by $\mathcal{I}^n$, $(f_i)$ by $(f_{i_1}\circ\cdots\circ f_{i_n})$, and $(p_i)$ by $(p_{i_1}\cdots p_{i_n})$, then the invariant set $A$ and the invariant measure $\mu$ do not change.

The Hausdorff and box counting dimensions agree for any self-similar set. The \emph{open set condition} holds if there is a nonempty open set $U$ such that $f_i(U)\subset U$ and $f_i(U)\cap f_j(U)=\varnothing$ for all $i\neq j\in\mathcal{I}$. In this case, the Hausdorff dimension of $A$ is the only positive number $s$ such that $\sum_{i\in\mathcal{I}} \lam_i^s=1$. Moreover, the \emph{uniform self-similar measure} $\mu$ given by the weights $\lam_i^s$ satisfies $\mu(B(x,r))=\Theta(r^s)$ for $x\in A$ and $r\in (0,1]$, with the implicit constants depending only on $(f_i)$.

In this article we will be mostly concerned with \emph{homogeneous} iterated function systems: those for which $\lambda_i\equiv\lam$ and $O_i\equiv O$ are constant for all $i\in\mathcal{I}$. In this case, the self-similar set $A$ can be explicitly written as an infinite arithmetic sum:
\[
A = \sum_{i=0}^\infty S_{\lam^i}(O^i E),
\]
where $E=\{t_i:i\in\mathcal{I}\}$ is the set of translations, and the self-similar measure $\mu$ can be expressed as an infinite convolution:
\[
\mu = *_{i=0}^\infty S_{\lam^i}(O^i\Delta),
\]
where $\Delta= \sum_{i\in\mathcal{I}} p_i \delta(t_i)$. Note that in dimension $1$ (where most of the focus will be), $O$ is either the identity or minus the identity, and the latter case can always be reduced to the first by iterating the IFS, as above.

If the system is homogeneous and the open set condition holds, then there is $c>0$ such that for all $n\in\N$, the points in the finite approximation
\[
A_n = \sum_{i=0}^{n-1} S_{\lam^i}(O^i E)
\]
are all distinct (i.e. there are $|E|^n$ of them) and $c\lam^n$ separated. See e.g. \cite[Example 1 in Section 6]{LauNgai99}. Moreover,
in this case the $L^q$ dimensions of $\mu$ are given by
\[
D(\mu,q) = \frac{\log \|\Delta\|_q^q}{(q-1)\log\lam}.
\]
The right-hand side majorizes the $L^q$ dimension without any separation assumption (always assuming homogeneity).

Finally, we point out that the limit in the definition of $L^q$ dimension exists for arbitrary self-similar measures, see \cite{PeresSolomyak00}.

\subsection{\texorpdfstring{$L^q$}{Lq} dimensions and Frostman exponents of self-similar measures}

Next, we obtain Theorem \ref{thm:infinity-dim-BCs} as a special case of a result valid for more general self-similar measures on $\R$. Fix $\Delta=\sum_{i\in\mathcal{I}} p_i \delta_{t_i}\in\mathcal{A}$ and $\lam\in (0,1)$, and let
\begin{equation} \label{eq:def-ssm-hom}
\mu = \mu_{\Delta,\lam} = *_{i=0}^\infty S_{\lam^i}\Delta
\end{equation}
be the associated self-similar measure. Bernoulli convolutions correspond to the special case $\Delta=\frac12(\delta_{-1}+\delta_1)$.

\begin{defn} \label{def:exponential-separation-hom-ssm}
Given a set $E\subset\R$ and $n\in\N$, we let $\mathcal{P}_{E,n}$ be the family of non-zero polynomials of degree at most $n$ and coefficients in $E-E$. Slightly abusing notation, we write $\mathcal{P}_{\Delta,n}=\mathcal{P}_{\supp(\Delta),n}$.

We say that a measure $\mu$ as in \eqref{eq:def-ssm-hom} has \emph{exponential separation} if there exists $R>0$ such that, for infinitely many $n$,
\begin{equation} \label{eq:ssm-exponential-separation}
\min_{P\in\mathcal{P}_{\Delta,n}} |P(\lam)| \ge \lam^{R n}.
\end{equation}
\end{defn}

Note that this is a property of $\supp(\Delta)$ and $\lam$, and not of the particular distribution of mass on $\supp(\Delta)$. Recall that if the open set condition  holds, then there is $c>0$ such that
\[
|P(\lam)| \ge c\lam^n \quad\text{for all }n\in\N, P\in\mathcal{P}_{\Delta,n}.
\]
Hence, exponential separation is a weaker property than the open set condition.

\begin{thm} \label{thm:dim-ssm}
Let $\mu=\mu_{\Delta,\lam}$ be a self-similar measure as in \eqref{eq:def-ssm-hom} with exponential separation.
Then for all $q\in (1,+\infty)$,
\[
D(\mu,q) = \min\left(\frac{\log\|\Delta\|_q^q}{(q-1)\log\lam},1\right).
\]
In particular, for every
\[
\alpha < \min\left(\frac{\log \|\Delta\|_\infty}{\log\lam},1\right),
\]
it holds that $\mu(B(x,r)) \le r^\alpha$ for all $r\in (0,r_0(\alpha))$ and all $x\in\R$.
\end{thm}

Before presenting the short deduction from Theorem \ref{thm:L-q-dim-dyn-ssm}, we make some remarks on this statement:
\begin{enumerate}[label={\upshape\alph*)}]
\item Theorem \ref{thm:infinity-dim-BCs} is an immediate consequence of the last claim in the theorem.
\item Recall from \S\ref{subsec-sssm} that the claim in the theorem is well-known under the open set condition. The point is that the separation assumption is far weaker than the open set condition. This notion of ``exponential separation'' was introduced in \cite{Hochman14} and, as explained there, it is a quantitative version of the ``no exact overlaps'' condition which is conjectured to already imply the claims in Theorem \ref{thm:dim-ssm}.
\item As  shown by Hochman \cite{Hochman14}, if $\lam$ and all points in $\supp(\Delta)$ are algebraic, then either \eqref{eq:ssm-exponential-separation} holds for all $n$, or $\lam$ is a root of some $P\in\mathcal{P}_{\Delta,n}$, $n\in\N$ (which corresponds to an exact overlap).
\item Hochman (\cite[Theorem 1.8]{Hochman14}, \cite[Theorem 1.10]{Hochman17}) has also shown that in quite general parametrized families of self-similar measures, the exponential separation assumption in Theorem  \ref{thm:dim-ssm} holds outside of a set of parameters of packing and Hausdorff co-dimension at least $1$.
\item The analog of Theorem \ref{thm:dim-ssm} for exact (or Hausdorff) dimension was established by Hochman \cite[Theorem 1.1]{Hochman14}. We recover his result in the homogeneous case by letting $q\to 1^{+}$ in Theorem \ref{thm:dim-ssm}.
\item  Note that $q\mapsto-\log\|\Delta\|_q^q$ is linear if and only if $\Delta$ is uniform on its support; otherwise, it is a strictly concave real-analytic function. It follows from the theorem that, under the separation assumption \eqref{eq:ssm-exponential-separation}, the map $q\mapsto \tau_\mu(q)$ is differentiable except, perhaps, at a single point $q>1$ such that $\|\Delta\|_q^q = \lam^{q-1}$. It follows from a result of D-J. Feng \cite{Feng07} that the multifractal formalism holds for $\mu$ and all $q\in (1,\infty)$ outside, possibly, of this point. See \cite{Feng07} for details.
\end{enumerate}

\begin{proof}[Proof of Theorem \ref{thm:dim-ssm}]
We apply Theorem \ref{thm:L-q-dim-dyn-ssm} with a constant function $\Delta$ (corresponding to a one-point set $X=\{0\}$). Such a trivial model is clearly pleasant and satisfies the continuity assumption in Theorem  \ref{thm:L-q-dim-dyn-ssm}. The support of $\mu_n:=\mu_{0,n}$ is
\[
\sum_{i=0}^{n-1} S_{\lam^i}(\Delta) =  \left\{ \sum_{i=0}^{n-1}  \lam^i y_i : y_i\in\supp(\Delta)\right\},
\]
so the model has exponential separation if and only if $\mu$ has exponential separation. The application of Theorem \ref{thm:L-q-dim-dyn-ssm} is therefore justified, and yields the claimed formula for $D_\mu(q)$. The latter claim for the Frostman exponent then follows from Lemma \ref{lem:Lq-dim-to-Frostman-exp} by letting $q\uparrow\infty$.
\end{proof}

\subsection{Some applications}

We present some consequences of Theorem \ref{thm:dim-ssm}.  Recall that the one-dimensional Sierpi\'{n}ski Gasket $S$ is the set of all points in $[0,1]^2$ of the form
\[
\left\{ \sum_{n=1}^\infty X_n 3^{-n} : X_n\in \{(0,0),(0,1),(1,0)\} \right\}.
\]
The gasket $S$ is a self-similar set, with open set condition, of Hausdorff dimension $1$. Furstenberg conjectured that all orthogonal projections of $S$ in directions with irrational slope also have Hausdorff dimension $1$; this was proved in \cite[Theorem 1.6]{Hochman14}. We can deduce a stronger statement from Theorem \ref{thm:dim-ssm}:

\begin{cor} \label{cor:sierpinski-gasket}
Let $\Pi_t(x,y)=x+t y$. For every Borel subset $A\subset S$ and for every $t\in\R\setminus\mathbb{Q}$,
\[
 \hdim(\Pi_t A) = \hdim(A).
\]
\end{cor}
\begin{proof}
Let $\mu$ be the uniform self-similar measure on $S$, so that $\mu(B(x,r))=\Theta(r)$ for $x\in S$ and $r\le 1$. For each $t\in\R$, the projection $\Pi_t\mu$ is the uniform self-similar measure for the iterated function system $\{ x/3, x/3+1, x/3+t\}$. As shown in the proof of \cite[Theorem 1.6]{Hochman14}, this IFS satisfies the exponential separation hypothesis \eqref{eq:ssm-exponential-separation} for all irrational $t$. From now on let $t$ be a fixed irrational number. We deduce from Theorem \ref{thm:dim-ssm} that
\[
\Pi_t\mu(B(y,\e)) = O_{t,\delta}(\e^{1-\delta})\quad\text{for all }\delta>0.
\]
In turn, Lemma \ref{lem:Frostman-exp-to-small-fiber} says that $\Pi_t^{-1}(B(y,\e))$ can be covered by $O_{t,\delta}(\e^{-\delta})$ balls of radius $\e$ for any $y\in\R$. Indeed, if $(x_j)_j$ is a maximal $\e$-separated subset of some set, then $(B(x_j,\e))_j$ covers the set.

Now fix a Borel subset $A\subset S$ of Hausdorff dimension $s$, and  $\delta>0$. By Frostman's Lemma (see e.g. \cite[Theorem 8.8]{Mattila95}) there is a Borel probability measure $\nu$ supported on $A$ such that $\nu(B(x,\e)) = O_{A,\delta}(\e^{s-\delta})$ for all $x\in\R^2, \e>0$. It follows that
\[
\Pi_t\nu(B(y,\e))  = O_{A,\delta,t}(\e^{s-2\delta}).
\]
Since $\delta>0$ was arbitrary, the conclusion follows from the mass distribution principle (see e.g. \cite[Proposition 2.1]{Falconer97}).
\end{proof}

The gasket $S$ could be replaced by the attractor of any iterated function system in the plane, satisfying the open set condition and of Hausdorff dimension at most $1$, of the form $(\lambda x+a_i,\lambda y+b_i)_{i\in\mathcal{I}}$ with $\lambda,a_i,b_i$ all rational. If $\lam,a_i,b_i$ are only assumed to be algebraic, then the same holds assuming that $t$ is transcendental, instead of irrational. The proof works verbatim since in this more general situation $\Pi_t S$ continues to be a self-similar set satisfying \eqref{eq:ssm-exponential-separation}, see the proof of \cite[Theorem 5.3]{Shmerkin15}.

When the Hausdorff dimension of the self-similar set is larger than $1$ we cannot reach the same conclusion, but Lemma \ref{lem:Frostman-exp-to-small-fiber} still provides an upper bound for the size of the fibers. We conclude this section by discussing some concrete classes of examples.
\begin{cor} %\label{cor:intersection-Cantor-set}
Let $A\subset [0,1)$ be a $p$-Cantor set, $p\ge 2$. Then for every irrational number $t\in\R$ and any $u\in\R$,
\[
\ubdim(A\cap (t A+ u)) \le \max(2\hdim(A)-1,0).
\]
\end{cor}
\begin{proof}
The product set $S=A\times A$ is the attractor of an iterated function system with rational coefficients satisfying the open set condition, and $\hdim(S)=2\hdim A$. As pointed out above, it is shown in the proof of \cite[Theorem 5.3]{Shmerkin15} that $\Pi_t S$ is a self-similar set satisfying \eqref{eq:ssm-exponential-separation} whenever $t$ is irrational (the argument for this holds regardless of the dimension of the self-similar set). Since the fiber $\Pi_t^{-1}(u)\cap S$ is, up to an affine change of coordinates, equal to $A\cap (t A+ u)$, the conclusion follows from Theorem \ref{thm:dim-ssm} and Lemma \ref{lem:Frostman-exp-to-small-fiber}.
\end{proof}

The corollary generalizes to $T_p$-invariant sets, by embedding them in $p$-Cantor sets of arbitrarily close dimension, see the proof of Theorem \ref{thm:Furstenberg} in \S \ref{subsec:proof-of-Furst-conjecture} below. The dimension of the intersections of the middle-thirds Cantor set with translates of itself (without scaling) was investigated by Hawkes \cite{Hawkes75}, and this was greatly generalized to $T_p$-invariant sets by Kenyon and Peres \cite{KenyonPeres91}. Without scaling, the situation is very different; in particular, $\hdim(A\cap A+u) > 2\hdim(A)-1$ for many values of $u$. We mention also a related result of M. Hochman \cite{Hochman12} for invariant \emph{measures}: if $\mu$ is $T_p$-invariant, $\hdim\mu\in (0,1)$ and $f(x)=t x+u$ with $\log t/\log p\notin \Q$, then $\mu$ and $f\mu$ are mutually singular.

Likewise, if $S$ is the standard Sierpi\'{n}ski gasket or the Sierpi\'{n}ski carpet, or more generally if $S$ is the attractor of an IFS of the form $\left(N^{-1}(x+a_i,y+b_i)\right)_{i\in I}$ with  $(a_i,b_i)\in [N]^2$, and $|\mathcal{I}|>N$ (so that $\hdim S>1$), then
\[
\ubdim(S\cap\ell) \le \hdim(S)-1
\]
for all lines $\ell$ with irrational slope. The intersections of these carpets with lines of \emph{rational} slope was investigated in several papers; see \cite{BFS12, BaranyRams14} and references there. In particular, in those two papers it is shown that for the gasket and many other carpets $S$, there are many lines with a given rational slope that intersect $S$ in a set of dimension $> \hdim(S)-1$. More precisely, given a rational slope, a typical slice (with respect to the uniform self-similar measure) has a constant dimension strictly larger than $\hdim(S)-1$.

\subsection{General self-similar measures on the line}
\label{subsec:general-ssm}

We conclude this section by extending Theorem \ref{thm:dim-ssm} to general (not necessarily homogeneous) self-similar measures on $\R$. Although we are no longer in a setting in which Theorem \ref{thm:L-q-dim-dyn-ssm} can be applied, we will see that the same approach, with minor changes, can be used to directly establish the desired result.

We begin by defining a notion of exponential separation, which again agrees with that in \cite{Hochman14}, and extends the one given here in the homogeneous case. We define a distance between two affine maps $g_i(x)=\lam_i x +t_i$ on $\R$ as
\[
d(g_1,g_2) = \left\{ \begin{array}{ccc}
                            |t_1-t_2| & \text{ if } & \lam_1=\lam_2 \\
                            1 & \text{ if } & \lam_1\neq\lam_2
                          \end{array}
  \right..
\]

Let $(f_i)_{i\in\mathcal{I}}$ be strictly contractive, invertible affine maps on $\R$, i.e. $f_i(x)=\lam_i(x)+t_i$, where $|\lam_i|\in (0,1)$ and $t_i\in\R$. Given a finite word $u\in\mathcal{I}^k$, we write $f_u = f_{u_1}\cdots f_{u_k}$,  $f_u(x)=\lam_u x+t_u$, and $p_u=p_{u_1}\cdots p_{u_k}$. If $k\ge 1$, we also write $u^{-}$ for the word obtained from $u$ by deleting the last symbol.

Given $m\in\N$, let $\Omega_m$ be the family of all words $u$ such that $\lam_u\le 2^{-m}$ but $\lam_{u^{-}}>2^{-m}$. We can now define:
\begin{defn}
We say that the IFS $(f_i)_{i\in\mathcal{I}}$ has \emph{exponential separation} if there are $R>0$ and a sequence $m_j\to\infty$ such that
\[
d(f_u,f_v) \ge 2^{-R m_j} \quad\text{for all }u\neq v\in \Omega_{m_j}.
\]
\end{defn}

\begin{thm} \label{thm:dim-general-ssm}
Let $(f_i)_{i\in\mathcal{I}}$ be an IFS with exponential separation, and consider a self-similar measure
\[
\mu = \sum_{i\in\mathcal{I}} p_i\, f_i\mu.
\]
Then $D(\mu,q)=\min(\wt{\tau}(q)/(q-1),1)$, where $\wt{\tau}(q)$ is the only solution to $\sum_{i\in\mathcal{I}} p_i^q |\lam_i|^{-\wt{\tau}(q)}=1$.
\end{thm}

As many of the steps in the proof of Theorem \ref{thm:dim-general-ssm} are small variants of corresponding steps in the proof of Theorem \ref{thm:L-q-dim-dyn-ssm}, we will present an outline emphasizing the main differences, and leave the verification of the details to the interested reader. For simplicity we will assume that $\lam_i>0$ for all $i$; the general case can be deduced with minor notational changes.

Let $\tau(q)=\tau(\mu,q)$. We have to show that either $\tau(q)=q-1$ or $\tau(q)=\wt{\tau}(q)$. Hence, in order to prove Theorem \ref{thm:dim-general-ssm} it is enough to establish:
\begin{prop} \label{prop:dim-general-ssm}
Under the assumptions of Theorem \ref{thm:dim-general-ssm}, if $q\in (1,\infty)$ is such that $\tau(q)<q-1$, then $\tau(q)= \wt{\tau}(q)$.
\end{prop}

To prove the proposition, we begin by observing that Lemmas \ref{lem:f-alpha-smaller-than-one}--\ref{lem:Lq-sum-over-small-set} hold if $T(q)$ if replaced by $\tau(q)$ and $\mu_x$ by $\mu$. Indeed, the proofs only use concavity of $T$, and Corollary \ref{cor:unif-continuity}, both of which remain true for $\tau$ and $\mu$ by the definition and basic properties of $\tau$ (since we are dealing with just one measure, one needs not worry about uniform convergence in this context).

As a consequence, Proposition \ref{prop:Lq-over-small-set-is-small} also remains valid with $\tau$ in place of $T$ and $\mu$ in place of $\mu_x$. Indeed, given $m\in\N$, we define
\[
\mu_m = \sum_{u\in\Omega_m} p_u \delta(t_u).
\]
We note that this does not fully agree with our earlier notation in the homogeneous case. Given $s\in\N$ and $I\in\DD_s$, we let $y_j$ be the atoms of $\mu_{s+2}$  such that $[y_j,y_j+2^{-s-2}]\cap I\neq\varnothing$, let $p_j$ be their masses, and define
\[
\mu_{I} = \sum_j p_j \delta_{y_j}.
\]
The proof of Proposition \ref{prop:Lq-over-small-set-is-small} then goes through using the measures $\mu_{I}$ instead of $\mu_{x,n,I}$.

In turn, Theorem \ref{thm:conv-with-ssm-flattens} remains valid if, once again, we replace $T(q)$ by $\tau(q)$ and $\mu_x$ by the fixed self-similar measure $\mu$. This is because the proof of Theorem \ref{thm:conv-with-ssm-flattens} relies only on Corollary \ref{cor:unif-continuity}, Lemmas \ref{lem:f-alpha-smaller-than-one} and \ref{lem:size-set-A-in-terms-of-f-alpha}, and Proposition \ref{prop:Lq-over-small-set-is-small}, all of which we have seen  continue to hold in our context.

The main change comes in the proof of the analog of Proposition \ref{prop:ssm-scale-Rm-norm}, which nevertheless remains valid:
\begin{prop}  \label{prop:general-ssm-smaller-scale}
Using the notation above, fix $q\in (1,\infty)$ such that $\tau$ is differentiable at $q$ and $\tau(q)<q-1$. Then, for any $R\in\N$
\[
\lim_{m\to\infty} \frac{\log\|\mu_{m}^{(R m)}\|_q^q }{m} =  -\tau(q)
\]
\end{prop}
\begin{proof}
The key difference with the setting of Proposition \ref{prop:ssm-scale-Rm-norm} is that $\mu$ is no longer a convolution of a scaled down version of itself and a discrete approximation. However, $\mu$ is still a convex combination of a ``small'' number of measures which do have this structure.  Indeed, given $m\in\N$, let $\Lambda_m$ be the set of contraction ratios $\{ \lam_u: u\in\Omega_m\}$. For $\lam\in\Lambda_m$, define
\[
\mu_{m,\lam} = \sum\{ p_u \delta(t_u)  : u\in\Omega_m,\lam_u = \lam\}.
\]
Note that $\mu_{m,\lam}$ is positive and finite but does not have mass $1$ in general. The elements of $\Lambda_m$ are of the form $\prod_{i\in\mathcal{I}} \lam_i^{n_i}$, where $\lam_i^{n_i} \ge (\min_{i\in\mathcal{I}}\lam_i)2^{-m}$. It follows that
\begin{equation} \label{eq:Lambda-size-small}
|\Lambda_m| \le O(m^{|\mathcal{I}|}),
\end{equation}
with the implicit constant depending only on $|\mathcal{I}|$ and $(\lam_i)_{i\in\mathcal{I}}$. By self-similarity we have
\begin{equation} \label{eq:decomposition-ssm-mu}
\mu = \sum_{u\in\Omega_m} p_u \,f_u\mu = \sum_{\lam\in \Lambda_m} \mu_{m,\lam} * S_{\lam}\mu,
\end{equation}
The idea is to apply the argument of the proof of Proposition \ref{prop:ssm-scale-Rm-norm} to the convolutions $\mu_{m,\lam} * S_{\lam}\mu$. Since, thanks to \eqref{eq:Lambda-size-small} and \eqref{eq:decomposition-ssm-mu}, $\mu$ is the sum of a sub-exponential number of such measures, the proof will go through with minor changes.

Recall that $\mu_m = \sum_{u\in\Omega_m} p_u\delta(t_u)$, so that $\mu_m= \sum_{\lam\in\Lambda_m} \mu_{m,\lam}$. Note also that, by the self-similarity identity \eqref{eq:decomposition-ssm-mu}, we have $\mu=\sum_{u\in\Omega_m} p_u (\delta(t_u) * S_{\lam_u}\mu)$. Since $S_{\lam_u}\mu$ is supported on an interval $[-\Theta(2^{-m}),\Theta(2^{-m})]$, an application of Lemma \ref{lem:Holder} yields
\begin{equation} \label{eq:eq-norm-discr-general}
\|\mu_m^{(m)}\|_q^q =\Theta_q(1) \|\mu^{(m)}\|_q^q.
\end{equation}
Using \eqref{eq:decomposition-ssm-mu}, the H\"{o}lder bound $\|\sum_{j\in\Lambda}\nu_j\|_q^q \le |\Lambda|^{q-1}\sum _{j\in\Lambda}\|\nu_j\|_q^q$ and Lemma  \ref{lem:discr-norm-conv-equivalence}, we get
\[
\|\mu^{((R+1)m)}\|_q^q \le O(1) |\Lambda_m|^{q-1} \sum_{\lam\in\Lambda_m} \| \mu_{m,\lam}^{((R+1)m)} * S_{\lam}\mu^{((R+1)m)} \|_q^q.
\]
Let $\mu_{m,\lam,I}$ be the normalized restriction of $\mu_{m,\lam}$ to $I$. Note that, for fixed $\lam\in\Lambda_m$, the family of supports of $\mu_{m,\lam,I}^{((R+1)m)}* (S_{\lam}\mu)^{((R+1)m)}$ has covering number $O(1)$. Using this together with Lemma \ref{lem:L-q-norm-almost-disj-supports}, we deduce that
\[
\|\mu^{((R+1)m)}\|_q^q \le O(1) |\Lambda_m|^{q-1} \sum_{\lam\in\Lambda_m} \sum_{I\in\DD_m} \mu_{m,\lam}(I)^q \|\mu_{m,\lam,I}^{((R+1)m)} * (S_\lam\mu)^{((R+1)m)}\|_q^q.
\]
Let $\rho_{m,\lam,I}=S_{1/\lam}(\mu_{m,\lam,I})$. Using that $\lambda=\Theta(2^{-m})$ for $\lambda\in\Lambda_m$ together with Lemmas \ref{lem:Holder} and \ref{lem:discr-norm-conv-equivalence}, we see that for each $\lambda\in\Lambda_m$ and $I\in\DD_m$ we have
\[
\|\mu_{m,\lam,I}^{((R+1)m)} * (S_\lam\mu)^{((R+1)m)}\|_q^q = O_q(1) \| \rho_{m,\lam,I}^{(Rm)} * \mu^{(Rm)}\|_q^q.
\]
We deduce from the last two displayed equations that there is $\lambda^*=\lambda^*(m)\in\Lambda_m$ such that
\[
 \|\mu^{((R+1)m)}\|_q^q \le O(1) |\Lambda_m|^q  \sum_{I\in\DD_m} \mu_{m,\lam^*}(I)^q \|\rho_{m,\lam^*,I}^{(Rm)} * \mu^{(Rm)}\|_q^q.
\]
Fix $\sigma>0$, and let $\DD'=\{ I\in\DD_m: \|\rho_{m,\lam^*,I}^{(Rm)}\|_q^q\le 2^{-\sigma m}\}$. By the analog of Theorem \ref{thm:conv-with-ssm-flattens} in our context, there exists $\e=\e(\sigma,q) \in (0,\sigma)$ such that (for $m$ large enough)
\[
I\in\DD' \Longrightarrow \|\rho_{m,\lam^*,I}^{(Rm)}*\mu^{(Rm)}\|_q^q\le 2^{-\e  m} \|\mu^{(Rm)}\|_q^q.
\]
Combining the last two displayed equations with the bound $O(1)|\Lambda_m|^q\le 2^{-\e m/2}$ valid for $m\gg_{q,\e} 1$, we get
\begin{align*}
 \|\mu^{((R+1)m)}\|_q^q  &\le  2^{\e m/2} 2^{-\e m} \|\mu^{(Rm)}\|_q^q \sum_{I\in\DD'} \mu_{\mu,\lam^*}(I)^q  +  2^{\e m/2}  \|   \mu^{(Rm)} \|_q^q \sum_{I\notin \DD'}  \mu_{m,\lam^*}(I)^q  \\
&\le    \|\mu^{(R m)}\|_q^q \left( 2^{\e m/2} 2^{-\e m} \|\mu^{(m)}\|_q^q   +  2^{\e m/2} \sum_{I\notin\DD'} \mu_{m,\lam^*}(I)^q \right).
\end{align*}
Since $-\tfrac 1m \log\|\mu^{(m)}\|_q^q$ converges (to $\tau(q)$), we know that
\[
 2^{\e m/2} 2^{-\e m}  \|\mu^{(R m)}\|_q^q  \|\mu^{(m)}\|_q^q \le \tfrac12 \|\mu^{((R+1)m)}\|_q^q
\]
for large enough $m$, and therefore (using $\e<\sigma$)
\[
\sum_{I\notin\DD'} \mu_{m,\lam^*}(I)^q \ge \tfrac{1}{2} 2^{-\e m/2}  \|  \mu^{(Rm)} \|_q^{-q}  \|\mu^{((R+1)m)}\|_q^q  \ge 2^{-\sigma m} 2^{-\tau(q)m}.
\]
On the other hand, similarly to \eqref{eq:estimate-mu-m-Rp1}, we can apply the pointwise inequality $\mu_m\ge \mu_{m,\lam^*}$ and then Lemma \ref{lem:Holder} to conclude (always assuming $m$ is large enough)
\begin{align*}
 \|\mu_{m}^{((R+1)m)}\|_q^q &\ge \|\mu_{m,\lam^*}^{((R+1)m)}\|_q^q  \ge\Omega(1) \sum_{I\in\DD_m} \mu_{m,\lam^*}(I)^q \| \rho_{m,\lam^*,I}^{(Rm)} \|_q^q \\
 &\ge \sum_{I\notin\DD'} \mu_{m,\lam^*}(I)^q  2^{-\sigma m} \ge 2^{-(\tau(q)+2\sigma) m}.
\end{align*}
The opposite inequality
\[
 \|\mu_{m}^{((R+1)m)}\|_q^q \le \|\mu_m^{(m)}\|_q^q \le O(1) \|\mu_m\|_q^q \le 2^{-(\tau(q)-\sigma)m}
\]
holds for large enough $m$ by \eqref{eq:eq-norm-discr-general} so, since $\sigma>0$ was arbitrary, this concludes the proof.
\end{proof}

\begin{proof}[Proof of Proposition \ref{prop:dim-general-ssm}]
It is enough to prove the statement for $q$ such that $\tau$ is differentiable at $q$.  Iterating the definition of $\wt{\tau}(q)$, we see that
\[
\sum_{u\in\Omega_m} p_u^q \lam_u^{-\wt{\tau}(q)} =1.
\]
Since $\lam_u \in (c 2^{-m},2^{-m})$ for $u\in\Omega_m$ and a constant $c>0$ depending only on the IFS, we deduce that
\begin{equation} \label{eq:symbolic-Lq-spectrum-general-ssm}
\wt{\tau}(q) = \lim_{m\to\infty} \frac{-\log\left(\sum_{u\in\Omega_m} p_u^q\right) }{m}.
\end{equation}

By the exponential separation assumption, there exist $R\in\N$ and a sequence $m_j\to\infty$ such that, for fixed $\lam\in\Lambda_{m_j}$, the distance between any two distinct atoms of $\mu_{m,\lam}$ is at least $2^{-Rm_j}$. Hence, by \eqref{eq:decomposition-ssm-mu} and H\"{o}lder's inequality,
\begin{equation} \label{eq:norm-at-smaller-scale-equals-symbolic-norm}
\|\mu_{m_j}^{(Rm_j)}\|_q^q \le |\Lambda_{m_j}|^{q-1} \sum_{\lam\in\Lambda_{m_j}} \|\mu_{m_j,\lam}^{(Rm)}\|_q^q
= |\Lambda_{m_j}|^{q-1} \sum_{u\in\Omega_{m_j}} p_u^q.
\end{equation}
On the other hand, one always has
\begin{equation} \label{eq:norm-at-smaller-scale-equals-symbolic-norm-2}
 \|\mu_m^{(Rm)}\|_q^q \ge \|\mu_m\|_q^q \ge \sum_{u\in\Omega_m} p_u^q.
\end{equation}

Combining Proposition \ref{prop:general-ssm-smaller-scale} and Equations \eqref{eq:Lambda-size-small}, \eqref{eq:symbolic-Lq-spectrum-general-ssm}, \eqref{eq:norm-at-smaller-scale-equals-symbolic-norm} and \eqref{eq:norm-at-smaller-scale-equals-symbolic-norm-2} yields the claimed equality $\tau(q) = \wt{\tau}(q)$.
\end{proof}

This concludes the proof of Theorem \ref{thm:dim-general-ssm}.

\section{Convolutions of self-similar measures and the proof of Theorem \ref{thm:Furstenberg}}
\label{sec:convolutions-and-Furst-conj}

\subsection{Convolutions of two self-similar measures and Furstenberg's conjectures}

\label{subsec:proof-of-Furst-conjecture}

 We turn to convolutions of homogeneous self-similar measures, and deduce Theorem \ref{thm:Furstenberg} as a corollary. As we observed in \S\ref{subsec:DDSSM}, the convolutions of the natural measures on a $p$-Cantor set and a $q$-Cantor set fit naturally into the setting of dynamically driven self-similar measures. The same argument works in greater generality:

\begin{lemma} \label{lem:mu-x-conv-ssm}
Let $0<\lam_2<\lam_1<1$ and $\Delta_1,\Delta_2\in\mathcal{A}$, and consider the self-similar measures
\begin{equation} \label{eq:def-ssm-for-convolution}
\eta_i = \eta_i(\Delta_i,\lam_i)= *_{n=0}^\infty S_{\lam_i^n}\Delta_i.
\end{equation}

Write $a_i=|\log(\lam_i)|$. On $X=[0,a_2)$, define the map
\[
\mathbf{T}(x) = x+a_1 \bmod(a_2).
\]
Moreover, let $\Delta:X\to \mathcal{A}$ be given by
\[
\Delta(x) = \left\{
\begin{array}{ll}
  \Delta_1* S_{e^x}\Delta_2 & \text{if } x\in [0,a_1)\\
  \Delta_1 & \text{if } x\in [a_1,a_2)
\end{array}
\right..
\]
Then if $\mu_x$ is given by \eqref{eq:def-mu-x} with $\lam=\lam_1$, we have
\[
\mu_x = \left\{
\begin{array}{ll}
\eta_1 * S_{e^x}\eta_2 & \text{ if } x\in [0,a_1)\\
\eta_1 * S_{e^{x-a_2}}\eta_2 & \text{ if } x\in [a_1,a_2)
\end{array}
\right..
\]
for all $x\in X$.
\end{lemma}
\begin{proof}
Let $n'(x)=|\{i\in[1,n]: \mathbf{T}^i(x)\in [0,a_1)\}|$. Then $\mathbf{T}^n(x)=x+n a_1-n'(x)a_2$, so that $e^{\mathbf{T}^n(x)}\lam_1^n = e^x \lam_2^{n'(x)}$, and therefore
\begin{equation} \label{eq:model-for-conv-ssm}
*_{i=1}^n \Delta(\mathbf{T}^i x) =  \left(*_{i=1}^n S_{\lam_1^i}\Delta_1\right)  *  S_{e^x} \left( *_{i=1}^{n'(x)} S_{\lam_2}^i\Delta_2 \right).
\end{equation}
The claim follows by convolving with $\Delta(x)$ to get $\mu_{n+1,x}$, and then letting $n\to\infty$.
\end{proof}

\begin{thm} \label{thm:dim-conv-ssm}
Let $\eta_1,\eta_2$ be as in \eqref{eq:def-ssm-for-convolution}. Assume $\log\lam_2/\log\lam_1\notin\mathbb{Q}$. Moreover, suppose that there is $R>0$ such that for infinitely many $n$ and all $P_j\in\mathcal{P}_{\Delta_j,n}$ (recall  Definition \ref{def:exponential-separation-hom-ssm}), $j=1,2$  it holds that
\[
|P_1(\lam_1)|,|P_2(\lam_2)| \ge \lam_1^{Rn}.
\]
Then
\begin{equation} \label{eq:Lq-dim-conv-ssm}
D(\eta_1*\eta_2,q) = \min\left(D(\eta_1,q)+D(\eta_2,q),1\right)
\end{equation}
for all $q\in (1,\infty)$.
\end{thm}
\begin{proof}

Let $(X,\mathbf{T},\Delta,\lam_1)$ be the model given by Lemma \ref{lem:mu-x-conv-ssm}. We identify $X$ with the circle (i.e. we identify $0$ and $\log\lam_1$), so that the $X$ becomes compact, and $\mathbf{T}$ is rotation by $\log \lam_1/\log\lam_2$ (which is irrational by assumption) on the circle. Hence $\mathbf{T}$ is uniquely ergodic (with the unique invariant measure $\mathbb{P}$ being normalized Lebesgue measure on $X$). If $\Delta_1$ and $\Delta_2$ are supported on a single point, then $\mu_x$ is an atom for all $x$ and there is nothing to do; otherwise, $\mu_x$ is non-atomic for all $x$. Finally, the map $x\mapsto \mu_x$ has a single discontinuity at $a_1$, as is evident from Lemma \ref{lem:mu-x-conv-ssm}. We have then checked that the model is pleasant. The assumptions on $x\mapsto \Delta(x)$ in Theorem \ref{thm:L-q-dim-dyn-ssm} also hold trivially.

We claim that our assumption on the separation of $\eta_1,\eta_2$ implies that our model has exponential separation. Let
\[
\Delta_{n,j} = \sum_{i=0}^{n-1} S_{\lam_j^i}(\Delta_j) = \left\{ \sum_{i=0}^{n-1} y_i \lam_j^i : y_i\in\Delta_j \right\}.
\]
Recall from \eqref{eq:model-for-conv-ssm} that all atoms of $\mu_{x,n}$ have the form
\[
\{ u_1 + e^x u_2 : u_1\in\Delta_{n,1}, u_2\in\Delta_{n,2}\}.
\]
Thus, for given $x\in X$, the smallest distance between atoms of $\mu_{x,n}$ is bounded above by
\[
\Phi_n(x) = \min\{ |P_1(\lam_1)|, |e^x P_2(\lam_2)|, |P_1(\lam_1) - e^x P_2(\lam_2)| : P_1\in \mathcal{P}_{\Delta_1,n}, P_2\in\mathcal{P}_{\Delta_2,n} \}.
\]
Here $|P_1(\lam_1)|$ corresponds to differences between pairs of atoms for which $u_2$ coincide, $|e^x P_2(\lam_2)|$ to pairs of atoms for which $u_1$ coincide, and $|P_1(\lam_1) - e^x P_2(\lam_2)|$ to pairs of atoms for which neither $u_1$ nor $u_2$ coincide. By assumption, $|P_j(\lam_j)|\ge \lam_1^{Rn}$ for infinitely many $n$, so we only have to deal with the third type of differences. Fix, then,  $n$ such that $|P_j(\lam_j)|\ge \lam_1^{Rn}$ for all $P_j\in \mathcal{P}_{\Delta_j,n}$.

Let $R'\gg R$. For fixed $P_j\in \mathcal{P}_{\Delta_j,n}$,
\[
|\{x:  |P_1(\lam_1) - e^x P_2(\lam_2)| \le \lam_1^{R' n}\}| \le O_{\Delta_1,\Delta_2}(1) \lam_1^{(R'-R)n}.
\]
Since $| \mathcal{P}_{\Delta_j,n}|\le O_{|\Delta_j|}(1)^n$, we deduce that
\[
|\{x:  |P_1(\lam_1) - e^x P_2(\lam_2)| \le \lam_1^{R' n} \text{ for some } P_j\in \mathcal{P}_{\Delta_j,n} \} |  \le O_{\Delta_1,\Delta_2}(1)^n \lam_1^{(R'-R)n}.
\]
Hence, if $R'$ is taken large enough (in terms of $R, \Delta_1, \Delta_2$ only), then there are infinitely many $n\in\N$ such that for almost all $x\in X$ it holds that $|P_1(\lam_1) - e^x P_2(\lam_2)|\ge \lam_1^{R' n}$ for any choice of $P_j\in\mathcal{P}_{\Delta_j,n}$. This establishes exponential separation.

We have verified that the application of Theorem \ref{thm:L-q-dim-dyn-ssm} is justified. In light of this theorem, we only need to check that the right-hand side in \eqref{eq:Lq-dim-conv-ssm} equals the right-hand side in \eqref{eq:main-thm}. Note that
\[
\|\Delta_1 * S_{e^x}\Delta_2\|_q^q = \|\Delta_1\|_q^q \|\Delta_2\|_q^q
\]
outside of a finite set of $x$. Hence, keeping in mind the definition of the map $\Delta$ from Lemma \ref{lem:mu-x-conv-ssm},
\[
\int_X \log\|\Delta(x)\|_q^q \,d\mathbb{P}(x) =  \log\|\Delta_1\|_q^q + \frac{\log(\lam_1)}{\log(\lam_2)}\log\|\Delta_2\|_q^q.
\]
Dividing by $(q-1)\log(\lam_1)$ we get that
\[
D(\eta_1*\eta_2,q) = \min\left(\frac{\log\|\Delta_1\|_q^q}{(q-1)\log(\lam_1)}+\frac{\log\|\Delta_2\|_q^q}{(q-1)\log(\lam_2)}, 1\right).
\]
Theorem \ref{thm:dim-ssm} applied to $\eta_1$ and $\eta_2$ concludes the proof.
\end{proof}

We point out that in the range $q\in (1,2]$, the above result was proved in \cite{NPS12} in some special cases and then, extending the same ideas, in \cite[Corollary 6.2]{GSSY16}, in even greater generality. For example, in \cite{GSSY16} no separation assumptions are made on $\eta_1,\eta_2$. However, the methods of \cite{NPS12, GSSY16} ultimately rely on Marstrand's projection theorem, which is known to fail in general if $q>2$.

As a corollary, we obtain a Furstenberg-like bound on the intersections of self-similar sets, which also answers affirmatively a question of De-Jun Feng.
\begin{cor} \label{cor:dim-intersection-ssm}
Let $\lam_1,\lam_2\in (0,1)$ with $\log\lam_1/\log\lam_2\notin\mathbb{Q}$. Suppose $E_1, E_2$ are finite sets such that $\{ \lam_j x + t: t\in E_j\}$ satisfies the open set condition for $j=1,2$. Let $A_1, A_2$ denote the corresponding self-similar sets.

Then for all invertible affine maps $g:\R\to\R$,
\[
\ubdim(A_1 \cap g(A_2)) \le \max(\hdim(A_1)+\hdim(A_2)-1,0).
\]
\end{cor}
\begin{proof}
Let $\eta_i$ be the uniform self-similar measure on $A_i$, and write $\mu=\eta_1\times\eta_2$ and $s=\hdim(A_1)+\hdim(A_2)$. Then $\mu(B(x,r)) = \Theta(r^s)$ for $x\in\supp(\mu)$, since the corresponding fact holds for $\eta_1,\eta_2$ thanks to the open set condition.

As rescaling $A_2$ does not change the assumptions, it is enough to prove the claim when $g$ is a translation. Let $\Delta_j$ be the uniform probability measure on $E_j$, and $\eta_j=\eta_j(\lam_j,\Delta_j)$ the associated self-similar measure. The hypotheses of Theorem \ref{thm:dim-conv-ssm} are met, so we know that
\[
D(\eta_1 * \eta_2,q) = \min\left(\frac{\log|E_1|}{\log(1/\lam_1)}+\frac{\log|E_2|}{\log(1/\lam_2)}, 1\right) = \min(s,1)
\]
for all $q>1$. The claim now follows from Lemmas \ref{lem:Lq-dim-to-Frostman-exp} and \ref{lem:Frostman-exp-to-small-fiber} applied to the function $(x,y)\mapsto x-y$ restricted to $A_1\times A_2$.
\end{proof}

We can now finish the proof of Theorem \ref{thm:Furstenberg}.
\begin{proof}[Proof of Theorem \ref{thm:Furstenberg}]
Let $A, B$ be $T_p$-invariant and $T_q$-invariant respectively, with $p$ and $q$ multiplicatively independent, and fix $\delta>0$. Given $N\in\N$, let
\[
E_{A,N} = \{ j p^{-N}: A\cap [j p^{-N},(j+1) p^{-N})\neq\varnothing  \},
\]
and define $E_{B,N}$ likewise. It is well known that Hausdorff and box-counting dimensions coincide for $T_p, T_q$-invariant sets, see e.g. \cite[Theorem 5.1]{Furstenberg08} for a more general fact. Hence by taking $N$ large enough we can ensure that
\[
|E_{A,N}| \le p^{N(\hdim(A)+\delta)}, \quad |E_{B,N}| \le q^{N(\hdim(B)+\delta)}.
\]
Let $A'$ be the homogeneous self-similar set with contraction $p^{-N}$ and translation set $E_{A,N}$, and define $B'$ analogously. The open set condition holds for $A',B'$ with open set $(0,1)$. Then
\[
 \hdim(A') = \frac{\log|E_{A,N}|}{\log p^N} < \hdim(A)+\delta,
\]
and likewise for $B'$. Also, by invariance of $A,B$ under $T_{p^N}, T_{q^N}$ respectively, $A\subset A', B\subset B'$. (Symbolically, $E_{A,N}$ corresponds to all initial words of length $N$ in $A$, and $A'$ to all concatenations of such words).

Since $\delta>0$ was arbitrary, the theorem follows from Corollary \ref{cor:dim-intersection-ssm} applied to $A',B'$.
\end{proof}

Corollary \ref{cor:dim-intersection-ssm} and Theorem \ref{thm:Furstenberg} remain valid for $C^1$ maps $g$. It is not hard to deduce this from the affine case and Furstenberg's theory of CP-processes \cite{Furstenberg08}, but since it would take us too far in a different direction, we defer a detailed proof of these and related results to a forthcoming article.

Recall from the introduction that another conjecture of Furstenberg, settled in \cite{HochmanShmerkin12}, concerns the dimension of the arithmetic sum of a $\times p$ and a $\times q$ invariant set. As a corollary, we are able to sharpen this when the sum of the dimensions is at most $1$:
\begin{cor} \label{cor:sums-of-invariant-sets}
Let $p,q$ be multiplicatively independent, and suppose that $A,B\subset [0,1)$ are closed and $T_p,T_q$-invariant, respectively. Assume $\hdim(A)+\hdim(B)\le 1$. Then for any subsets $A'\subset A, B'\subset B$,
\[
 \hdim(A'+B') = \hdim(A'\times B')
\]
\end{cor}

We note that in general $\hdim(A'\times B')  \ge \hdim(A')+\hdim(B')$ and the inequality can be strict, but there is an equality if either $A'$ or $B'$ have equal Hausdorff and upper box-counting dimensions.

\begin{proof}[Proof of Corollary \ref{cor:sums-of-invariant-sets}]
Suppose first that $\hdim(A)+\hdim(B)<1$. By embedding $A,B$ in $p^N, q^N$-Cantor sets of almost the same dimension as in the proof of Theorem \ref{thm:Furstenberg}, we may assume that $A,B$ are already a $p, q$-Cantor set respectively. The proof is now nearly identical to that of Corollary \ref{cor:sierpinski-gasket}, using Theorem \ref{thm:dim-conv-ssm} in place of Theorem \ref{thm:dim-ssm}.

If $\hdim(A)+\hdim(B)=1$, then we proceed in the same way but now the sums of the dimensions of the $p,q$-Cantor sets containing $A,B$ is $1+\delta$, where $\delta$ is arbitrarily small. The argument of  Corollary \ref{cor:sierpinski-gasket} still goes through with very minor modifications; details are left to the interested reader.
\end{proof}

A minor variant of the same argument recovers the full conjecture of Furstenberg on sums of $T_p$ and $T_q$ invariant sets. However, apart from some special cases, the methods from this paper do not appear to yield a different proof of the corresponding statement for convolutions of invariant measures, recall \eqref{eq:dim-conv-invariant-measures}.

\subsection{Convolutions of several self-similar measures}

Theorem \ref{thm:dim-conv-ssm} generalizes easily to convolutions of an arbitrary number of self-similar measures. This provides an example of application of Theorem \ref{thm:L-q-dim-dyn-ssm} in which $X$ is a torus of arbitrary dimension.

\begin{thm} \label{thm:dim-conv-many-ssm}
Let $0<\lam_1<\ldots<\lam_k<1$, $k\ge 2$, be numbers such that $(1/\log\lam_j)_{j=1}^k$ is linearly independent over $\Q$.  Fix $\Delta_1,\ldots,\Delta_k\in\mathcal{A}$, and write
\[
\eta_j =\eta_j(\Delta_j,\lam_j) = *_{n=0}^\infty S_{\lam_j^n} \Delta_j
\]
for the corresponding self-similar measures. Moreover, suppose that there is $R>0$ such that for infinitely many $n$ it holds that
\begin{equation} \label{eq:exponential-sep-component-measures}
|P_j(\lam_j)| \ge \lam_k^{Rn}\quad \text{for all } P_j\in\mathcal{P}_{\Delta_j,n}, j=1,\ldots,k.
\end{equation}
Then
\begin{equation*} %\label{eq:Lq-dim-conv-many-ssm}
D(\eta_1*\cdots*\eta_k,q) = \min\left(\sum_{j=1}^k D(\eta_j,q),1\right).
\end{equation*}
for all $q\in (1,\infty)$.
\end{thm}
\begin{proof}
The proof is similar to that of Theorem \ref{thm:dim-conv-ssm}, so we will skip some details. We write $a_j=|\log(\lam_j)|$.  Let $X=[0,a_1)\times\cdots\times [0,a_{k-1})$, and let $\mathbf{T}:X\to X$ be given by
\[
\mathbf{T}(x_1,\ldots,x_{k-1}) = (x_1+a_k \bmod a_1,\ldots, x_{k-1}+a_k \bmod a_{k-1}).
\]
Up to re-parametrization, this is translation by $(a_k/a_1,\ldots,a_k/a_{k-1})$ on the $(k-1)$-torus, which is uniquely ergodic if (and only if) $(1,a_k/a_1,\ldots,a_k/a_{k-1})$ is linearly independent over $\Q$; see e.g. \cite[Corollary 4.15]{EinsiedlerWard11}. An easy calculation using the linear independence of $1/\log \lam_j$ shows that this is indeed the case.

Given $x\in X$, we let $J(x)=\{ j\in\{1,\ldots,k-1\}: x_j \in [0,a_k)\}$, and define $\Delta:X\to \mathcal{A}$ as
\[
\Delta(x) = \left( *_{j\in  J(x)} S_{\exp(x_j)} \Delta_j \right) * \Delta_k .
\]
We have already remarked that $(X,\mathbf{T})$ is uniquely ergodic. The same argument from Lemma \ref{lem:mu-x-conv-ssm} shows that the measures  generated by this model are
\[
\mu_x = S_{\exp(x_1-\mathbf{1}(x_1\in [0,a_k))a_1)}\eta_1 * \cdots  * S_{\exp(x_{k-1}-\mathbf{1}(x_1\in [0,a_k)) a_{k-1})}\eta_{k-1}* \eta_k.
\]
The model $(X,\mathbf{T},\Delta,\lam_k)$ is now readily checked to be pleasant, while the map $\Delta(\cdot)$ also meets the hypotheses in Theorem \ref{thm:L-q-dim-dyn-ssm}.

To establish exponential separation, we notice that the difference between two atoms of $\mu_{x,n}$ has the form
\[
\sum_{j=1}^k  s_j e^{x_j} P_j(\lam_j),
\]
where $s_j\in\{0,1\}$, not all $s_j$ are zero, $P_j\in \mathcal{P}_{\Delta_j,n}$, and we set $x_k=0$. For $n$ such that \eqref{eq:exponential-sep-component-measures} holds, the same argument in the proof of Theorem \ref{thm:dim-conv-ssm}, together with Fubini and an induction on the number of non-zero $s_j$, shows that the distance between atoms of $\mu_{n,x}$ is at least $\lam_k^{R' n}$ for a.e. $x$, where $R'$ depends on $R$, the $\Delta_i$ and $k$ only.

We have checked that Theorem \ref{thm:L-q-dim-dyn-ssm} can be applied. A calculation like the one in the proof of Theorem \ref{thm:dim-conv-ssm} yields
\[
\int_X \log\|\Delta(x)\|_q^q \,d\PP(x) = \log\|\Delta_k\|_q^q + \sum_{j=1}^{k-1} \frac{\log \lam_k}{\log\lam_j} \log\|\Delta_j\|_q^q,
\]
so that Theorems \ref{thm:L-q-dim-dyn-ssm} and  \ref{thm:dim-ssm} yield the desired conclusion.
\end{proof}

\subsection{Embeddings of self-similar sets}

Let us denote by $A_\lam$ any self-similar set arising from a homogeneous IFS with contraction ratio $\lam$, satisfying the open set condition and of dimension strictly smaller than $1$. A special case of a conjecture of D-J. Feng, W. Huang and H. Rao \cite[Conjecture 1.2]{FHR14} asserts that $A_\lam$ cannot be affinely embedded into $A_{\lam'}$ unless $\log\lam/\log\lam'\in\Q$. In \cite{FHR14} this is proved in some special cases, and some further new cases were recently established by A. Algom \cite{Algom18}. However the general case was not known even for central Cantor sets (i.e. self-similar sets generated by two maps). It follows immediately from Corollary \ref{cor:dim-intersection-ssm} that if $\log\lam/\log\lam'\notin\Q$, then for every affine map $h:\R\to\R$,
\[
\hdim(A_\lam \cap h(A_{\lam'})) \le \hdim(A_\lam)+\hdim(A_{\lam'})-1 <  \min(\hdim(A_\lam),\hdim(A_{\lam'})),
\]
so that no affine immersion is possible. We can easily extend this to the case in which the set we want to embed is an arbitrary non-trivial self-similar set:

\begin{cor}
Suppose $A=\bigcup_{i\in\mathcal{I}} \lambda_i A+t_i$, $B=\bigcup_{j\in\mathcal{J}} \lam' B+t'_j$ are self-similar sets, with $A$ not a singleton, and $B$ homogeneous, satisfying the open set condition, and of dimension strictly smaller than $1$. If there is a $C^1$ map $h:\R\to\R$ such that $h(A)\subset B$, then $\log \lam_i/\log\lam'$ is rational for all $i$.
\end{cor}
\begin{proof}
Suppose that, on the contrary, $\log \lam_i/\log\lam'$ is irrational for some $i$, and yet $h(A)\subset B$ for some $C^1$ map $h$. Without loss of generality, assume that $\log\lam_1/\log\lam'$ is irrational. We may also assume that, writing $f_j(x)=\lam_j x+t_j$, the fixed points of $f_1$ and $f_2$ are different (if all the $f_j$ had the same fixed point, then $A$ would equal this point). If $N$ is sufficiently large, then $(f_{2}f_{1}^N, f_{1}^N f_{2})$ is a homogeneous IFS satisfying the open set condition, and its attractor $A_N$ is contained in $A$, so that $h(A_N)\subset A$. On the other hand, if $\log(\lam_2\lam_1^N)/\log(\lam')$ is rational then, by our assumption, $\log(\lam_2\lam_1^{N+1})/\log(\lam')$ is irrational.

We have thus reduced the problem to the case of $A$ homogeneous and satisfying the open set condition. Under these assumptions, \cite[Theorem 1.1]{FHR14} implies that there is an \emph{affine} embedding of $A$ into $B$. But, as we have seen, this is ruled out by  Corollary \ref{cor:dim-intersection-ssm}.
\end{proof}

\section{Sections and projections of planar self-similar sets}
\label{sec:planar-sss}

Our next geometric application involves homogeneous self-similar sets and measures on the plane. It was observed in several previous works, going back at least to \cite{PeresShmerkin09}, that methods devised to study geometric properties of cartesian products of linear self-similar sets and measures often can also be applied to the study of self-similar sets and measures on the plane. The next lemma may help clarify the reason behind this; compare with Lemma \ref{lem:mu-x-conv-ssm}.

\begin{lemma} \label{lem:model-proj-hom-ssm}
Fix $\alpha\in [0,2\pi)$, $\lam\in (0,1)$ and a finitely supported probability measure $\wt{\Delta}=\sum_{i\in\mathcal{I}} p_i \delta(t_i)$ on $\R^2$. Denote rotation by $\alpha$ by $\mathbf{R}_\alpha$, and let
\[
 \mu = *_{n=0}^\infty S_{\lam^n}\mathbf{R}_\alpha^n(\wt{\Delta})
\]
be the associated homogeneous self-similar measure. Given $x\in S^1$, let $P_x(y) = \langle x,y\rangle$ be the orthogonal projection onto a line in direction $x$. Furthermore, let $\Delta(x)=P_x\wt{\Delta}$.

Then the measures $\mu_x$ generated by the model $(S^1,\mathbf{R}_{-\alpha}, \Delta,\lam)$ are the projections $P_x\mu$. Moreover, the model is pleasant if and only if $\alpha/\pi\notin\Q$.

\end{lemma}
\begin{proof}
 Immediate, since
 \[
 \left\langle x, S_{\lam^n} \mathbf{R}_\alpha^n(y) \right\rangle = \left\langle \mathbf{R}_{-\alpha}^n x,  S_{\lam^n} (y) \right\rangle,
 \]
 and rotation by $\beta\in [0,2\pi)$ is uniquely ergodic if and only if $\beta/\pi$ is irrational.
\end{proof}

\begin{thm} \label{thm:dim-proj-planar-ssm}
 Let $\mu$ and $P_x$ be as in Lemma \ref{lem:model-proj-hom-ssm}. Assume further that $\alpha/\pi\notin\Q$, and that the open set condition holds. Then for every $x\in S^1$ and every $q\in (1,\infty)$,
 \[
  D(P_x\mu,q) = \min\left( \frac{\log\|\wt{\Delta}\|_q^q}{(q-1)\log\lam},1 \right).
 \]
\end{thm}
\begin{proof}
 Let $\mu_n = *_{i=0}^{n-1}  S_{\lam^n}\mathbf{R}_\alpha^n(\wt{\Delta})$. By the open set condition, $\mu_n$ has $|\mathcal{I}|^n$ atoms, which are $c \lam^n$-separated for some $c>0$. Note that $P_x(\mu_n)=\mu_{x,n}$ (the measures generated by the model from Lemma \ref{lem:model-proj-hom-ssm}). In particular, the atoms of $\mu_{x,n}$ are the projections of the atoms of $\mu_n$.

 Let $R$ be a large enough integer to be chosen later. By elementary geometry, for a given pair $a,b$ of distinct atoms of $\mu_n$, the set of $x\in S^1$ such that $|P_x a - P_x b|\le \lam^{Rn}$ has measure $O_c(\lam^{(R-1)n})$. Hence, the set of $x\in S^1$ such that the atoms of $\mu_{n,x}$ are all distinct and $\lam^{Rn}$-separated has measure $1-O_c(|\mathcal{I}|^{2n}) \lam^{(R-1)n}$. This implies that if $R$ is taken large enough in terms of $|\mathcal{I}|$, then for almost all $x\in S^1$ there is $n_0=n_0(x)$ such that the atoms of $\mu_{n,x}$ are distinct and $\lam^{Rn}$ separated for all $n\ge n_0$. Hence the model from Lemma \ref{lem:model-proj-hom-ssm} has exponential separation.

 Since the hypothesis on $\Delta$ is trivially satisfied, we can apply Theorem \ref{thm:L-q-dim-dyn-ssm} to conclude that
 \[
  D(P_x\mu,q) = \min\left( \frac{\int \log \|P_x\wt{\Delta}\|_q^q\,dx}{(q-1)\log \lam},1\right),
 \]
 which gives the claim since $P_x$ is injective on $\wt{\Delta}$ for all but a finite set of $x$.
\end{proof}

We obtain the following corollary on linear sections of planar self-similar sets; compare with Corollary \ref{cor:dim-intersection-ssm}.
\begin{cor} \label{cor:dim-intersection-sss-lines}
 Fix $\lam\in (0,1)$, $\alpha\in [0,2\pi)$ such that $\alpha/\pi$ is irrational, and a finite set $(t_i)_{i\in\mathcal{I}}$ of translations in $\R^2$. Assume that the IFS $\{ \lam \mathbf{R}_\alpha(x)+t_i\}_{i\in\mathcal{I}}$ satisfies the open set condition, and denote its invariant set by $E$.

 Then
 \[
  \ubdim(E\cap \ell) \le \max(\hdim(E)-1,0).
 \]
 for all lines $\ell\subset\mathbb{R}^2$.
\end{cor}
\begin{proof}
 Immediate from Lemmas \ref{lem:Lq-dim-to-Frostman-exp} and \ref{lem:Frostman-exp-to-small-fiber} applied to $P_x$, and Theorem \ref{thm:dim-proj-planar-ssm} applied to the uniform self-similar measure on $E$.
\end{proof}

We make some remarks about this corollary.
\begin{enumerate}[label={\upshape\alph*)}]
\item Let $E$ be any Borel set with $\hdim(E)\ge 1$. It follows from Marstrand's intersection theorem (see e.g. \cite[Theorem 10.10]{Mattila95}) that, given a direction $x\in S^1$, \emph{almost all} lines $\ell$ in direction $x$ satisfy $\hdim(E\cap\ell) \le \hdim(E)-1$. There has been great interest in improving almost all-type of results for classes of natural sets, but most of the progress achieved concerns projections rather than the more subtle problem of intersections. For some classes of \emph{random} stochastically self-similar sets, even stronger bounds on intersections were obtained in \cite[Section 11]{ShmerkinSuomala18}. D-J. Feng has some unpublished results for deterministic sets, using ad-hoc constructions.  To the best of our knowledge, Corollary \ref{cor:dim-intersection-sss-lines} is the first result of this kind for a natural class of deterministic sets.
\item It is also natural to consider the dual question of obtaining \emph{lower} bounds on the dimension of $E\cap\ell$ for lines $\ell\subset\R^2$ when $\hdim(E)>1$. Of course, many such intersections are empty, but one would like to know that the intersections are large (of dimension equal or close to $\hdim(E)-1$) for \emph{many} lines $\ell$ in a given direction (measured, for example, in terms of Hausdorff dimension). Progress on this problem was achieved recently by K. Falconer and X. Jin \cite{FalconerJin15}.
\item Using Furstenberg's theory of CP-processes and galleries \cite{Furstenberg08}, it is possible to obtain a version of Corollary \ref{cor:dim-intersection-sss-lines} where lines are replaced by $C^1$ or even differentiable curves; we hope to address this at detail in a forthcoming paper. On the other hand, no such result can hold for Lipschitz curves since any set of upper box-counting dimension less than $1$ can be embedded in a Lipschitz curve.
\item The hypothesis that $\alpha/\pi$ is irrational is necessary: if $C$ is the middle-thirds Cantor set, then the diagonal of $C\times C$ is an affine copy of $C$. However, the homogeneity assumption is likely an artifact of the proof.
\end{enumerate}

\section{Absolute continuity and \texorpdfstring{$L^q$}{Lq} densities}
\label{sec:abs-cont}

We turn to the problem of absolute continuity, and smoothness of the densities, of self-similar and related measures. Compared to Sections \ref{sec:dim-ssm-and-applications}--\ref{sec:planar-sss}, our results here will be less explicit: we show that in many parametrized families, the measures have a density in $L^q$ for all parameters outside of some very small set.  In particular, we will establish Theorem \ref{thm:abc-cont-BCs}. Unfortunately, however, either for Bernoulli convolutions or the other parametrized families we consider, we do not know how to find even one explicit parameter which is not exceptional.

The main ideas in this section are borrowed from \cite{Shmerkin14, ShmerkinSolomyak16}; the reason we improve upon existing results is that Theorem \ref{thm:L-q-dim-dyn-ssm} provides stronger information about $L^q$ dimensions to begin with.

Recall that the Fourier transform of a Borel probability measure $\mu$ on $\R$ is defined as
\[
\widehat{\mu}(\xi) = \int \exp(2\pi i x\xi)\,d\mu(x).
\]
Given a model $(X,\mathbf{T},\Delta,\lam)$ and $k\in\N$, let us consider the measures
\[
\mu_x^{(k)} = *_{i=0}^\infty S_{\lam^k} \Delta(\mathbf{T}^{ki}(x)).
\]
These are precisely the measures generated by the model $(X,\mathbf{T}^k,\Delta,\lam^k)$, which is pleasant whenever the original model is; however we will not need to use this.

The next theorem presents our general result on densities of $\mu_{x}$. We will deduce several applications afterwards.
\begin{thm} \label{thm:abs-continuity}
Let $(X,\mathbf{T},\Delta,\lam)$ be a model satisfying the assumptions of Theorem \ref{thm:L-q-dim-dyn-ssm}, and assume furthermore that $X$ is either a singleton or infinite. Fix $q\in (1,+\infty)$ and assume also that
\[
 \frac{\int_X \log\|\Delta(x)\|_q^q \,d\mathbb{P}(x)}{(q-1)\log\lam} >  1.
\]
Suppose $y\in X$ is such that for infinitely many $k\in\N$ there exist constants $C(k),\delta(k)>0$ such that the Fourier transform of $\mu_y^{(k)}$ satisfies
\[
\left|\widehat{\mu_y^{(k)}}(\xi)\right| \le C(k)\, |\xi|^{-\delta(k)} \quad\text{for all }\xi\neq 0.
\]

Then $\mu_y$ is absolutely continuous and has a density in $L^q$.
\end{thm}
\begin{proof}
Using the convolution structure of $\mu_y$, we decompose
\[
\mu_y = \left(*_{k\nmid i} \Delta(\mathbf{T}^{i}y)\right)*\left(*_{k\mid i} \Delta(\mathbf{T}^{i}y)\right)=:\nu_y^{(k)}*\mu_y^{(k)}.
\]
If we can show that
\begin{equation} \label{eq:dim-skip-meas-1}
D(\nu_y^{(k)},q)=1
\end{equation}
for all large enough $k$, then \cite[Theorem 4.4]{ShmerkinSolomyak16}, together with our assumption on the Fourier decay $\mu_y^{(k)}$, will allow us to conclude that $\mu_y$ has a density in $L^q$.

For fixed $k$, consider the model $(X',\mathbf{T}',\Delta',\lam)$, where $X'=X\times [k]$, $\mathbf{T}'(x,j)=(\mathbf{T}x,j+1 \bmod k)$ and
\[
\Delta'(x,j) = \left\{
\begin{array}{ll}
  \Delta(x) & \text{ if } j\neq 0\\
  \delta_0 & \text{ if } j=0
\end{array}
\right..
\]
The measures $\mu'_x$ generated by this model are precisely $\nu_x^{(k)}$, as is immediate from the definition of $\Delta'$. This model satisfies all the assumptions in Theorem \ref{thm:L-q-dim-dyn-ssm}. Indeed, exponential separation is inherited from the base model, since the atoms of $\mu'_{x,n}$ are a subset of the atoms of $\mu_{x,n}$. Unique ergodicity (with invariant measure $\mathbb{P}'=\mathbb{P}\times (\tfrac1k \sum_j \delta_j)$) follows from the unique ergodicity of $(X,\mathbf{T})$ (note that there may be no periodic points, for otherwise the uniform measure on the orbit would be $\mathbf{T}$-invariant, contradicting unique ergodicity). The rest of the assumptions in Theorem \ref{thm:L-q-dim-dyn-ssm} are immediate.

Applying Theorem \ref{thm:L-q-dim-dyn-ssm} and recalling the form of $\mathbb{P}'$, we conclude that for any $y\in X$,
\[
D(\nu_y^{(k)},q) =  \min\left( \frac{\frac{k-1}{k}\int_X \log\|\Delta(x)\|_q^q \,d\mathbb{P}(x)}{(q-1)\log\lam},1\right) =1,
\]
provided $k$ is large enough. This establishes \eqref{eq:dim-skip-meas-1} and concludes the proof.
\end{proof}

We remark that the theorem provides the correct range for the possibility of having an $L^q$ density (other than perhaps the endpoint), since measures $\mu$ with an $L^q$ density satisfy $D(\mu,q)=1$; this can be seen from the inequality $(\int_I f)^q \le |I|^{q-1} \int_I f^q$ for all intervals $I$, where $f$ is the $L^q$ density of $\mu$.

As a first application, we can now conclude the proof of Theorem \ref{thm:abc-cont-BCs}.
\begin{proof}[Proof of Theorem \ref{thm:abc-cont-BCs}]
Erd\H{o}s \cite{Erdos40} and Kahane \cite{Kahane71} proved that there is a set  $\mathcal{E}'\subset (0,1)$ of zero Hausdorff dimension, such that if $\lam\in (0,1)\setminus\mathcal{E}'$, then $|\widehat{\nu_\lam}(\xi)| \le C(\lam) |\xi|^{-\delta(\lam)}$ for some $C(\lam),\delta(\lam)>0$. See also \cite{PSS00} for an exposition of the argument.

Let $\mathcal{E}_1 = \{ \lam\in (0,1/2):\lam^k\in\mathcal{E}' \text{ for some } k\}$, which still has zero Hausdorff dimension. Consider the model $\mathcal{X}_\lam$ with trivial dynamics as in the proof of Theorem \ref{thm:dim-ssm}, and recall from Definition \ref{def:ssc-BCs} and the discussion afterwards that there is another zero-dimensional set $\mathcal{E}_2$ such that $\mathcal{X}_\lam$ has exponential separation for $\lam\in (1/2,1)\setminus\mathcal{E}_2$. The measure $\mu_x^{(k)}$ for the model $\mathcal{X}_\lam$ is just $\nu_{\lam^k}$. Part (i) of the theorem then follows from Theorem \ref{thm:abs-continuity} with exceptional set $\mathcal{E}=\mathcal{E}_1\cup \mathcal{E}_2$.

The second part follows from the first, the identity $\nu_\lam = \nu_{\lam^2} * S_\lam \nu_{\lam^2}$, and the fact that the convolution of two $L^2$ functions is continuous.
\end{proof}

The method of Erd\H{o}s-Kahane has been applied to many other parametrized families of fractal measures, see \cite[Section 3]{ShmerkinSolomyak16} for some examples. Using this, one can extend Theorem \ref{thm:abc-cont-BCs} to more general families of self-similar measures. We state one such result.
\begin{thm} \label{thm:abs-cont-parametrized-ssm}
Let $u\mapsto (\lambda(u),t_1(u),\ldots,t_m(u))$ be a real-analytic map from an open domain $U\subset\R^\ell$ to $\left((-1,0)\cup (0,1)\right) \times \R^m$. Assume that for all $\omega\neq \omega'\in \{1,\ldots,m\}^\N$ there is $u\in U$ such that
\begin{equation} \label{eq:non-degeneracy}
\sum_{i=0}^\infty t_{\omega_i}(u) \lambda(u)^i \neq \sum_{i=0}^\infty t_{\omega'_i}(u) \lambda(u)^i.
\end{equation}
Given a probability vector $p=(p_1,\ldots,p_m)$, write $\Delta_u^p = \sum_{j=1}^m p_j \delta(t_j(u))$, and denote the associated self-similar measure by $\nu_u^p = *_{i=0}^\infty S_{\lam(u)^i} \Delta_u^p$.  Then there exists a set $\mathcal{E}\subset U$ of Hausdorff dimension at most $\ell-1$ such that if $u\in U\setminus\mathcal{E}$ and $\|\Delta_u^p\|_q^q < |\lam_u|^{q-1}$,  then $\nu_u^p$ is absolutely continuous with a density in $L^q$.
\end{thm}
\begin{proof}
The proof is essentially identical to that of \cite[Theorem A]{ShmerkinSolomyak16}. Let $\mathcal{X}_u^p$ be the model with trivial dynamics associated to $\nu_u^p$. It follows from \cite[Theorem 1.7]{Hochman17} and the non-degeneracy assumption \eqref{eq:non-degeneracy} that there is a set $\mathcal{E}'\subset  U$ of Hausdorff (and even packing) dimension $\le \ell-1$ such that $\mathcal{X}_u^p$ has exponential separation for all $u\in U\setminus\mathcal{E}_1$ and all $p$.

On the other hand, for each $k$ there is a set $\mathcal{E}''_k$ of zero Hausdorff dimension such that the measure $*_{i=0}^\infty S_{\lam(u)^{ki}} \Delta_u^p$ has power Fourier decay for all $u\notin U\setminus \mathcal{E}''_k$ and all $p$. The proof of this fact is contained in the proof of \cite[Theorem A]{ShmerkinSolomyak16}; in short, one uses two variants of the Erd\H{o}s-Kahane argument depending on whether or not the function $u\mapsto\lam(u)$ is constant. In light of Theorem \ref{thm:abs-continuity}, the claim follows with exceptional set $\mathcal{E}'\cup (\cup_{k=1}^\infty \mathcal{E}''_k)$.
\end{proof}

Note, however, that just as in \cite{ShmerkinSolomyak16} here we are limited to homogeneous iterated function systems, as the argument to pass from full $L^q$ dimension to $L^q$ density depends strongly on the structure of the measures as infinite Bernoulli convolutions. In \cite{SSS17} absolute continuity was obtained for a.e. parameter for some families of non-homogeneous self-similar measures, but no information on the densities was obtained.

As another application of Theorem \ref{thm:abs-continuity}, we obtain the following result on projections of planar self-similar measures:
\begin{cor} \label{cor:abs-cont-proj-ssm}
Let $\mu$ be as in Lemma \ref{lem:model-proj-hom-ssm}. Assume that the open set condition holds. Then there is a set $\mathcal{E}\subset [0,2\pi)$ of zero Hausdorff dimension (depending only $\lam,\alpha,\supp(\wt{\Delta})$), such that $P_x\mu$ is absolutely continuous with an $L^q$ density for all $q$ such that $\|\wt{\Delta}\|_q^q<\lam^{(q-1)}$.
\end{cor}
\begin{proof}
If $\alpha/\pi\in\Q$, then we can assume that $\alpha=0$ by iterating the original IFS. In this case $(P_x\mu)_{x\in S^1}$ is a family of self-similar measures satisfying the assumptions of \cite[Theorem A]{ShmerkinSolomyak16}, so the claim holds as explained in the above discussion.

If $\alpha/\pi\notin \Q$, consider the model $(S^1,\mathbf{R}_{-\alpha},\Delta,\lam)$ from Lemma \ref{lem:model-proj-hom-ssm}. The measures $\mu_x^{(k)}$ are projections of the  self-similar measure $*_{n=0}^\infty S_{\lam^n} \mathbf{R}_{k\alpha}^n \wt{\Delta}$. It follows from \cite[Proposition 3.3]{ShmerkinSolomyak16} that there exists a set $\mathcal{E}_k\subset [0,2\pi)$ of zero Hausdorff dimension, depending only on $\lam,\alpha,k$ and $\supp(\wt{\Delta})$, such that the projection $\mu_x^{(k)}$ has a power Fourier decay for all $x\in [0,2\pi)\setminus \mathcal{E}_k$. The claim then follows from Theorem \ref{thm:abs-continuity} with exceptional set $\mathcal{E}=\cup_{k\in\N} \mathcal{E}_k$.
\end{proof}

Recall that the Fourier transform of a measure $\mu$ on $\R^2$ is
 \[
\widehat{\mu}(\xi) = \int \exp(2\pi i \langle y,\xi \rangle)\,d\mu(y),
\]
and that if $v\in S^1$, then $\widehat{P_v\mu}(\xi)=\widehat{\mu}(\xi v)$. In particular, if $\mu$ has power Fourier  decay (in the sense that $|\widehat{\mu}(\xi)|=O(|\xi|^{-\delta})$ for $\xi\in\R^2\setminus\{0\}$ and some $\delta>0$), then so do all its projections $P_v\mu$.

If the planar self-similar measure $\mu$  has power Fourier decay and $\alpha/\pi$ is irrational, then the proof of the Corollary \ref{cor:abs-cont-proj-ssm} together with the above observations show that $P_x\mu$ has an $L^q$ density for \emph{all} $x\in S^1$, whenever   $\|\wt{\Delta}\|_q^q<\lam^{(q-1)}$. Although we know of no explicit example of such measure $\mu$, in parameter space power Fourier decay occurs outside of very small exceptional sets; see \cite[Theorem D]{ShmerkinSolomyak16b}.

We obtain a further corollary for convolutions of two self-similar measures, with the parameter coming in the scaling. A direct application of Theorem \ref{thm:abs-continuity} is somewhat awkward because the corresponding measures $\mu_k^{(k)}$ do not have a particularly nice structure. However, the proof of \cite[Theorem D]{ShmerkinSolomyak16}, using Theorem \ref{thm:dim-conv-ssm} to calculate the $L^q$ dimensions of self-similar measures and their convolutions, yields our final result; the verification of the details is left to the reader.
\begin{cor}
Let
\[
\eta_j=\eta_j(\Delta_j,\lam_j) = *_{n=0}^\infty S_{\lam_j^n} \Delta_j
\]
be two homogeneous self-similar measures satisfying the open set condition on the real line.

Then there is a set $\mathcal{E}\subset\R$ of zero Hausdorff dimension, such that if $t\in\R\setminus\mathcal{E}$ and $q>1$ is such that $D(\eta_1,q)+D(\eta_2,q)>1$, then the convolution $\eta_1 * S_t \eta_2$ is absolutely continuous with a density in $L^q$.
\end{cor}

%\bibliographystyle{plain}
%\bibliography{furstenberg}

\end{document}